\documentclass[12pt,a4paper]{article}  % -*- coding: utf-8 -*-
\usepackage[a4paper,margin=2.5cm]{geometry}
\usepackage[english]{babel}
\usepackage[utf8]{inputenc}
\usepackage[T1]{fontenc}
\usepackage{lmodern}
\usepackage{newtxtext}
\usepackage{amsmath}
\usepackage{amsfonts}
\usepackage{amssymb}
\usepackage{amsthm}
\usepackage{mathrsfs}
\usepackage{stmaryrd}
\usepackage{wasysym}
\usepackage{pifont}
\usepackage{url}
\usepackage{graphicx}
\usepackage[usenames,dvipsnames]{xcolor}
\usepackage{hyperref}
\usepackage{cleveref}
\usepackage{enumitem}
\setlist[itemize]{nosep}
\setlist[enumerate]{nosep}
\theoremstyle{definition}
\newtheorem{comcnt}{Anything}[section]
\newcommand\thingy{%
\refstepcounter{comcnt}\medskip\noindent\textparagraph\textbf{\thecomcnt.} }
\newtheorem{dfn}[comcnt]{Definition}
\newtheorem{prop}[comcnt]{Proposition}
\newtheorem{thm}[comcnt]{Theorem}
\newtheorem{cor}[comcnt]{Corollary}
\newtheorem{rmk}[comcnt]{Remark}
\newtheorem{warn}[comcnt]{Warning}
\newtheorem{lem}[comcnt]{Lemma}
\newtheorem{exm}[comcnt]{Example}
\newtheorem{exer}[comcnt]{Exercise}
\theoremstyle{plain}
\newtheorem{ques}[comcnt]{Question}
\mathchardef\emdash="07C\relax
\mathchardef\hyphen="02D\relax
\DeclareUnicodeCharacter{00A0}{~}
\DeclareUnicodeCharacter{A76B}{z}
\newcommand{\Om}{{\boldsymbol{\omega}}}
\newcommand{\N}{\mathbb{N}}
\renewcommand{\P}{\mathcal{P}}
\newcommand{\D}{\mathscr{D}}
\newcommand{\op}{{\operatorname{op}}}
\renewcommand{\iff}{\Leftrightarrow} % make it shorter
\newcommand{\id}{\operatorname{id}}
\newcommand{\Hom}{\operatorname{Hom}}
\newcommand{\coT}{\operatorname{coT}}
\newcommand{\class}{\operatorname{class}}
\newcommand{\sh}{\operatorname{sh}}
\newcommand{\Sh}{\mathbf{Sh}}
\newcommand{\TZero}{\mathrm{T0}}
\newcommand{\TOne}{\mathrm{T1}}
\newcommand{\TTwo}{\mathrm{T2}}
\newcommand{\TThree}{\mathrm{T3}}
\newcommand{\meetDZero}{\mathbin{\wedge_{\mathrm{T0}}}}
\newcommand{\meetDOne}{\mathbin{\wedge_{\mathrm{T1}}}}
\newcommand{\meetDTwo}{\mathbin{\wedge_{\mathrm{T2}}}}
\newcommand{\meetDThree}{\mathbin{\wedge_{\mathrm{T3}}}}
\newcommand{\meetTOne}{\mathbin{\curlywedge_{\mathrm{T1}}}}
\newcommand{\meetTTwo}{\mathbin{\curlywedge_{\mathrm{T2}}}}
\newcommand{\meetTThree}{\mathbin{\curlywedge_{\mathrm{T3}}}}
\newcommand{\redTOne}{\mathrel{\preceq_{\mathrm{T1}}}}

\newcommand{\redTThree}{\mathrel{\preceq_{\mathrm{T3}}}}
\newcommand{\Psymb}{\mathtt{P}}
\newcommand{\Qsymb}{\mathtt{Q}}
\title{An order-reversing embedding of Turing degrees into Arthur–Nimue–Merlin degrees}
\author{
  Jean \textsc{Abou Samra}\footnote{
    Eötvös Loránd University, Budapest, Hungary.
    Funded by the European Union through the HOTT grant (ERC 101170308).
  }
  \hphantom{*}
  \& David Alexander \textsc{Madore}\footnote{
    Télécom Paris (LTCI) and Institut Polytechnique de Paris, Palaiseau, France.}
}
\begin{document}
\pretolerance=8000
\tolerance=50000
\selectlanguage{english}
\maketitle

% \begin{center}
% {\tiny\noindent
% \immediate\write18{sh ./vc > vcline.tex}
% Git: \input{vcline.tex}
% \immediate\write18{echo ' (stale)' >> vcline.tex}
% \par}
% \end{center}

\begin{abstract}
The Arthur–Nimue–Merlin degrees are a generalization of the Turing
degrees introduced by \textsc{Kihara} as a tangible description of the
partially ordered set of Lawvere–Tierney topologies on the effective
topos (equivalently, subtoposes of the effective topos).  They are
defined in terms of a three-player game that introduces both angelic
and demonic non-determinism into oracle queries.  We construct an
order embedding of the Turing degrees with their order reversed into
the Arthur–Nimue–Merlin degrees, whose image we call the “co-Turing
degrees”; we then study the order relationship of these co-Turing
degrees with the (naturally embedded) Turing degrees within the
Arthur–Nimue–Merlin degrees.
\end{abstract}

\begin{narrower}
\end{narrower}

%
%
%

% Notational conventions:
% The argument to f is called m, and for a T3 oracle, F∈f(m).
% The return value of f is called u (so for a T3 oracle, u∈F).
% The argument to g is called n, and for a T3 oracle, G∈g(n).
% The return value of g is called v (so for a T3 oracle, v∈G).

% Try to speak of querying the oracle g “for” some n (rather than
% using the word “question”), and try to reserve the term “value” for
% the return value of the oracle.

\section{Introduction}

\thingy\label{intro-riddles} We begin with a riddle.  King Arthur is
desperate to know whether the Grail is in a particular French castle.
He can ask arbitrary questions to the all-knowing mage Merlin, who
promises to answer truthfully, but with a twist.  Instead of merely
answering “yes” or “no” in plain English, the mischievous Merlin
provides a Turing machine, which halts if and only if the answer is
“yes”.  Can Arthur, a mere mortal limited by the Church–Turing thesis,
find the Grail's location before the end of days?

% (This is a reference to the taunting French knights in the movie
% “Monty Pythons and the Holy Grail”, who claim to have a Grail in
% their castle.)

Indeed he can, with a simple trick: first ask Merlin whether the Grail
is in the castle, then ask whether it is \emph{not} in the castle.
Exactly one of the two resulting Turing machines halts, and he can run
them in parallel to know which one, revealing the answer.

Let us spice up the rules.  Each of Merlin's replies now consists of
\emph{two} Turing machines; for a “yes” answer, both or none halt,
whereas for a “no” answer, exactly one halts.\footnote{We are indebted
to Will \textsc{Sawin} for suggesting this particularly elegant
variant of the riddle.} It turns out that, under these rules, no
strategy can ensure that Arthur learns the answer to an arbitrary
boolean question in finite time.  A proof of this fact (which we give
in ¶\ref{answer-to-riddle}) seems to require the main technical result of this
paper, which shows that asking the question then its negation is
essentially the only useful trick in a precise sense.

\thingy\label{intro-context}
\textbf{Context.} Our object of study is an extension of Turing degrees where
the oracles can feature certain forms of non-determinism, as in this
riddle. Our motivations and ideas are grounded in the study of realizability,
where this extension naturally arises. Nevertheless, this paper is
meant to be accessible to an audience of classical computability
theorists without a background in realizability, except for
\cref{section-topos-theory} where we explain how the proof of our main
result can be understood in topos-theoretic terms.
Still, let us briefly recount how the notion arose.

The effective topos, defined by Hyland, is a model of constructive
mathematics — more precisely, an (elementary) topos\footnote{In this
paper, the bare word \emph{topos} always means \emph{elementary topos}
rather than \emph{Grothendieck topos}.} with natural
numbers object and thus a model of intuitionistic higher-order
logic — which can be viewed as the “world of computable mathematics”.
This topos contains the topos of sets, namely the “world of classical
mathematics”, as a smallest non-degenerate subtopos, and in between them
is an array of subtoposes corresponding to worlds intermediate between
computable and classical mathematics. In his seminal work introducing
the effective topos, Hyland associated to each Turing degree $\mathbf{d}$ a
subtopos of “$\mathbf{d}$-computable mathematics”, which admits a description
similar to the effective topos with plain computability replaced
by computability relative to $\mathbf{d}$. Hyland proved that this defines
an embedding of Turing degrees into subtoposes of the effective topos
\cite[theorem 17.2]{HylandEff}. However, there are many subtoposes which
do not arise in this way.

As a general fact of topos theory, the subtoposes of an elementary
topos are described by the so-called “Lawvere–Tierney topologies”
(also known by a variety of other names such as “local operators”) on the topos.
Although this characterization is already more concrete, to explicitly
describe the Lawvere–Tierney topologies on the effective topos is a
non-trivial task. A first step was undertaken by \cite{LeeVanOosten},
who introduced the subclass of \emph{basic} Lawvere–Tierney topologies
on the effective topos and described them using certain well-founded
trees called \emph{sights}. However, the main
advance was made by \cite{KiharaLT}, who proposed a considerably simpler
description which encompasses all Lawvere–Tierney topologies on the
effective topos. This description gives a new point
of view on these objects, as a generalization of Turing degrees
incorporating two distinct forms of non-determinism.

An oracle of this generalized kind is a
function\footnote{Alternatively, they can be described as partial
  multivalued functions from $\N \times \Lambda$ to $\N$, where
  $\Lambda$ is an unspecified set which may depend on the oracle.  We
  will clarify this in ¶\ref{differences-with-kihara-setup}.} $g \colon \N \to
\P(\P(\N))$, in which each powerset stands for one kind of
nondeterminism. When a program queries the oracle
$g$ at $n$, an external agent chooses a set $G \in g(n)$
in a benevolent way (angelic non-determinism), then another external
agent chooses a natural number $v \in G$ in an adversarial way
(demonic non-determinism), and the program receives the value $v$ (and
only this).

The mathematically precise definition of “benevolent” and
“adversarial” uses a three-player imperfect information game. A
program computing $f \colon \N \to \N$ is represented by a player,
Arthur\footnote{The names are presumably borrowed
from Arthur–Merlin interactive proof protocols in complexity theory,
which bear a slight resemblance to this game.},
who is restricted to computable strategies. Arthur is allied with
Nimue, the benevolent external agent, and both play against Merlin,
the adversarial external agent.  The initial configuration is
given by an input $m$. At each turn, either Arthur declares a final
value of his computation, winning if and only if this final value
is the correct $f(m)$, or he queries the oracle at some $n$, in which
case Nimue chooses $G \in g(n)$, then Merlin chooses $v \in G$ and
Arthur receives $v$ (crucially, without seeing $G$).
By definition, $f$ is computable with oracle
$g$ when Arthur and Nimue have a winning strategy against Merlin
in this game. This can be extended to $f \colon \N \to \P(\P(\N))$ in a
straightforward way, which defines the reducibility relation between
functions $f, g \colon \N \to \P(\P(\N))$, and the usual quotient which
identifies mutually reducible functions yields the partially ordered
set of Arthur–Nimue–Merlin degrees. Precise definitions are spelled
out in our \cref{section-arthur-nimue-merlin-games} or
in \cite{KiharaLT}; and in \cref{section-lattice-structure} we will
describe the lattice structure on these generalized degrees.

This description by a game also sheds light on Hyland's embedding
of Turing degrees into Arthur–Nimue–Merlin degrees: it is
obtained by identifying a Turing oracle $g \colon \N \to \N$ with the
generalized oracle $\hat{g} \colon \N \to \P(\P(\N)), n \mapsto
\{\{g(n)\}\}$. When the game is played with $\hat{g}$, neither Nimue nor
Merlin have any choice in their moves and Arthur is in essence using
the ordinary Turing oracle $g$.

To lighten notations, we henceforth refer to Turing degrees as
\textbf{T0 degrees} and to the most general Arthur–Nimue–Merlin degrees as
\textbf{T3 degrees}\footnote{In \cite{KiharaOracles}, the term “extended Turing–Weihrauch
degrees” is used. However, the focus is more on the analogous objects
for the Kleene–Vesley topos, and we prefer to eschew the
association with Weihrauch's name since we never deal with
call-once oracles. Hence we choose “Arthur–Nimue–Merlin degrees”, to
be clarified with “T1”, “T2” or “T3”  to distinguish the T3 degrees
from the intermediate generalizations introduced in the following
paragraphs.  (We might also suggest “W1” through “W3” for Weihrauch
degrees, see ¶\ref{passing-remark-weihrauch}.)
We hope that this suggestion will not cause the situation described
in \cite{xkcd927}.}  As the terminology suggests, this is because
several other particular types of oracles are worthy of attention,
intermediate between the “T0” and “T3” levels.

We call \textbf{T1 degrees} the degrees of partial functions
$g \colon \N \dasharrow \N$, converted to general oracles
by setting $\hat{g}(n) := \{\{g(n)\}\}$ when $g(n)$ is defined
and $\hat{g}(n) := \varnothing$ otherwise. In the game played with
$\hat{g}$, Arthur may only query the oracle on inputs where $g$
is defined, otherwise Nimue has no available move and hence
Arthur and Nimue lose immediately. Conversely, the game defining
reducibility from $f$ to $g$ can only start on initial configurations
where the input $m$ is such that $f(m)$ is defined. Summarizing
the situation, $f \colon \N \dasharrow \N$ is T1-reducible to
$g \colon \N \dasharrow \N$ when $f$ has an extension which is
partial computable with oracle $g$, with the provision that
oracle queries outside the domain of $g$ make the computation fail.
This reducibility is one of many which can be devised for partial
functions (see \cite{SassoSurvey}, \cite{CooperPartial},
\cite[prop. II.3.25]{Odifreddi1} and \cite{MO112617} for a discussion).
Furthermore, \cite[corollary 2.13]{FaberVanOosten} proved that the
subtoposes of the effective topos which arise from T1 degrees are
exactly those that are themselves realizability toposes.
A systematic investigation of the structure of the T1 degrees was
initiated (under the name “subTuring degrees”) in \cite{KiharaNg1}.

In a similar fashion, we call \textbf{T2 degrees} the degrees of
partial functions $g \colon \N \dasharrow \P(\N)$ (which we may think
of as “partial multivalued” functions), converted to general oracles
by setting $\hat{g}(n) := \{g(n)\}$ when $g(n)$ is defined and
$\hat{g}(n) := \varnothing$ otherwise. In this case, Nimue has
no choice but Merlin's strategy can be non-trivial. When Arthur
queries the oracle on $n$, the set $g(n)$ must be defined (otherwise
Arthur loses), and Merlin chooses any value in $g(n)$ to return to
Arthur.  So the T2 degrees correspond to a fairly natural notion of
Turing reduction of non-deterministic (i.e., adversarially
non-deterministic) functions.  Let us give just one example to
illustrate this.

It is well-known that there exist T0 (Turing) degrees strictly between
computable functions and the halting problem, but none seems to arise
naturally. On the other hand, there is a natural intermediate degree in the T2 degrees.
Define $g(e) := \{0\}$ if the $e$-th program with no input and boolean
output terminates and returns false, $g(e) := \{1\}$ if it terminates
and returns true, and $g(e) := \{0, 1\}$ if it does not terminate.
This corresponds informally to having an oracle that will give you the
return value of a boolean-valued program that terminates, but may
return either boolean for one that does not terminate.
The T2 degree of $g$ is called $\mathtt{LLPO}$ in \cite[example 2.5]{KiharaLT}
due to its connection with the lesser limited principle of omniscience
in constructive mathematics. The T0 degrees above this T2 degree
are exactly the “PA degrees”, a classical object of
study in computability (\cite[definition 15.3.11]{CooperBook}
and \cite[theorem 10.3.3]{Soare}).  See
also ¶\ref{digression-separating-degree} below.

To summarize, we move from T0 (Turing) to T1 degrees by allowing
partial functions, from T1 to T2 degrees by introducing adversarial
nondeterminism (personified by “Merlin”), and from T2 to the most
general T3 degrees by introducing benevolent nondeterminism
(personified by “Nimue”).  Since these intermediate levels can also be
elegantly characterized at the level of Lawvere–Tierney topologies
(see the second paragraph
of \cref{statement-oracle-topology-correspondence}), we believe the
T0–T3 terminology is appropriate.

Another important subclass of the T3 degrees,
incomparable with T2 degrees, is given by the \textbf{basic}
T3 degrees first studied in \cite{LeeVanOosten}, which correspond
to the functions $\N \to \P(\P(\N))$ which are constant.
In other words, Arthur does not even need to explicitly ask a question
to the oracle (but Nimue can “read questions in his mind”), and
Nimue chooses a set of possible moves for Merlin in a fixed
set of possible move sets.  So now the game effectively becomes one
where Nimue attempts to communicate information to Arthur but has to
go through their adversary Merlin.

The phenomenon that a well-studied class
of T0 (Turing) degrees becomes simply a cone in the generalized
degrees, which we have illustrated
concerning the T2 degrees is repeated in the basic T3 degrees.
Indeed, \textsc{Pitts}, in his doctoral thesis
(\cite[example 5.8]{PittsThesis}) considered the T3 oracle
which consists of the cofinite subsets of $\N$.
Informally, this means that Nimue chooses a natural number, then
Merlin can replace it by any larger number before Arthur can see it.  And
in \cite[theorem 2.2]{vanOostenPittsOperator} it is shown that the
functions $\N \to \N$ computable from this oracle exactly coincide
with the hyperarithmetical (i.e., $\Delta_1^1$) ones.

A deep investigation of basic degrees, focusing on their combinatorial
properties (and their unexpected relation with the Rudin–Keisler
order), was even more recently undertaken in \cite{KiharaNg2}.

\thingy \textbf{Contribution.} We make a further advance in the
understanding of Arthur–Nimue–Merlin degrees by exhibiting an
embedding of the Turing (T0) degrees with the opposite of their
standard order into the basic T3 degrees. Namely, we define for
each Turing degree $\mathbf{d}$ a basic\footnote{Thus, our “co-Turing
degree” construction provides an element of answer to \cite[§6,
  question 1]{KiharaLT} which asked for other “interesting” basic T3
degrees.} T3 degree $\coT(\mathbf{d})$, called its
\emph{co-Turing degree}, so that the map $\coT$ is an order-reversing embedding,
i.e., $\mathbf{d}_0 \leq \mathbf{d}_1 \iff \coT(\mathbf{d}_0) \geq \coT(\mathbf{d}_1)$ for all Turing degrees
$\mathbf{d}_0, \mathbf{d}_1$. Furthermore, we show that a non-zero Turing degree is
never comparable with the co-Turing degree of a non-zero Turing
degree.

The proof of the latter fact relies on a property of meets in T1 degrees,
namely that there exists no minimal pair of T0 degrees in the T1 degrees.
As a preliminary, we give a short self-contained proof of this
fact, already proved in \cite{KiharaNg1} as a consequence of more
sophisticated results; this proof is credited to \cite{MO507736}.
We also seize the opportunity to describe the meet in T2 and T3 degrees,
although it is not needed for the main result, and make a few related
remarks.

\thingy \textbf{Outline.} \Cref{section-arthur-nimue-merlin-games}
precisely describes the Arthur–Nimue–Merlin game and recalls
the basic properties of our classes of degrees: this is essentially a
summary of some parts of \cite{KiharaLT} but with a few inessential
differences in notation and terminology.
In \cref{section-lattice-structure}, we describe the binary join and
meet of degrees and prove some results relating to their lattice
structure.  In \cref{section-decoding-and-coding-degrees}
we define the \emph{decoding degree} and \emph{coding
degree} associated with a promise problem and we prove
our main technical result, \cref{main-strategy-result},
characterizing when a coding degree is the omniscient degree. In
\cref{section-co-turing-degrees}, we define the co-Turing degrees
as a particular case of coding degrees and prove our main result,
\cref{reverse-embedding-theorem}, according to which the mapping
from Turing degrees to co-Turing degrees is an order-reversing
embedding. In \cref{section-comparison-of-codegree-with-degree}, we
prove that there is no non-trivial comparison of co-Turing degrees
with Turing degrees. \Cref{section-topos-theory} is devoted to
reinterpreting the constructions in topos-theoretic terms and is the
only section which (after an initial recapitulation on Lawvere–Tierney
topologies on the effective topos, which is self-contained)
requires familiarity with the theory of toposes
and (categorical) realizability. Finally,
in \cref{section-questions}, we propose various questions for further
research.

\thingy\label{general-notations}\textbf{General notations and conventions.}
We work in classical logic unless otherwise specified.  However, we
still adopt the convention from constructive mathematics to call
\textbf{inhabited} a set which has an element (this is classically
equivalent to being non-empty). We write $f \colon A \to B$ for a total
function and $f \colon A \dasharrow B$ for a partial function. In the
latter case, we write $f(x){\downarrow}$ to mean that $f$ is defined
at $x$, and $f(x){\downarrow} = y$ to mean that furthermore the value
equals $y$. The powerset of $X$ is denoted by $\P(X)$.  We deal with
“multivalued” functions by treating them explicitly as functions with
value in a powerset (e.g., we consider functions $\N \to \P(\N)$
where some authors would prefer to write $\N \rightrightarrows \N$).

We fix a computable Gödel coding of tuples by natural numbers, denoted
$(m_0,\ldots,m_{k-1}) \mapsto \langle m_0,\ldots,m_{k-1}\rangle$.
The $e$-th partial computable function is $\varphi_e \colon \N \dasharrow \N$
and the $e$-th partial computable function with oracle
$f \colon \N \dasharrow \N$ is $\varphi_e^f$. In the definition of
$\varphi_e^f$, if the execution of the $e$-th program on $n$
queries the oracle $f$ where it is undefined, then $\varphi_e^f(n)$
is undefined. The numbering of programs is independent of the oracles
(e.g., there exists $e$ such that $\varphi_e^f = f$ for all $f$).
We freely identify a subset $A \subseteq \N$ with its characteristic
function $1_A$.

We denote by $U \sqcup V$ the labeled disjoint union of two sets
$U, V \subseteq \N$, defined as $U \sqcup V := \{\langle 0, u \rangle
\mid u \in U\} \cup \{\langle 1, v \rangle \mid v\in V\}$.

\bigskip

\thingy\label{specific-notations}\textbf{Specific notations and
  definitions} introduced in this paper (numbers refer to the
paragraph where they are defined).
\begin{itemize}
\item T0–T3 oracles, “basic” oracles: \ref{t0-to-t3-oracles}
\item “reduction game” ($\redTThree$): \ref{definition-reduction-game}
\item T0–T3 degrees ($\mathscr{D}_{\mathrm{T0}}$ through $\mathscr{D}_{\mathrm{T3}}$, $[f]_\TThree$): \ref{t0-to-t3-degrees}
\item $\Om$ (“omniscient” degree): \ref{bottom-top-and-omniscient-degrees}
\item $[f\oplus g] = [f] \vee [g]$ (binary join of degrees): \ref{join-of-degrees}
\item $[f\meetTOne g]_{\mathrm{T1}} = [f]_{\mathrm{T1}} \meetDOne [g]_{\mathrm{T1}}$ (binary meet of T1 degrees): \ref{meet-of-t1-degrees}
\item $[f\meetTThree g]_{\mathrm{T3}} = [f]_{\mathrm{T3}} \meetDThree [g]_{\mathrm{T3}}$ (binary meet of T2/T3 degrees): \ref{meet-of-t2-degrees}
\item $\mathbf{d}_{P,Q}$ (“decoding” oracle/degree) and $\mathbf{c}_{P,Q}$ (“coding” oracle/degree): \ref{definition-coding-decoding-degrees}
\item “promise problem”: \ref{promise-problem}
\item $\mathrel{\preceq_{\mathrm{m}}}$ (many-to-one reduction of promise problems): \ref{definition-many-to-one-reduction}
\item $H_0^A, H_1^A$: \ref{definition-codegrees}
\item $\tilde{\mathbf{c}}_A$ (“co-Turing” degree) and $\tilde{\mathbf{d}}_A$: \ref{definition-codegrees}
\item $\coT([A]_{\mathrm{T0}}) = \tilde{\mathbf{c}}_A$: \ref{reverse-embedding-theorem}
\item $U\Rrightarrow V$: \ref{definition-lawvere-tierney-topologies}
\item “Lawvere–Tierney topology on the effective topos”: \ref{definition-lawvere-tierney-topologies}
\item $g \mapsto J_g$ and $J \mapsto J^\natural$: \ref{definition-oracle-topology-correspondence}
\item $\varpi(G)$ (“$G$ is inhabited”): \ref{preparation-for-t3-oracle-to-internal-topology}
\item $X_{P,Q}$: \ref{definition-x-p-q-object}
\item $\coT(j)$ (for $j$ a Lawvere–Tierney topology): \ref{codegree-construction-for-topologies}
\end{itemize}

\section{Arthur–Nimue–Merlin games and degrees}\label{section-arthur-nimue-merlin-games}

The purpose of this section is to provide the reader not familiar
with \cite{KiharaLT} with a definition of the Arthur–Nimue–Merlin
degrees and a summary of their essential properties.  We also properly
introduce the terminology “T0” through “T3”, which is our own.
Readers who \emph{are} familiar with \textit{op. cit.} are advised
to read ¶\ref{differences-with-kihara-setup}.

\bigskip

We start by introducing the various kinds of oracles.

\begin{dfn}\label{t0-to-t3-oracles}
We define formally:
\begin{itemize}
\item A \textbf{T0 oracle} (or \textbf{ordinary Turing oracle}) to be
  a function $\mathbb{N} \to \mathbb{N}$.
\item A \textbf{T1 oracle} to be a partial function $\mathbb{N}
  \dasharrow \mathbb{N}$.
\item A \textbf{T2 oracle} to be a partial function $\mathbb{N}
  \dasharrow \mathcal{P}(\mathbb{N})$.
\item A \textbf{T3 oracle} to be a function $\mathbb{N}
  \to \mathcal{P}(\mathcal{P}(\mathbb{N}))$.
\item A \textbf{basic (T3) oracle} to be an element of
  $\mathcal{P}(\mathcal{P}(\mathbb{N}))$.
\end{itemize}
We agree on the following identifications:
\begin{itemize}
\item A T0 oracle $g\colon \mathbb{N} \to \mathbb{N}$ is identified
  with the T1 oracle obtained by considering the total function $g$ as
  a partial function $\mathbb{N} \dasharrow \mathbb{N}$ that happens
  to be defined everywhere.
\item A T1 oracle $g\colon \mathbb{N} \dasharrow \mathbb{N}$ is
  identified with the T2 oracle given by the partial function
  $\mathbb{N} \dasharrow \mathcal{P}(\mathbb{N})$ that takes $n$ to
  $\{g(n)\}$ if $g(n)$ is defined (and undefined otherwise).
\item A T2 oracle $g\colon \mathbb{N} \dasharrow
  \mathcal{P}(\mathbb{N})$ is identified with the T3 oracle given by
  the total function $\mathbb{N} \to
  \mathcal{P}(\mathcal{P}(\mathbb{N}))$ that takes $n$ to $\{g(n)\}$
  if $g(n)$ is defined, and $\varnothing$ otherwise.
\item A basic oracle $\mathscr{G} \in \mathcal{P}(\mathcal{P}(\mathbb{N}))$ is
  identified with the T3 oracle given by the constant function
  $\mathbb{N} \to \mathcal{P}(\mathcal{P}(\mathbb{N}))$ with
  value $\mathscr{G}$.
\end{itemize}
\end{dfn}

\thingy Intuitively, the oracles should be understood in the following
way: querying a T0 oracle $\mathbb{N} \to \mathbb{N}$ corresponds to
asking the question “what is the value $g(n)$?”.  The same holds for a
T1 oracle $\mathbb{N} \dasharrow \mathbb{N}$ except that now we add “I
promise that this value is defined!”.  For a T2 oracle $\mathbb{N}
\dasharrow \mathcal{P}(\mathbb{N})$, the query is to be understood as
“give me any one of the values in $g(n)$” (this is the introduction of
adversarial nondeterminism).  And for a T3 oracle $\mathbb{N} \to
\mathcal{P}(\mathcal{P}(\mathbb{N}))$, the query means “my ally will
choose the best possible $G \in g(n)$ and now please give me any
one of the values in $G$” (so we now have both benevolent and
adversarial forms of nondeterminism).

\bigskip

To make these intuitions precise, we introduce the Arthur–Nimue–Merlin game.
The following definitions are slightly reformulated from \cite[definitions 2.9 \& 2.10]{KiharaLT}.

\begin{dfn}\label{definition-reduction-game}
Let $f,g\colon\mathbb{N} \to
\mathcal{P}(\mathcal{P}(\mathbb{N}))$ be T3 oracles. We define the
\textbf{reduction game} of $f$ to $g$ as follows:
\begin{itemize}
\item The game has three players named Arthur, Nimue and Merlin.
  Arthur and Nimue are allied against Merlin (so the two possible
  outcomes are that Arthur and Nimue win and Merlin loses,
  or Arthur and Nimue lose and Merlin wins).
\item The game starts with an initial configuration or
  “challenge”\footnote{One may imagine that it is a challenge chosen
    by Merlin for Arthur and Nimue:
    see ¶\ref{differences-with-kihara-setup}.}, which consists of a
  natural number $m$ and a set $F \in f(m)$. Nimue and Merlin can see
  both $m$ and $F$, while Arthur can only see $m$.  Accordingly, we
  will sometimes refer to $m$ as the “public” part and $F$ as the “secret”
  part of the initial configuration.
\item The moves are played in the following order: Arthur, then Nimue,
  then Merlin play in turn, repeatedly, until Arthur, during his turn,
  plays a special move that ends the game (or until a player cannot play
  and loses).
\item When it is his turn to play, Arthur has two kinds
  of moves available: either declare a natural number of his choice
  as the final value of the game, which ends the game,
  or query the oracle at a natural number of his choice.
\item If Arthur declares $u \in \N$ as the final value of the game,
  then Nimue and Arthur win if $u \in F$ (where $F \subseteq \N$ is
  the secret part of the initial configuration), otherwise Merlin wins.
\item If Arthur queries the oracle at some $n \in \N$,
  his choice of $n$ must be such that $g(n)$ is inhabited,
  otherwise Nimue and Arthur lose immediately. If $g(n)$ is
  indeed inhabited, then the game continues and Nimue plays
  next. Both Nimue and Merlin can see the choice of $n$.
\item When it is her turn after Arthur queried the oracle
  on $n$ (in which case $g(n)$ is inhabited), Nimue chooses a set
  $G \in g(n)$. Her choice is visible to Merlin but not to Arthur.
  The game continues with Merlin's turn.
\item If $G$ is empty, then Merlin loses immediately. If $G$
  is inhabited, Merlin chooses a reply $v \in G$, which both
  Arthur and Nimue can see, and the game continues with
  Arthur's turn.
\item If the game never ends, then Merlin wins by default.
\end{itemize}

A \textbf{Nimue strategy} for this game is a function taking an
initial configuration and a finite sequence of moves by Merlin and
Arthur to an element of $\mathcal{P}(\mathbb{N})$ to be played by Nimue;
an \textbf{Arthur strategy} is a \emph{partial computable} function
taking a natural number (representing the public part, $m$, of the initial
configuration) and a finite sequence of natural numbers (representing
moves by Merlin) to a move to be played by Arthur (coded, e.g., as
either $\langle 0, n \rangle$ to query the oracle for $n$,
or $\langle 1, u \rangle$ to declare $u$ as the final value).
If the move returned by the strategy is undefined or not of
the appropriate form, then Arthur and Nimue will lose the game.
An \textbf{Arthur+Nimue strategy} is a pair consisting of an Arthur
strategy and a Nimue one.

An Arthur+Nimue strategy is \textbf{winning} when Arthur and Nimue win
the game when they play following that strategy, no matter what the
initial configuration and Merlin's moves are\footnote{We do not define
the notion of winning Merlin strategy because Merlin cannot have one
except in trivial cases: Arthur might win “by accident” by
declaring a good value on his very first move.}.

We say that $f$ is \textbf{reducible} to $g$ (or T3-reducible to $g$),
and write $f \preceq g$ (or $f \mathrel{\preceq_{\mathrm{T3}}} g$)
when there exists a winning Arthur+Nimue strategy in the reduction game
of $f$ to $g$. We also write $f \equiv g$ (or
$f \mathrel{\equiv_{\mathrm{T3}}} g$) when $f \preceq g$ and
$g \preceq f$.
\end{dfn}

\thingy\label{differences-with-kihara-setup} The setup we have defined
differs in two inessential ways from that used in \cite{KiharaLT},
which we now record.

The first is that we define our general (T3) oracles as functions
$g\colon \mathbb{N} \to \mathcal{P}(\mathcal{P}(\mathbb{N}))$ rather
than $\check g\colon \mathbb{N}\times\Lambda \dasharrow
\mathcal{P}(\mathbb{N})$ (which in \textit{op. cit.} is called a
“bilayer” function), where $\Lambda$ is an arbitrary set which can
vary from function to function.  In the latter presentation, $\Lambda$
is the set of moves available to Nimue, and the connection between the
two descriptions is simply that the function $g\colon
\mathbb{N}\to\mathcal{P}(\mathcal{P}(\mathbb{N}))$ takes $n$ to the
set $\{\check g(n,\lambda) \mid \lambda\in\Lambda,\; \check
g(n,\lambda){\downarrow}\}$ of values of the function $\check g$ at
values the form $(n,\lambda)$ where it is actually defined.
(Conversely, given $g\colon \mathbb{N} \to
\mathcal{P}(\mathcal{P}(\mathbb{N}))$ we may take $\Lambda$ to be the
union of all the $g(n)$, and we define $\check g(n,\lambda)$ to be
$\lambda$ if $\lambda \in g(n)$ or undefined otherwise.)  We find
the presentation as functions $\mathbb{N} \to
\mathcal{P}(\mathcal{P}(\mathbb{N}))$ more convenient because it is
more symmetric and avoids having to deal with “unspecified”
sets $\Lambda$.

The second difference is that we consider the data $(m,F)$ as an
“initial configuration” rather than an opening move by Merlin.  Of
course this is immaterial since we are demanding that an Arthur+Nimue
strategy apply to \emph{all} initial configurations, so we might as
well assume the initial configuration has been chosen by Merlin.  But
this convention will be slightly more convenient when we need to speak
of, say “the sequence of moves by Merlin” or “a strategy for such an
initial configuration”.

\begin{prop}
T3 reducibility $\mathrel{\preceq_{\mathrm{T3}}}$ is a preorder (i.e.,
reflexive and transitive) relation between functions $\mathbb{N} \to
\mathcal{P}(\mathcal{P}(\mathbb{N}))$.  In particular, T3 equivalence
$\mathrel{\equiv_{\mathrm{T3}}}$ is an equivalence relation.
\end{prop}
\begin{proof}[Reference]\cite[prop. 2.13]{KiharaLT}
\end{proof}

\begin{dfn}\label{t0-to-t3-degrees}
The partially ordered set of \textbf{T3 degrees} is the quotient of the
set of T3 oracles (i.e., functions $\N \to \P(\P(\N))$) by the equivalence
relation $\mathrel{\equiv_{\mathrm{T3}}}$, equipped with the partial order
induced by $\mathrel{\preceq_{\mathrm{T3}}}$ (also written
$\mathrel{\preceq_{\mathrm{T3}}}$). We write $[g]_{\mathrm{T3}}$ for
the T3 degree of $g$ and $\mathscr{D}_{\mathrm{T3}}$ for the set of
all T3 degrees.

We define the \textbf{T0}, resp. \textbf{T1}, resp. \textbf{T2},
resp. \textbf{basic} T3 degrees to be the T3 degrees defined by a T0,
resp. T1, resp. T2, resp. basic T3 oracle (all seen as particular T3
oracles as explained in \ref{t0-to-t3-oracles}).  We define
$\mathrel{\preceq_{\mathrm{T0}}}$ and
$\mathrel{\equiv_{\mathrm{T0}}}$ to be the restriction to the
T0 oracles of $\mathrel{\preceq_{\mathrm{T3}}}$ and
$\mathrel{\equiv_{\mathrm{T3}}}$ respectively, and denote
$[g]_{\mathrm{T0}}$ for the T0 degree
of $g$ and $\mathscr{D}_{\mathrm{T0}}$ for the set of all T0 degrees,
and analogously for the other levels T1 and T2.
\end{dfn}

\thingy\label{passing-remark-weihrauch} We emphasize that our T0–T3
oracles are “Turing” kind oracles, meaning that they can be used
\emph{in unrestricted fashion} (there is no limit to how much Arthur
can query the oracle, as long as the computation eventually
terminates).  In \cite{KiharaLT} and \cite{KiharaOracles}, these
“Turing” kind oracles are studied along with “Weihrauch” kind oracles,
that is, oracles that are to be called \emph{exactly once} (or
\emph{at most once} for the “pointed Weihrauch” kind): formally, this
means that the reduction is defined by a variant of the game in which
the number of moves is limited to a single oracle call by Arthur.
(But the rest of the definitions are unchanged, so we might call them
“W0” through “W3” oracles; the “W3” ones have been studied, in a more
general context, as “extended Weihrauch predicates” in
\cite[definition 3.7]{BauerReducibility}.)  But these “Weihrauch” kind
oracles will not concern us in this paper, and the only reason we
mention them is to dispel possible confusion and explain how the
definitions we use connect with other ones found in the literature.

Another possible source of confusion which will also not concern us is
that, besides the 0–3 level, and the Turing-or-Weihrauch kind of
reduction, there is a third orthogonal dimension of oracleness.  Here
we are concerned with computations on $\N$ even if we generalize them
to allow various forms of nondeterminism; but it is possible to do
something analogous for computations on $\N^\N$, or, for the truly
adventurous, a “partial combinatory algebra” or even a relative one.
This is explained in great generality in \cite{KiharaOracles}.  Again,
the only reason we mention these concepts is to hopefully help the
reader navigate these perhaps confusingly similar notions\footnote{For
example, the most usual form of Weihrauch reducibility is what we
would call the study of “W2-level oracles on the computable part of
Kleene's second algebra $\N^\N$” to specify the various dimensions.},
but in the present paper, only computations on $\N$ will appear.

These reassurances having been made, it is safe to forget the previous
two paragraphs.

\begin{dfn}\label{bottom-top-and-omniscient-degrees}
The \textbf{bottom} degree, denoted $\mathbf{0}$ is the T0 degree of the constantly zero
function, which is also the basic T3 degree defined by $\{\{0\}\}$ (in
other words, Merlin's only move in response to an arbitrary move by
Arthur is to play $0$).  As a Turing degree, this is the usual degree
of all computable functions.

The \textbf{top} degree, denoted $\top$, is the T2 degree of the
constant function $\mathbb{N} \dasharrow \mathcal{P}(\mathbb{N})$ with value
$\varnothing$, which is also the basic T3 degree defined by
$\{\varnothing\}$ (in words, Merlin never has any move available).

The \textbf{omniscient}\footnote{This terminology is ours (in
\cite{KiharaLT} this degree is denoted $\mathtt{Error}_{1/2}$), but seems
self-explanatory.  However, it may be worth pointing out that the
game-theoretic description, and specifically the “Nimue” character,
offer a satisfactory answer to what “omniscience” means when we say we
can ask questions to an “omniscient” oracle and still want to avoid
self-referential questions like “will you answer ‘no’ to this
question?” — in fact, Arthur doesn't even need to phrase questions to
the oracle since it is a basic oracle, the questions are implied by
the common strategy he shares with Nimue.} degree, denoted $\Om$, is
the basic T3 degree defined by $\{\{n\} \mid n \in \N\}$.
\end{dfn}

The reader unfamiliar with the theory might wish to check their
understanding by working through the following, which serves as a
warmup for our later “decoding” oracles:

\begin{exer}\label{exercise-reduction-to-a-bit}
The basic T3 oracle $\{\{n\} \mid n \in \N\}$ is T3-equivalent to basic T3
oracle $\{\{0\}, \{1\}\}$ (so either one defines the degree $\Om$).
\end{exer}
\begin{proof}[Solution]
The reduction of $\{\{0\}, \{1\}\}$ to $\{\{n\} \mid n \in \N\}$ is
obvious: Nimue just plays the set ($\{0\}$ or $\{1\}$) which is part
of the initial configuration, forcing Merlin to communicate its unique
element to Arthur, and Arthur declares it as final value.

For the reduction of $\{\{n\} \mid n \in \N\}$ to $\{\{0\}, \{1\}\}$, we
essentially need to argue that if Nimue can communicate one bit to
Arthur, she can communicate an arbitrary natural number.  For example,
a winning Arthur+Nimue strategy might be the one where, if the secret
part of the initial configuration is $\{n\}$, then on the $n$-th
move, Nimue plays $\{1\}$, and on all other moves $\{0\}$, while
Arthur's strategy is to repeatedly query the oracle until Merlin
returns the bit $1$, and, if this happens on the $n$-th move, declare
the final value as being $n$.
\end{proof}

\thingy \label{reduction-from-om} It is straightforward to check that $\mathbf{0}$ and $\top$ are the
smallest and greatest T3 degree respectively, that $\mathbf{0} < \Om < \top$,
and moreover that $\Om$ is the second-greatest degree, i.e., the greatest
degree less than $\top$. Indeed, for $g \colon \N \to \P(\P(\N))$,
we have the following alternative:

\begin{itemize}
\item There exists $n$ such that $\varnothing \in g(n)$. Then
$[g]_{\mathrm{T3}} = \top$ because for every oracle $f$, in the
reduction game from $f$ to $g$, Arthur and Nimue can win as follows: Arthur
queries the oracle at $n$, and Nimue plays $\varnothing$,
so that Merlin has no legal moves and loses immediately.
\item Otherwise, $g$ reduces to $\Om$, because in the corresponding
reduction game, Arthur and Nimue have the following strategy:
given the initial configuration consisting of $n$ and $G \in g(n)$,
Nimue picks $v \in G$, which must exist, and plays $\{v\}$ so that
Merlin is forced to play $v$; Arthur queries the oracle once and
returns the resulting value directly.
\end{itemize}

Because the specific case of the reduction game of $\Om$
to $g$ comes up frequently, let us explicitate this particular case:
in this game, Nimue is given a target natural number, or without loss
of generality (per \cref{exercise-reduction-to-a-bit}) just a bit,
and Arthur's goal is, without seeing the target, but with Nimue's
help, to declare a value equal to the given target. The
fact that there is obviously no strategy to do this without
Nimue's help proves that $\Om$ is not a T2 degree.

\thingy \label{simplification-of-reducibility}
The T0 degrees are precisely the Turing degrees with their usual
partial order.  Indeed, it is straightforward to check that for $f,g$
two T0 oracles, an Arthur strategy in the reduction game of $f$ to $g$
is tantamount to an oracle machine program which computes $f$ from $g$
(and a Nimue strategy is trivial in this situation as Nimue's only
legal move is ever to play the only element of $\{\{g(n)\}\}$ after
Arthur has played $n$).

Similarly, for partial functions $f, g \colon \N \dasharrow \N$, we
have $f \redTOne g$ if and only if there exists $e$ such that
$\varphi_e^g$ extends $f$, i.e., such that
$\varphi_e^g(m){\downarrow}=f(m)$ if $f(m){\downarrow}$.

\begin{warn}\label{warning-meaning-of-computable-for-t1}
We stress, in particular, that a partial function $\N \dasharrow \N$
with T1 degree $\mathbf{0}$ is one which has a partial computable
\emph{extension}\footnote{One might tentatively call “subcomputable”
  such a function, but we have refrained from using this terminology.}
but not necessarily a partial computable function itself.

For example, \emph{any} partial function $\mathbb{N} \dasharrow
\mathbb{N}$ which equals $0$ wherever it is defined, has T1 degree
$\mathbf{0}$, but of course only countably many of such functions are
among the $\varphi_e$.
\end{warn}

\thingy\label{when-are-oracles-computable} As we have just said, for a
T1 oracle $f \colon \mathbb{N} \dasharrow \mathbb{N}$ to satisfy
$[f]_{\mathrm{T1}} = \mathbf{0}$ means that it has a (partial) computable
extension.  Let us also briefly mention what it means to be of
degree $\mathbf{0}$ for the other levels.

For a T2 oracle $f \colon \mathbb{N} \dasharrow
\mathcal{P}(\mathbb{N})$, the statement $[f]_{\mathrm{T2}} =
\mathbf{0}$ means that there is a partial computable function
$\varphi_e$ such that whenever $f(m){\downarrow}$ we have
$\varphi_e(m){\downarrow} \in f(m)$ (in other words, we can computably
pick an element of every $f(m)$ which is defined).

For a T3 oracle $f \colon \mathbb{N} \to
\mathcal{P}(\mathcal{P}(\mathbb{N}))$, the statement
$[f]_{\mathrm{T3}} = \mathbf{0}$ means that there is a partial
computable function $\varphi_e$ such that whenever $f(m)$ is
inhabited, we have $\varphi_e(m){\downarrow} \in \bigcap_{F\in f(m)}
F$.  (Indeed, in the reduction game of $f$ to $0$, the initial configuration
consists of an $m$ which is shown to Arthur and an $F \in f(m)$ which is
not shown to Arthur, and as Nimue and Merlin are essentially not
playing here, any strategy of Arthur will behave in the same way
whatever choice $F \in f(m)$ was
made, so in the end it must produce an element of $\bigcap_{F\in f(m)}
F$.)  Note that $\varphi_e(m)$ may fail to be defined
when $f(m) = \varnothing$.

In particular, a basic oracle $\mathscr{F} \in
\mathcal{P}(\mathcal{P}(\mathbb{N}))$ is of degree $\mathbf{0}$ iff
$\bigcap_{F\in \mathscr{F}} F$ is inhabited
(this is \cite[prop. 3.1(iii)]{LeeVanOosten}).

\section{Lattice structure on the degrees}\label{section-lattice-structure}

In this section, we collect various results, some of which should be
considered as part of the folklore, pertaining to the lattice structure
of the T1–T3 degrees.   We state and prove slightly more
than is strictly required for our main results, hoping that it will be
convenient for further study of Arthur–Nimue–Merlin degrees to have a
description of their lattice structure.

\medskip

The binary join (i.e., sup, least upper bound) operation is
unproblematic: it is represented by the standard construction for the
join of two Turing degrees, which remains applicable at all other
levels: intuitively, having access to the join of two oracles means
that Arthur can freely query either one.  This is made precise as
follows:

\begin{prop}\label{join-of-degrees}
The partially ordered sets $\D_{\mathrm{T0}}, \D_{\mathrm{T1}}, \D_{\mathrm{T2}}, \D_{\mathrm{T3}}$ have all binary
joins (which we denote by $\vee$), and the embeddings
$\D_{\mathrm{T0}} \hookrightarrow \D_{\mathrm{T1}} \hookrightarrow \D_{\mathrm{T2}} \hookrightarrow \D_{\mathrm{T3}}$
preserve them. In each level, we have $[f] \vee [g] = [f \oplus g]$
where $(f \oplus g)(2n) := f(n)$ and $(f \oplus g)(2n+1) := g(n)$.
\end{prop}
\begin{proof}
We need simply observe that $f\oplus g \redTThree h$ iff $f\redTThree
h$ and $g\redTThree h$ (since the $\oplus$ construction preserves each
T0–T3 level, if it computes the join at the T3 level, the same holds
for the lower levels).  But this is now obvious by observing that the
initial configuration for the reduction game of $f\oplus g$ to $h$
tells Arthur+Nimue which reduction strategy ($f\redTThree h$
or $g\redTThree h$) to apply.
\end{proof}

It is also true that binary joins exist in basic T3 degrees and are
preserved by the embedding into T3 degrees, but the construction must
be different; this is described in \cite[prop. 5.12]{LeeVanOosten}.

\medskip

\thingy\label{recollection-meet-of-t0-degrees} Now we consider the
more interesting case of binary meets (inf, greatest lower
bounds). These are well-known not to always exist in the T0 degrees
(see \cite[definition 3.4.1(v) and corollary 6.5.5]{Soare} or
\cite[exercises 10.3.10 and 13.4.24]{CooperBook}), which might be
seen as a deficiency of the T0 degrees.

We will now proceed to explain why T1, T2 and T3 degrees have all
binary meets, and how they relate.  The construction is more delicate,
but still fairly explicit.  Intuitively, having access to the meet of
two oracles means that Arthur can only get answers that he can get
from \emph{both} of these oracles, but exactly what “getting answers” means
needs to be clarified for each level (and we will see that we get
different operations at the T1 level and at the T2/T3 levels).

Let us start with T1 meets, which will play a crucial role in
\cref{section-comparison-of-codegree-with-degree}.

\begin{prop}\label{meet-of-t1-degrees}
The partially ordered set $\D_{\mathrm{T1}}$ has all binary meets. An explicit
construction is the following: for T1 oracles
$f, g \colon \N \dasharrow \N$, we have
$[f]_\TOne \meetDOne [g]_\TOne = [f \meetTOne g]$,
where $f \meetTOne g \colon \N \dasharrow \N$
is defined on $\langle p, q, m, n \rangle$ precisely when $\varphi_p^f(m)$ and
$\varphi_q^g(n)$ are both defined and have the same value, and then
is equal to that value.
\end{prop}

\begin{proof}
This is \cite[prop. 3.1]{KiharaNg1}, but we repeat the proof for
the reader's convenience.  We have $f \meetTOne g \redTOne f$ because we
can use $f$ to compute $\varphi_p^f(n)$, and likewise $f
\meetTOne g \redTOne g$. On the other hand, if $h \redTOne f$
and $h \redTOne g$, by the description of T1 reducibility in the
second paragraph of \ref{simplification-of-reducibility},
we have $p$ and $q$ such
that whenever $h(n)$ is defined, it equals both $\varphi_p^f(n)$ and
$\varphi_q^g(n)$, so $h \redTOne f \meetTOne g$.
\end{proof}

\begin{rmk}
Many small variations on the definition of $f \meetTOne g$ are possible.
Clearly we could take $m=n=0$ (by the s-m-n theorem).
Slightly less obvious is the fact, sometimes known as “Posner's
trick”, that we could (also, or alternatively) take $p=q$: indeed, if
$f$ and $g$ coincide whenever both are defined, then the result is
obvious, and if there is $k$ such that they are defined and differ at
$k$, then the program can test which oracle it is dealing with by
first querying the value at $k$.
\end{rmk}

\thingy As we have seen in \cref{join-of-degrees}, the binary join
operation does not depend on the level at which it is computed.  The
situation is completely different for the binary meet between the T0
and T1 levels: not only can the T0 meet fail to exist (as we have
recalled in \ref{recollection-meet-of-t0-degrees}), but even if it
does, it can fail to equal the T1 meet — indeed, we will see shortly
that it is essentially \emph{never} equals to the T1 meet.

Recall that the Turing (T0) degrees are known to contain \textbf{minimal pairs},
namely, pairs of non-zero degrees whose (T0) meet is zero
(cf. \cite[theorem 6.2.3]{Soare} or \cite[theorem 12.5.2]{CooperBook}).
The following proposition shows that, in contrast, the T1 meet of two T0
degrees can only be zero if one of them is: this is
\cite[corollary 4.8]{KiharaNg1}, and the latter will be sufficient for
our purposes, but we still record a slight generalization and give a
more direct proof, for the idea of which we are indebted\footnote{The
author of this answer did not respond to our comment asking them how
they prefer to be acknowledged, and we do not now their true
identity, so we must cite them as “Pastebee”.} to
\cite{MO507736}.

\begin{prop}\label{computability-of-t1-meet}
\textit{(Plain version.)}  If $\mathbf{f} \in \D_{\mathrm{T1}}$ and $\mathbf{g} \in \D_{\mathrm{T0}}$ are such that
$\mathbf{f} \meetDOne \mathbf{g} = \mathbf{0}$, then
already $\mathbf{f} = \mathbf{0}$ or $\mathbf{g} = \mathbf{0}$.

\textit{(Relativized version.)}  If $\mathbf{f} \in \D_{\mathrm{T1}}$ and $\mathbf{g} \in \D_{\mathrm{T0}}$ and $\mathbf{h} \in \D_{\mathrm{T0}}$ are such that
$\mathbf{f} \meetDOne \mathbf{g} \leq \mathbf{h}$, then
in fact $\mathbf{f} \leq \mathbf{h}$ or $\mathbf{g} \leq \mathbf{h}$.
\end{prop}

\begin{proof}
Let us first prove the “plain” version.

Without loss of generality, we can assume our degrees to be
represented by oracles with values in $\{0, 1\}$, namely
$f\colon\N\dasharrow\{0,1\}$ and $g\colon\N\to\{0,1\}$.  Indeed, it is easy
to check that the T1 degree of a function $f \colon \N \dasharrow \N$
equals the T1 degree of $\tilde f \colon \N \dasharrow \{0, 1\}$
where $\tilde f(\langle m,u\rangle){\downarrow}$ precisely when
$f(m){\downarrow}$, in which case $\tilde f(\langle m,u\rangle) = 1$
if $f(m) = u$ and $\tilde f(\langle m,u\rangle) = 0$ if $f(m) \neq u$
(and of course, if $g$ is total, then so is $\tilde g$, and anyway the
argument is standard in this case).

Now (recall ¶\ref{simplification-of-reducibility}) our assumption
$\mathbf{f} \meetDOne \mathbf{g} = \mathbf{0}$ means that $f \meetTOne
g$ has an extension which is a partial computable function
$\varphi_e$.

Let $r$ be the code of the oracle program that takes an argument,
reads the oracle's value at that argument and directly returns it,
so that $\varphi_r^k = k$ for all $k \colon \N \dasharrow \N$.
Similarly, let $r'$ read the oracle's value $v$ at its argument and
return its opposite $\lnot v := 1-v$ as a boolean (the behavior of $r'$
on non-booleans is irrelevant).

Consider the following putative algorithm $X$ to compute an extension of $f$.
Given an input $m$, we use dovetailing to search
in parallel for $n \in \N$ such that $\varphi_e(\langle r, r, m, n \rangle)$
and $\varphi_e(\langle r, r', m, n \rangle)$ are both defined and equal,
and then return this common value. Let us check the partial correctness
of this algorithm. If it is called on an input $m$ such that $f(m)$
is defined and \emph{if} it finds $n$ in its search, then either $f(m) = g(n)$, in
which case $\varphi_e(\langle r, r, m, n \rangle){\downarrow} = f(m)$,
or $f(m) = \lnot g(n)$, in which case
$\varphi_e(\langle r, r', m, n \rangle){\downarrow} = f(m)$;
in both cases, the algorithm's return value is correct. (The hypothesis
that $\mathbf{g}$ is T0 is being used here to ensure that $g(n)$ is
defined, allowing the case distinction.)
Hence, if the algorithm $X$ terminates on all inputs where $f$ is defined, then
$\mathbf{f} = \mathbf{0}$.

Now assume that $m_0$ is an input such that $f(m_0)$ is defined but the
algorithm $X$ just described does not terminate. In this case, we claim that
$\mathbf{g} = \mathbf{0}$ because the following other algorithm $Y$ computes $g$:
given an input $n$, compute in parallel $\varphi_e(\langle r, r, m_0, n \rangle)$
and $\lnot \varphi_e(\langle r, r', m_0, n \rangle)$ until one of the two
terminates, and then return it. The assumption that the previous putative
algorithm $X$ for $f$ does not terminate on $m_0$ means that the values
$\varphi_e(\langle r, r, m_0, n \rangle)$ and
$\lnot \varphi_e(\langle r, r', m_0, n\rangle)$ can never be both defined and
different, so it suffices to see that one of them must be defined
and correct. If $f(m_0) = g(n)$, then $\varphi_e(\langle r, r, m_0, n\rangle)$
is, whereas if $f(m_0) = \lnot g(n)$, then
$\lnot \varphi_e(\langle r, r', m_0, n\rangle)$ is.

This concludes the proof of the “plain” version of the theorem.

\medskip

To prove the “relativized” version, we make the following observation:
the proof of the “plain” version which we have just given applies just
as well if we replace computability by computability relative to an
ordinary Turing oracle $h \colon \N \to \N$ (so we can replace
$\varphi_e$ by $\varphi_e^h$ throughout).
\end{proof}

\begin{rmk}\label{remark-on-t1-meet}
The hypothesis that $\mathbf{h}$ be T0 was used in the proof of the
relativized version in order to run computations in parallel.
But let us explain why it is indeed necessary for the conclusion to hold.
Indeed, if $\mathbf{f}$ and $\mathbf{g}$
are incomparable T0 degrees (which exist by the classical Kleene–Post
theorem \cite[theorem 6.1.1]{Soare}), and if we let $\mathbf{h} :=
\mathbf{f} \meetDOne \mathbf{g}$ (which is a T1 degree) then
$\mathbf{f} \meetDOne \mathbf{g} \leq \mathbf{h}$
holds trivially, so if we could apply the proposition to $\mathbf{h}$
being a T1 degree we would get either
$\mathbf{f} \leq \mathbf{h} \leq \mathbf{g}$ or $\mathbf{g} \leq
\mathbf{h} \leq \mathbf{f}$, contradicting the assumption that
$\mathbf{f},\mathbf{g}$ are incomparable.

We similarly observe that, still working under the assumption that
$\mathbf{f}$ and $\mathbf{g}$ are incomparable T0 degrees
even if the meet $\mathbf{f} \meetDZero \mathbf{g}$ happens to exist,
the obvious inequality $\mathbf{f} \meetDZero \mathbf{g} \leq
\mathbf{f} \meetDOne \mathbf{g}$ \emph{must} be strict.  Indeed,
otherwise, applying the proposition to $\mathbf{h} :=
\mathbf{f} \meetDZero \mathbf{g}$ (which we assumed exists), then
$\mathbf{f} \meetDOne \mathbf{g} \leq \mathbf{h}$
would imply $\mathbf{f} \leq \mathbf{h} \leq \mathbf{g}$
or $\mathbf{g} \leq \mathbf{h} \leq \mathbf{f}$, again contradicting
the assumption that $\mathbf{f},\mathbf{g}$ are incomparable.

On the other hand, we do not know whether the assumption that
$\mathbf{g}$ be T0 (which was used in the proof) is truly essential
(see \cref{question-computability-of-t1-meet}).
\end{rmk}

\bigskip

Although we will not do much with it, we describe the meet on the
T2 and T3 levels for completeness.  (Recall that $U \sqcup V :=
\{\langle 0, u \rangle \mid u \in U\} \cup \{\langle 1, v \rangle \mid
v\in V\}$.)

\begin{prop}\label{meet-of-t2-degrees}
The partially ordered sets $\D_{\mathrm{T2}}$ and $\D_{\mathrm{T3}}$ have all binary meets,
and the embedding $\D_{\mathrm{T2}} \hookrightarrow \D_{\mathrm{T3}}$ preserves them.
On the T2 level, an explicit construction is the following:
for T2 oracles $f, g \colon \N \dasharrow \P(\N)$, we have
$[f]_\TTwo \meetDTwo [g]_\TTwo = [f \meetTTwo g]_\TTwo$
where $f \meetTTwo g$ is defined on $\langle m, n \rangle$ exactly
when $f(m)$ and $g(n)$ both are and is then equal to $f(m) \sqcup g(n)$.
On the T3 level, we have the construction
$[f]_\TThree \meetDThree [g]_\TThree = [f \meetTThree g]_\TThree$
where $(f \meetTThree g)(\langle m, n \rangle)$ is again defined
when $f(m)$ and $g(n)$ both are, and then equal to
$\{U \sqcup V \mid U \in f(m), V \in g(n)\}$.
\end{prop}

In words, when querying $f \meetTTwo g$ or $f \meetTThree g$, Arthur puts
forth two questions, one for $f$ and one for $g$; on the T3 level Nimue
also makes two corresponding moves; then Merlin answers one question of
his choice among the two.  (Symmetrically, of course, when trying to
reduce $f \meetTThree g$ to some other oracle, the initial
configuration will present two challenges, one for $f$ and one
for $g$, and Arthur must eventually declare a value answering one of
these challenges, indicating which.)

\begin{proof}[Proof of \cref{meet-of-t2-degrees}]
It is enough to show that the construction provides the meet in the
T3 degrees, since when it is applied to two T2 oracles it gives a
T2 oracle as described.

The fact that $f \meetTThree g \redTThree f$ and $f \meetTThree g \redTThree g$
is immediately checked. We now assume that $h \redTThree f$ and
$h \redTThree g$ and prove that $h \redTThree f \meetTThree g$.
For this, Arthur and Nimue use the following strategy in the reduction game
of $h$ to $f \meetTThree g$.  They mentally simulate two reduction games
in parallel, one of $h$ to $f$ and one of $h$ to $g$ (for which both
they have winning strategies by assumption), starting with the
same initial configuration (that demanded by the reduction game
of $h$ to $f \meetTThree g$).  Each time it is Arthur's
turn to play in the reduction of $h$ to $f \meetTThree g$,
he consults his strategies in the reduction of $h$ to $f$ and
$h$ to $g$: if either one tells him to declare a final value, then he
does so (taking the value from the strategy that terminates first);
otherwise, both tell him to query the oracle with
questions $m$ (to $f$) and $n$ (to $g$), and he then queries the $f \meetTThree g$
oracle at $\langle m,n\rangle$.  Similarly, when it is Nimue's turn
to play, she performs the moves dictated by her strategies in the
reduction of $h$ to $f$ and $h$ to $g$.  Merlin will then output a
move of the form $\langle 0,u\rangle$ (providing an answer of the $f$
oracle) or $\langle 1,v\rangle$ (providing one of $g$): in the
first case, the move $u$ is assumed to have been played by Merlin in
the reduction game of $h$ to $f$ while the reduction game of
$h$ to $g$ is left unchanged, and the second case is symmetric.  Since
each move makes progress in at least one of the two reduction games,
eventually the game will terminate with Arthur declaring a value,
which must be correct because it is produced by a winning strategy in
one of the two simulated games.
\end{proof}

\begin{rmk}
We could bring the construction of ¶\ref{meet-of-t2-degrees} (say, for
T2 oracles) closer to that of ¶\ref{meet-of-t1-degrees} by allowing
Arthur to put forth a query of the form $\langle p,q,m,n\rangle$ and
Merlin to respond with an element of the labeled disjoint union of the
set of possible values of $\varphi_p^f(m)$ and of $\varphi_q^g(n)$ (a
more precise definition being: $J_f^\natural(p) \sqcup
J_g^\natural(q)$ with the notations
of ¶\ref{definition-oracle-topology-correspondence}), which would make
the above proof easier.  The point here, however, is that, for T2/T3
degrees, we do not \emph{need} to resort to this trick of going
through a “universal” construction $p \mapsto \varphi_p^f(0)$ whereas
for T1 degrees, we \emph{do} need: trying to define the T1 meet by the
function taking $\langle m,n\rangle$ to the common value of
$f(m)$ and $g(n)$ when both are defined or equal, would obviously fail
for silly reasons (e.g., $f$ might take values in $\{0,1\}$ and $g$
in $\{2,3\}$).
\end{rmk}

\begin{rmk}\label{remark-on-t2-versus-t1-meet}
One remarkable result of \cite{KiharaNg1} is that even though
the T1 degrees form a lattice, that lattice is \emph{not distributive}
(\textit{op. cit.}, theorem 4.9), nor even modular (\textit{op. cit.},
theorem 4.10). In contrast, the lattice of T2 or T3 degrees \emph{is} distributive.
For the T3 degrees, anticipating on the point of view exposed in
\cref{section-topos-theory}, this follows from the identification with
Lawvere–Tierney topologies in \cite[corollary 3.5]{KiharaLT} (which we
reproduced as \cref{statement-oracle-topology-correspondence})
combined with the
general fact that the Lawvere–Tierney topologies on an elementary
topos form a Heyting algebra (\cite[example A4.5.14(f)]{JohnstoneElephant};
we will return to this in ¶\ref{digression-heyting-operation})
and in particular a distributive lattice. For the T2 degrees, this
follows from the fact just proved that the inclusion
$\D_{\mathrm{T2}} \hookrightarrow \D_{\mathrm{T3}}$ preserves binary meets and joins.

Alternatively, it can easily be checked, directly from the
descriptions in propositions \ref{join-of-degrees}
and \ref{meet-of-t2-degrees}, that $(f\oplus g) \meetTThree (f\oplus
h) \, \redTThree \, f\oplus (g\meetTThree h)$ holds in
general\footnote{In contrast, it is worth thinking about why
\cref{meet-of-t1-degrees} does \emph{not} permit the same proof to
carry over to T1 degrees.}, which is enough to prove distributivity
for T3 (and hence T2) degrees.

It follows that the inclusion $\D_{\mathrm{T1}} \hookrightarrow \D_{\mathrm{T2}}$ cannot preserve binary
meets and joins, and since it does preserve binary joins
(\cref{join-of-degrees}), we conclude that it does not preserve binary
meets.  So there exist T1 degrees $\mathbf{f},\mathbf{g}$
such that the obvious inequality $\mathbf{f} \meetDOne \mathbf{g} \leq
\mathbf{f} \meetDTwo \mathbf{g}$ is strict.
We do not understand this situation well; see
\cref{question-t2-versus-t1-meet}.
\end{rmk}

We now proceed to prove two propositions concerning the T3 meet that
are very similar to \cref{computability-of-t1-meet} for the T1 meet.
Either one of them will suffice for our purposes (which are very
modest), but since we believe they may be of independent interest we
have chosen to give both.  The reader is advised to pay careful
attention to the hypotheses on $\mathbf{f},\mathbf{g},\mathbf{h}$.

\begin{prop}\label{computability-of-t3-meet}
Let us say that a T3 oracle $g \colon \N \to \P(\P(\N))$ is
\textbf{totally inhabited} when $g(n)$ is inhabited for all $n$.

Let $\mathbf{f}, \mathbf{g} \in \D_{\mathrm{T3}}$ and $\mathbf{h} \in
\D_{\mathrm{T1}}$, where $\mathbf{g}$ is represented by a \emph{totally
inhabited} oracle. If $\mathbf{f} \meetDThree \mathbf{g} \leq \mathbf{h}$, then in fact $\mathbf{f} \leq \mathbf{h}$ or $\mathbf{g} \leq \mathbf{h}$.
\end{prop}
\begin{proof}
Let $\mathbf{f},\mathbf{g},\mathbf{h}$ be represented by functions $f
\colon \N \to \P(\P(\N))$, $g \colon \N \to \P(\P(\N))$ with $g(n)$
inhabited for every $n$, and $h \colon \N \dasharrow \N$ respectively.

In light of \cref{meet-of-t2-degrees}, let us first unwrap what it
means that $f \meetTThree g \redTThree h$.  Note that since $h$ is T1,
Nimue does not play, and we can think of Arthur's strategy $e$ as
defining a partial $h$-computable function $\varphi_e^h$
(cf. ¶\ref{simplification-of-reducibility}), taking as input the
Arthur-visible part $\langle m,n\rangle$ of the initial configuration
for $f \meetTThree g$.  Arthur cannot see the secret part ($F \in
f(m)$ and $G \in g(n)$) of the initial configuration in which he is
supposed to play, so, absent help from Nimue, he will need to play in
$\bigcap_{F\in f(m)} F$ or $\bigcap_{G\in g(n)} G$, depending on which
challenge he chooses to address (we refer
to ¶\ref{when-are-oracles-computable} for the appearance of
$\bigcap_{F\in f(m)} F$ here).  So, the assumption is that there is an
$e$ such that, for all $m,n$ with $f(m)$ inhabited, \emph{either}
“Arthur addresses the $f$ challenge”: $\varphi_e^h(\langle
m,n\rangle){\downarrow} = \langle 0,u\rangle$ with $u \in
\bigcap_{F\in f(m)} F$, \emph{or} “Arthur addresses the $g$ challenge”:
$\varphi_e^h(\langle m,n\rangle){\downarrow} = \langle 1,v\rangle$
with $v \in \bigcap_{G\in g(n)} G$.  The assumption that $g(n)$ is
inhabited for all $n$ ensures that $\varphi_e^h(\langle
m,n\rangle){\downarrow}$ as soon as $f(m)$ is inhabited (with no need
to discuss on $n$).

Now by classical logic (and using the fact that $\varphi_e^h$ is a
single-valued function, defined in the cases we have just said), one
of two things must happen: \emph{either} for all $m$ such that $f(m)$
is inhabited there exists $n$ with $\varphi_e^h(\langle
m,n\rangle){\downarrow} = \langle 0,\emdash\rangle$, \emph{or} there
exists $m_0$ with $f(m_0)$ inhabited such that for all $n$ we have
$\varphi_e^h(\langle m_0,n\rangle){\downarrow} = \langle
1,\emdash\rangle$.

In the first case, the following strategy $X$ shows that $f \redTThree
h$: given $m$, we search sequentially\footnote{We emphasize that we
  are searching \emph{sequentially} here: all the $\varphi_e^h(\langle
  m,n\rangle)$ are guaranteeed to be defined thanks to our assumption
  on $g$, so we do not need to dovetail, which is fortunate as
  dovetailing is not possible under a T1 oracle.} for an $n$ such that
$\varphi_e^h(\langle m,n\rangle){\downarrow} = \langle 0,u\rangle$,
and when such an $n$ is found, we return $u$ (i.e., declare it as
final value).  The assumption of the first case shows precisely that
the loop will terminate and $u \in \bigcap_{F\in f(m)} F$, so this is
a correct strategy for reducing $f$ to $h$.

In the second case, the following strategy $Y$ shows that $g
\redTThree h$: simply compute $\varphi_e^h(\langle
m_0,n\rangle){\downarrow} = \langle 1,v\rangle$ and return $v$ (i.e.,
declare it as final value).  The assumption of the second case shows
precisely that this will give $v \in \bigcap_{G\in g(n)} G$, so this
is a correct strategy for reducing $g$ to $h$.
\end{proof}

\begin{prop}\label{computability-of-meet-of-basic-and-t3-degrees}
If $\mathbf{f} \in \D_{\mathrm{T3}}$, and $\mathbf{g}$ is a
\emph{basic} T3 degree, and $\mathbf{h} \in \D_{\mathrm{T2}}$ are such
that $\mathbf{f} \meetDThree \mathbf{g} \leq \mathbf{h}$, then in fact
$\mathbf{f} \leq \mathbf{h}$ or $\mathbf{g} = \mathbf{0}$.
\end{prop}
\begin{proof}
Let $\mathbf{f},\mathbf{g},\mathbf{h}$ be represented by
$f \colon \N \to \P(\P(\N))$, $\mathscr{G} \in \P(\P(\N))$, and $h \colon
\N \dasharrow \P(\N)$ respectively.

In light of \cref{meet-of-t2-degrees}, let us unwrap what it means that
$f \meetTThree \mathscr{G} \redTThree h$.  The initial configuration
consists of two challenges, one for the $f$ and one for the
$\mathscr{G}$ oracle, and Arthur can choose to address either one: the
first challenge takes the form of a public $m$ and a secret element $F$
of $f(m)$, to which Arthur is supposed to respond by giving an element
of $F$; the second challenge takes the form of a secret element $G$ of
$\mathscr{G}$, to which Arthur is supposed to respond by giving an
element of $G$; Arthur can choose to address either challenge (and must
specify which).  And the hypothesis $f \meetTThree \mathscr{G}
\redTThree h$ is that Arthur has a computable winning strategy for
this game, without Nimue's help (since $h$ is T2), but with the help
of (adverserially non-deterministic) output from the $h$ oracle.  Note
that Arthur does not see anything from the second challenge: all he
sees is the visible part $m$ of the challenge to the $f$ oracle.

Now if there is an initial $m_0$ and a sequence of moves by Merlin
(responding to queries to the $h$ oracle) such that Arthur's strategy
chooses to address the second challenge in this situation, then the
element it produces is in $\bigcap_{G\in\mathscr{G}} G$ because it
cannot depend on the secret $G$ part of the initial configuration
(more precisely: the same play would have unfolded identically for any
element of $\mathscr{G}$).  But this means that
$\bigcap_{G\in\mathscr{G}} G$ is inhabited, i.e.,
$[\mathscr{G}]_{\mathrm{T3}} = \mathbf{0}$
(cf. ¶\ref{when-are-oracles-computable}).  Otherwise, Arthur's
strategy is to \emph{always} address the $f$ oracle (no matter what
the initial configuration and the answers by Merlin), and this shows
$f \redTThree h$.
\end{proof}

\begin{rmk}
Clearly, propositions \ref{computability-of-t1-meet},
\ref{computability-of-t3-meet}
and \ref{computability-of-meet-of-basic-and-t3-degrees} have a common
theme (“under some assumptions, $\mathbf{f} \wedge \mathbf{g}
\leq \mathbf{h}$ can only happen if $\mathbf{f} \leq \mathbf{h}$ or
$\mathbf{g} \leq \mathbf{h}$”) and there is also a common theme to
their proofs.  The deeper meaning of this theme escapes us, however.

Since a basic T3 oracle is either $\varnothing$ (which is of
degree $\mathbf{0}$) or totally inhabited when extended to a constant
function, propositions \ref{computability-of-t3-meet}
and \ref{computability-of-meet-of-basic-and-t3-degrees} have a common
special case (or should we say a “meet”?), namely the case when
$\mathbf{f}$ is T3, $\mathbf{g}$ is a \emph{basic} T3 degree, and
$\mathbf{h}$ is T1.  In fact, the only case we will really use in the
sequel is the even more special one when $\mathbf{f}$ is T0,
$\mathbf{g}$ is \emph{basic}, and
$\mathbf{h} = \mathbf{0}$, so either proposition suffices.

We may wonder, however, whether the two propositions have a natural common
generalization (or should we say a “join”?).  In light of the first
sentence of the previous paragraph, the obvious candidate would be:
“If $\mathbf{f} \in \D_{\mathrm{T3}}$, and $\mathbf{g} \in
\D_{\mathrm{T3}}$ is represented by a totally inhabited oracle, and
$\mathbf{h} \in \D_{\mathrm{T2}}$ are such that $\mathbf{f}
\meetDThree \mathbf{g} \leq \mathbf{h}$, then in fact $\mathbf{f} \leq
\mathbf{h}$ or $\mathbf{g} \leq \mathbf{h}$.”  This statement,
however, is \emph{false}: indeed, as in \cref{remark-on-t1-meet}, let
$\mathbf{f},\mathbf{g}$ be incomparable T0 degrees, and let
$\mathbf{h} = \mathbf{f} \meetTThree \mathbf{g}$ which is also
$\mathbf{f} \meetTTwo \mathbf{g}$ by \cref{meet-of-t2-degrees}, so it
is a T2 degree: then if the quoted statement were to hold, we could
conclude that either $\mathbf{f} \leq \mathbf{h} \leq \mathbf{g}$ or
$\mathbf{g} \leq \mathbf{h} \leq \mathbf{f}$, contradicting the
assumption that $\mathbf{f},\mathbf{g}$ are incomparable.
\end{rmk}

\thingy\label{digression-heyting-operation}
\textit{A digression.} Even though this will play only a very minor
role in what we say below, we would be remiss if we did not at least
mention the \textbf{Heyting operation} on the T3 degrees.  As we
mentioned earlier, $\mathscr{D}_\TThree$ is a Heyting algebra: this
follows from the identification with Lawvere–Tierney topologies in
\cite[corollary 3.5]{KiharaLT}, which we reproduced
as \cref{statement-oracle-topology-correspondence} below, and the
general fact that the Lawvere–Tierney topologies on a topos form a
Heyting algebra (\cite[example A4.5.14(f)]{JohnstoneElephant}).  This
can certainly be verified more directly, however: we have not done
this, but we can at least sketch a concrete description of the Heyting
operation $\mathbf{g} \mathbin{\Rightarrow_\TThree} \mathbf{h}$ on
T3 degrees, which by definition is the largest degree $\mathbf{f}$
such that $\mathbf{f} \meetDThree \mathbf{g} \leq \mathbf{h}$.  We
will, however, leave it as an exercise to check that this description
is indeed correct.

Specifically, if $g,h$ are T3 oracles representing
$\mathbf{g},\mathbf{h}$, then $\mathbf{g}
\mathbin{\Rightarrow_\TThree} \mathbf{h}$ is represented by a
T3 oracle $s$ described as follows.  First, we consider a slightly
modified form of our reduction game of $g$ to $h$ in that Arthur is now
additionally allowed to play at any point a special “surrender” move,
which takes an arbitrary value $w\in\N$ as parameter, ending the game
and causing Merlin to win.  Let us say that a “valid” Arthur+Nimue
strategy in this game (a weaker condition than a “winning” strategy)
is one where, against any legal moves by Merlin, Arthur and Nimue
either win or end with Arthur surrendering in finite time (in other
words, the strategy always produces a legal move, never leaving Nimue
in a position where she cannot move, and the game always terminates in
finite time); equivalently, a “valid” Arthur+Nimue strategy is just a
winning one if we fictitiously treat surrender as winning for
Arthur+Nimue (making the game very boring!).  Obviously there are
always valid Arthur+Nimue strategies, but there are only winning ones
iff $g \redTThree h$ (by definition).

Now we let $\check s(e,\sigma) \in \P(\N)$ be defined as follows: if
$(e,\sigma)$ is a valid Arthur+Nimue strategy (as we have just
defined), with $e \in \N$ being the (Gödel code of the) Arthur part
and $\sigma$ the Nimue part, we let $\check s(e,\sigma)$ be the the
set of $w\in\N$ such that there is at least one sequence of moves by
Merlin that causes Arthur to surrender with parameter $w$.  (If
$(e,\sigma)$ is \emph{not} a valid Arthur+Nimue strategy, then $\check
s(e,\sigma)$ is left undefined.  If $(e,\sigma)$ \emph{is} valid but
there are no case where Arthur surrenders, then $\check s(e,\sigma) =
\varnothing$.)  Finally\footnote{What we are doing here is that we
have defined our T3 oracle though its presentation as a function
$\check s\colon \mathbb{N}\times\Lambda \dasharrow
\mathcal{P}(\mathbb{N})$ as explained in
¶\ref{differences-with-kihara-setup} (for once, this is more
convenient), and we convert it back to our usual presentation
$\N\to\P(\P(\N))$.}, we let $s(e)$ be the set of the $\check
s(e,\sigma)$ for all $\sigma$ such that $(e,\sigma)$ is a valid
Arthur+Nimue strategy (if there are no such $\sigma$, then $s(e)$ will
be $\varnothing$; but if there is $\sigma$ such that $\check
s(e,\sigma) = \varnothing$ then $s(e)$ will be defined and contain
$\varnothing$).  This defines our $s\colon \N \to \P(\P(\N))$.

More informally, $s$ takes a possibly-surrendering (but otherwise
valid) Arthur strategy $e$ for the reduction of $g$ to $h$, and
returns a value $w$ with which it might actually
surrender.  (If there is none, meaning $g \redTThree h$, then for
some $e$, Arthur will never give up, never surrender: then
$\varnothing \in s(e)$ so $[s]_\TThree = \top$.)

Using the description of T3 meets given in \cref{meet-of-t2-degrees},
one can then check that $f \meetTThree g \redTThree h$
iff $f \redTThree s$, so that $s$ indeed represents $\mathbf{g}
\mathbin{\Rightarrow_\TThree} \mathbf{h}$.

% It is conceivable that I phrased the above paragraph deliberately in
% such a way as to place the quote “Never give up! Never surrender!”
% (from ‘Galaxy Quest’) in a serious math paper.
% <U+1F605 SMILING FACE WITH OPEN MOUTH AND COLD SWEAT>

\medskip

Using this Heyting operation, propositions
\ref{computability-of-t3-meet}
and \ref{computability-of-meet-of-basic-and-t3-degrees} can be
expressed as follows:
\begin{itemize}
\item If $\mathbf{g}$ is totally inhabited and $\mathbf{h}$ is T1,
  then $\mathbf{g} \mathbin{\Rightarrow_\TThree} \mathbf{h}$ either
  equals $\mathbf{h}$ or $\top$.
\item If $\mathbf{g}$ is basic and $\mathbf{h}$ is T2, the same
  conclusion holds.
\end{itemize}

\section{Decoding and coding degrees associated to promise problems}\label{section-decoding-and-coding-degrees}

\begin{dfn}\label{definition-coding-decoding-degrees}
Let $P,Q\subseteq\mathbb{N}$ be inhabited and disjoint.  We define:
\begin{itemize}
\item The \textbf{decoding oracle} associated to $(P,Q)$ to be the T1
  oracle defined by the partial function $\mathbb{N} \dasharrow
  \mathbb{N}$ that takes the value $0$ on $P$ and $1$ on $Q$ and is
  undefined otherwise.  We denote by $\mathbf{d}_{P,Q}$ the
  corresponding (T1) degree, and sometimes, by abuse of language, the
  decoding oracle itself.
\item The \textbf{coding oracle} associated to $(P,Q)$ to be the basic
  T3 oracle defined with the element $\{P,Q\}$ of
  $\mathcal{P}(\mathcal{P}(\mathbb{N}))$.  We denote by
  $\mathbf{c}_{P,Q}$ the corresponding (T3) degree, and sometimes, by
  abuse of language, the coding oracle itself.
\end{itemize}
\end{dfn}

\thingy In more informal terms: the decoding oracle associated to
$(P,Q)$ takes an element of one of those sets and tells Arthur which
one of them it is in, whereas the coding oracle provides Arthur with
the answer to an arbitrary “yes”/“no” question (i.e. communicates an
arbitrary bit chosen by Nimue) but encodes it by giving him an element
of $P$ for 0 and $Q$ for 1.

Thus, the setup of our opening riddle (¶\ref{intro-riddles}) concerns
the coding oracles $\mathbf{c}_{P, Q}$ where first $P := \bar K$ and
$Q := K$, and then $P := (K\times K) \,\cup\, (\bar K \times \bar K)$ and
$Q := (K\times \bar K) \,\cup\, (\bar K \times K)$, where $K := \{e \mid
\varphi_e(0){\downarrow}\}$ is the halting set and $\bar K :=
\mathbb{N} \setminus K$ its complement.  In the first part of the
riddle we explained, in effect\footnote{A rigorous justification will
come from combining
\cref{easy-example-decoding-checker-inequality-can-be-strict} and
\cref{checker-coding-is-same-as-plain-coding}(ii) below.},
that $\mathbf{c}_{\bar K, K} = \Om$.
We will answer the second part (showing that then $\mathbf{c}_{P, Q}
< \Om$) in ¶\ref{answer-to-riddle}.

\thingy\label{promise-problem} In the context of complexity theory,
the problem of deciding, given an element guaranteed to be in
$P \cup Q$, whether it is in $P$ or in $Q$, is known as a
\textbf{promise problem}.  We might therefore say that the decoding
oracle associated to $(P,Q)$ solves the corresponding promise problem,
whereas the coding oracle solves an arbitrary question but presents
its answer as an instance of the promise problem. In particular,
$\mathbf{d}_{P,Q} = \mathbf{0}$ means precisely that the promise
problem $(P,Q)$ is \textbf{decidable} or “computably solvable”.

\begin{exm}\label{example-set-and-complement}
If $P$ and $Q$ are complement of one another, then $\mathbf{d}_{P,Q} =
[P]_{\mathrm{T0}} = [Q]_{\mathrm{T0}}$.

In general, we have $\mathbf{d}_{P,Q} \leq [P]_{\mathrm{T0}}$ and,
symmetrically, $\mathbf{d}_{P,Q} \leq [Q]_{\mathrm{T0}}$ (which we can
put together as: $\mathbf{d}_{P,Q} \leq [P]_{\mathrm{T0}} \meetDOne
[Q]_{\mathrm{T0}}$).
\end{exm}

Let us immediatly record the following easy but crucial fact:

\begin{prop}\label{coding-and-decoding-give-omniscience}
We have $\mathbf{d}_{P,Q} \vee \mathbf{c}_{P,Q} = \Om$.  (In
particular, if $\mathbf{d}_{P,Q} = \mathbf{0}$ then $\mathbf{c}_{P,Q}
= \Om$.)
\end{prop}
\begin{proof}
If we have access to both the coding and the
decoding oracle, this allows Nimue to communicate a bit to Arthur
(recall that we explicitly described the case of a reduction from
$\Om$ in ¶\ref{reduction-from-om}, and that communicating a bit is
enough for omniscience by \cref{exercise-reduction-to-a-bit}).

More precisely, the reduction of $\Om = [\{\{0\},\{1\}\}]_{\TThree}$
to $\mathbf{d}_{P,Q} \vee \mathbf{c}_{P,Q}$ goes as follows: if $b$ is
the target bit (which Nimue knows but Arthur does not), then on the
first move, Arthur queries the $\mathbf{c}_{P,Q}$ oracle, Nimue's
strategy plays $P$ if $b=0$ and $Q$ if $b=1$, forcing Merlin to
communicate to Arthur an element of $P$ if $b=0$ and $Q$ if $b=1$; and
on the second move, Arthur queries the $\mathbf{d}_{P,Q}$ oracle at
the value he received on the first move, to deduce the value of $b$.
\end{proof}

\begin{exm}\label{example-disjoint-c-e-sets}
If $P$ and $Q$ are computably enumerable and disjoint, then
$\mathbf{d}_{P,Q} = \mathbf{0}$.  Indeed, to decide whether an element
is in $P$ or in $Q$ under the promise that it be in $P \cup Q$, we
simply enumerate elements of $P$ and of $Q$ in parallel until we find
the one we sought.  It then follows from
\cref{coding-and-decoding-give-omniscience} that
$\mathbf{c}_{P,Q} = \Om$.

The following particular case is of importance (and closely related to
the definition of the $\tilde{\mathbf{d}}_A,\tilde{\mathbf{c}}_A$ we
will give in \cref{definition-codegrees}): the sets $H_0,H_1$ where
$H_i := \{e \mid \varphi_e(0){\downarrow} = i\}$ are computably
enumerable and disjoint, so $\mathbf{d}_{H_0,H_1} = \mathbf{0}$ (we
will generalize this
in \cref{decoding-relative-halting-is-turing-degree}).  This example
shows that the inequality $\mathbf{d}_{P,Q} \leq [P]_{\mathrm{T0}}
\meetDOne [Q]_{\mathrm{T0}}$ mentioned
in \cref{example-set-and-complement} can be strict, since here
$[H_0]_{\mathrm{T0}} = [H_1]_{\mathrm{T0}} = \mathbf{0}'$.
\end{exm}

\begin{warn}\label{warning-computably-separable}
To avoid possible confusion, we emphasize that the assertion that the
promise problem $(P,Q)$ is decidable is different from the
stronger statement that $P$ and $Q$ are \textbf{computably separable},
which means that there is a \emph{total} computable function taking the
value $0$ on $P$ and $1$ on $Q$.

A counterexample to the equivalence is provided by the sets $H_0,H_1$
of \cref{example-disjoint-c-e-sets}, which are well-known to be
computably inseparable (\cite[exercise 1.6.26]{Soare}), yet we have
just pointed out that the promise problem $(H_0,H_1)$ is decidable.

\thingy\label{digression-separating-degree}
\textit{A digression.} While the T1 degree $\mathbf{d}_{P,Q}$
expresses the extent to which the promise problem is decidable,
there is also a degree naturally associated to
the question of whether $P,Q$ are computably separable, which we might
call $\mathbf{s}_{P,Q}$; this time, it is a T2 degree.  It is defined by the
function $\mathbb{N} \dasharrow \mathcal{P}(\mathbb{N})$ taking the
value $\{0\}$ on $P$, the value $\{1\}$ on $Q$, and $\{0,1\}$
elsewhere.  The special case $\mathbf{s}_{H_0,H_1}$ has been called
$\mathtt{LLPO}$ in \cite[example 2.5]{KiharaLT}, and we have mentioned
it in the introduction (¶\ref{intro-context}), in connection with the
“PA degrees”: it is $>\mathbf{0}$ because $H_0,H_1$ are computably
inseparable, but also $<\mathbf{0}'$ because the low basis theorem
(\cite[theorem 3.7.2]{Soare}) guarantees that there are Turing degrees
$\geq\mathbf{s}_{H_0,H_1}$ which are $<\mathbf{0}'$.  So, unlike
$\mathbf{d}_{H_0,H_1}$, which is $\mathbf{0}$ as have just said, the
T2 degree $\mathbf{s}_{H_0,H_1}$ is strictly intermediate between
$\mathbf{0}$ and $\mathbf{0}'$.  In general, we have $\mathbf{d}_{P,Q}
\leq \mathbf{s}_{P,Q} \leq [P]_\TZero \meetDTwo [Q]_\TZero$, and
as we have just said both of these inequalities can be strict, but we
know little else to say about the relation between these degrees.
\end{warn}

\begin{rmk}\label{degenerate-cases-of-decoding-and-coding-degrees}
We could define $\mathbf{d}_{P,Q}$ without the assumption that
$P$ and $Q$ are inhabited, but it is trivial (i.e., computable) if
$P=\varnothing$ or $Q=\varnothing$, and we have chosen to exclude this
case to avoid annoying case distinctions.  One might also attempt to
define $\mathbf{d}_{P,Q}$ without the assumption that $P$ and $Q$ are
disjoint, but there are several plausible candidates here; the most
logical one in our context (if we are to keep the
property \ref{coding-and-decoding-give-omniscience}) would
give the T3 degree $\Om$ if $P\cap Q$ is inhabited: this is not
particularly useful, so we do not do this.

Symmetrically, we could define $\mathbf{c}_{P,Q}$ without the
assumption that $P$ and $Q$ are inhabited, but it the top degree if
$P=\varnothing$ or $Q=\varnothing$; we could also define it without
the assumption that $P$ and $Q$ are disjoint, but when $P\cap Q$ is
inhabited it is trivial (i.e., computable) because Merlin can always
play an element of $P\cap Q$, making the oracle useless.
\end{rmk}

\bigskip

Let us now turn our attention to many-to-one reduction of promise
problems, which, as the next two propositions show, is an easy way to
obtain inequalities among the $\mathbf{d}_{P,Q}$ and
$\mathbf{c}_{P,Q}$ degrees:

\begin{dfn}\label{definition-many-to-one-reduction}
Let $P_1,Q_1 \subseteq \mathbb{N}$ be inhabited and disjoint, and
similarly for $P_2,Q_2 \subseteq \mathbb{N}$.  We say that the promise
problem $(P_2,Q_2)$ is \textbf{many-to-one reducible} to the promise
problem $(P_1,Q_1)$ and we write $(P_2,Q_2)
\mathrel{\preceq_{\mathrm{m}}} (P_1,Q_1)$, when there exists
$f\colon\mathbb{N}\dasharrow\mathbb{N}$ partial computable, defined at
least on $P_2\cup Q_2$, and such that $f(P_2) \subseteq P_1$ and
$f(Q_2) \subseteq Q_1$.
\end{dfn}

\begin{prop}\label{many-to-one-reduction-implies-decoding-oracles-reduction}
The reducibility $(P_2,Q_2) \mathrel{\preceq_{\mathrm{m}}} (P_1,Q_1)$
implies $\mathbf{d}_{P_2,Q_2} \leq \mathbf{d}_{P_1,Q_1}$ (as
T1 degrees).
\end{prop}
\begin{proof}
Given $f\colon\mathbb{N}\dasharrow\mathbb{N}$ partial computable,
defined at least on $P_2\cup Q_2$, such that $f(P_2) \subseteq P_1$
and $f(Q_2) \subseteq Q_1$, to compute the decoding oracle associated
to $(P_2,Q_2)$ from that associated to $(P_1,Q_1)$, we simply apply
the latter on the image by $f$ of the given element in $P_2\cup Q_2$.
\end{proof}

\begin{prop}\label{many-to-one-reduction-implies-coding-oracles-reduction}
The reduction $(P_2,Q_2) \mathrel{\preceq_{\mathrm{m}}} (P_1,Q_1)$
implies $\mathbf{c}_{P_1,Q_1} \leq \mathbf{c}_{P_2,Q_2}$ (as
T3 degrees).
\end{prop}
\begin{proof}
Given $f\colon\mathbb{N}\dasharrow\mathbb{N}$ partial computable,
defined at least on $P_2\cup Q_2$, such that $f(P_2) \subseteq P_1$
and $f(Q_2) \subseteq Q_1$, to compute the coding oracle associated to
$(P_1,Q_1)$ from that associated to $(P_2,Q_2)$, we simply apply $f$
to the element in $P_2\cup Q_2$ returned by $(P_2,Q_2)$ coding oracle.
\end{proof}

\bigskip

We now investigate the relationship between the promise problems
$(P,Q)$ and $(P\times Q, \penalty0\, Q\times P)$.  The reason we care about
$(P\times Q, \, Q\times P)$ is that it encapsulates the essence of the
trick used to answer the first riddle in ¶\ref{intro-riddles}: “ask
the question twice, once positively, and then negatively”.  So we wish
to know how $\mathbf{d}_{P\times Q, \, Q\times P}$ relates to
$\mathbf{d}_{P,Q}$, and $\mathbf{c}_{P\times Q, \, Q\times P}$ to
$\mathbf{c}_{P,Q}$.  Concerning the first, we have a very easy
inequality:

\begin{prop}\label{checker-decoding-is-easier-than-plain-decoding}
We have $\mathbf{d}_{P\times Q, \, Q\times P} \leq \mathbf{d}_{P,Q}$.
\end{prop}
\begin{proof}
This follows from
\cref{many-to-one-reduction-implies-decoding-oracles-reduction}
since the first projection witnesses $(P,Q)
\mathrel{\preceq_{\mathrm{m}}} (P\times Q, \penalty0\, Q\times P)$.
\end{proof}

\begin{exm}\label{easy-example-decoding-checker-inequality-can-be-strict}
This above inequality can be strict.  To illustrate this, let $P$ be a
computably enumerable set, and $Q$ be its complement.  Then
$\mathbf{d}_{P,Q}$ is the Turing degree of $P$ (as we have already
noted in \cref{example-set-and-complement}), which can certainly
be $>\mathbf{0}$.  However, $\mathbf{d}_{P\times Q, \, Q\times P} =
\mathbf{0}$: indeed, to decide whether a pair promised to be in
$(P\times Q) \,\cup\, (Q\times P)$ is in $P\times Q$ or $Q\times P$, we
just enumerate $P$ until we find the first or the second coordinate of
our pair.

The reasoning of the previous sentence applies, more generally, when
$P$ is a computably enumerable set, and $Q$ is disjoint from $P$
(i.e., an arbitrary \emph{subset} of its complement): so we have
$\mathbf{d}_{P\times Q, \, Q\times P} = \mathbf{0}$ in all such cases
(although little can then be said about $\mathbf{d}_{P,Q}$: compare
with \cref{example-disjoint-c-e-sets}).
\end{exm}

\begin{exm}\label{example-dedekind-cuts}
A different example is also worth giving.  Let $P$ and
$Q$ be respectively the lower and upper Dedekind cut of a real number
$x$, i.e., $P := \{r \in \mathbb{Q} \mid r < x\}$ and
$Q := \{r \in \mathbb{Q} \mid r > x\}$ (after having chosen a fixed
Gödel encoding of rationals by natural numbers). Then $\mathbf{d}_{P, Q}$
is the Turing degree of $x$, which can be arbitrary, while
$\mathbf{d}_{P \times Q, \, Q \times P}$ is decidable simply by comparing
the two rationals in the pair.  This example shows that we cannot even
hope to bound $\mathbf{d}_{P, Q}$ as a function of $\mathbf{d}_{P
  \times Q, \, Q \times P}$.
\end{exm}

The relation of $\mathbf{c}_{P\times Q, \, Q\times P}$ to
$\mathbf{c}_{P,Q}$ is easier — they are equal:

\begin{prop}\label{checker-coding-is-same-as-plain-coding}
\textbf{(i)} We have $\mathbf{c}_{P\times Q, \, Q\times P} = \mathbf{c}_{P,Q}$.

\textbf{(ii)} We have $\mathbf{d}_{P\times Q, \, Q\times P} \vee
\mathbf{c}_{P,Q} = \Om$.  (In particular, if $\mathbf{d}_{P\times Q,
  \, Q\times P} = \mathbf{0}$ then $\mathbf{c}_{P,Q} = \Om$.)
\end{prop}
\begin{proof}
Concerning (i), the inequality $\mathbf{c}_{P\times Q, \, Q\times P}
\geq \mathbf{c}_{P,Q}$ follows from
\cref{many-to-one-reduction-implies-coding-oracles-reduction} since the
first projection witnesses $(P,Q) \mathrel{\preceq_{\mathrm{m}}}
(P\times Q, \penalty0\, Q\times P)$.

For the inequality $\mathbf{c}_{P\times Q, \, Q\times P} \leq
\mathbf{c}_{P,Q}$, we apply the essence of the trick used to answer
the first riddle in ¶\ref{intro-riddles}: namely, Arthur's strategy is
to query the $\mathbf{c}_{P,Q}$ oracle twice, Nimue knows this and
plays either $P$ then $Q$ or $Q$ then $P$ according as the initial
configuration requires Arthur to produce an element of $P\times Q$ or
$Q\times P$.  Arthur can then simply collect the two values he
received into a pair and declare this.

Finally, (ii) follows trivially from (i) and
\cref{coding-and-decoding-give-omniscience} (but applied to $(P\times
Q, \, Q\times P)$).
\end{proof}

\bigskip

We now arrive at our main technical result: we have just seen that
$\mathbf{d}_{P\times Q, \, Q\times P} = \mathbf{0}$ implies
$\mathbf{c}_{P,Q} = \Om$; in fact, this is an equivalence:

\begin{thm}\label{main-strategy-result}
\textit{(Plain version.)}  Let $P,Q \subseteq \mathbb{N}$ be inhabited
and disjoint.  If the coding oracle associated to $(P,Q)$ is
omniscient, then the promise problem $(P \times Q, \, Q \times P)$
is decidable. In other words, if $\mathbf{c}_{P,Q} = \Om$ then
$\mathbf{d}_{P\times Q, \, Q\times P} = \mathbf{0}$ (as we have just
said, this is an equivalence).

\textit{(Relativized version.)} Let $P, Q \subseteq \N$ be inhabited
and disjoint, and let $\mathbf{a}$ be a T0 degree. If
$\mathbf{c}_{P, Q} \vee \mathbf{a} = \Om$ then
$\mathbf{d}_{P \times Q, \, Q \times P} \leq \mathbf{a}$.
\end{thm}

\begin{proof}
Let us first prove the “plain” version.

For the entirety of the proof, we fix elements $p \in P$ and $q \in Q$,
which exist by assumption.

We assume that $\mathbf{c}_{P, Q}$ is omniscient, which means that
there exists a winning Arthur+Nimue strategy in the game reducing $\Om$
to $\mathbf{c}_{P, Q}$ (recall that we described the special
case of reductions from $\Om$ in \ref{reduction-from-om}).
We fix such a strategy and call $\sigma$ Arthur's part of the strategy;
thus, $\sigma$ takes a finite sequence of natural numbers representing
moves by Merlin to a move by Arthur, which can be “query the oracle
once more” or “declare a final value $u$”.

Given an element $(x,y) \in (P\times Q) \,\cup\, (Q\times P)$, we consider
the function $M_{(x,y)} \colon \{\Psymb,\Qsymb\}^2 \to
\mathbb{N}$ (where $\Psymb,\Qsymb$ are two
labels\footnote{If the reader has difficulty distinguishing
 $\Psymb,\Qsymb$ from $P,Q$ typographically, we hasten
 to reassure that this will not cause any confusion.})
defined explicitly by:
\[
\begin{aligned}
M_{(x,y)}(\Psymb,\Psymb) &= p \\
M_{(x,y)}(\Psymb,\Qsymb) &= x \\
M_{(x,y)}(\Qsymb,\Psymb) &= y \\
M_{(x,y)}(\Qsymb,\Qsymb) &= q \\
\end{aligned}
\]

This function is designed to ensure the following property (\textdagger):
if $(x,y) \in P\times Q$, then $M_{(x,y)}(\Psymb, r) \in P$ and
$M_{(x,y)}(\Qsymb, r) \in Q$ for any $r \in
\{\Psymb,\Qsymb\}$, and symmetrically, if $(x,y) \in Q\times P$, then
$M_{(x,y)}(\ell, \Psymb) \in P$ and $M_{(x,y)}(\ell, \Qsymb)
\in Q$ for any $\ell \in \{\Psymb,\Qsymb\}$.

We define an ancillary game, which is a two-player imperfect
information game parameterized by $(x, y) \in (P \times Q) \,\cup\, (Q \times P)$.
The two players are “Left” and “Right”. On the $i$-th turn,
Left chooses a value $\ell_i \in \{\Psymb, \Qsymb\}$, and
Right simultaneously chooses $r_i \in \{\Psymb, \Qsymb\}$. (Each player
has seen the other's moves up to this turn, but the two moves on this
turn are chosen independently.) What happens next depends on the
value of Arthur's strategy $\sigma$ on the sequence
$M_{(x, y)}(\ell_0, r_0), \ldots, M_{(x, y)}(\ell_i, r_i)$.
If $\sigma$ returns “query the oracle”, then the ancillary game continues.
If $\sigma$ returns “declare a final value $u$” where $u \in \{0, 1\}$,
then Left wins if $u = 0$, whereas Right wins if $u = 1$. If $\sigma$
fails to return or returns a malformed value, or if the game lasts
infinitely long, then the result is a draw (no player wins).

Assume temporarily that $(x, y) \in P \times Q$.

We observe that Left has a winning strategy in the ancillary game.
Indeed, by (\textdagger), Left can control whether the successive
elements provided to $\sigma$ are in $P$ or $Q$, and using Nimue's
part of the Arthur+Nimue winning strategy, this enables Left to
make Arthur's strategy output 0.

Let $L$ be a winning strategy for Left. This is a function from finite
sequences in $\{\Psymb, \Qsymb\}$, representing past moves by Right,
to elements of $\{\Psymb, \Qsymb\}$, representing the next move by Left.
For every sequence $s \in \{\Psymb, \Qsymb\}^\N$ of moves by Right, the
play where Left plays according to $L$ and Right plays the successive
elements of $s$ must end in a victory for Left after a finite number $m(s)$
of moves. The function $m : \{\Psymb, \Qsymb\}^\N \to \N$ is
continuous (where $\{\Psymb, \Qsymb\}^\N$ is topologized by the product
of the discrete topology)
because $m(s)$ only depends on the first $m(s)$ terms of $s$.
By compactness of the Cantor space $\{\Psymb, \Qsymb\}^\N$, it
must be bounded by some constant $M$. This means that a finite fragment
of $L$ — its restriction to finite sequences in $\{\Psymb, \Qsymb\}$
of length at most $M$ — already suffices to witness that Left has
a winning strategy.

We have proved that for all $(x, y) \in P \times Q$, there exists
a finite fragment of strategy, namely a function only defined on
finite sequences in $\{\Psymb, \Qsymb\}$ up to a certain length $M$,
which certifies that Left can win, in the sense that running the
game up to the $M$-th turn using this strategy fragment for the moves
of Left and any sequence of length $M$ in $\{\Psymb, \Qsymb\}$
for the moves of Right results in the game already ending with a
victory for Left.

Symmetrically, for $(x, y) \in Q \times P$, there exists a finite
strategy fragment which certifies that Right can win. Finally,
the two conditions are mutually exclusive since the players cannot
both win at the same time. This shows that $(x, y) \in P \times Q$
if and only if Left has a winning finite strategy fragment and
$(x, y) \in Q \times P$ if and only if Right has a winning finite strategy
fragment.

We deduce an algorithm to decide, given $(x, y) \in (P \times Q)
\,\cup\, (Q \times P)$, whether $(x, y) \in P \times Q$
or $(x, y) \in Q \times P$: enumerate finite strategy fragments in parallel
to search at the same time for a winning strategy fragment for Left
and a winning strategy fragment for Right. Whether a strategy
fragment is winning for a given player can be semi-decided by
simulating the plays with this fragment for the given player and
all possible sequences up to the length of the fragment for the other
player.

This concludes the proof of the “plain” version of the theorem.

\medskip

To prove the “relativized” version, we make the following observation:
the proof of the “plain” version which we have just given applies just
as well if we replace computability by computability relative to an
ordinary Turing\footnote{We emphasize here that this would not hold
for a T1 oracle because we use the ability to run several computations
in parallel.  See also
question \ref{question-parallelizable-t1-degrees}.} oracle $A
\subseteq \mathbb{N}$.  This gives us the following
result (\textsection): if Arthur+Nimue have a winning strategy in the
game reducing $\Om$ to $\mathbf{c}_{P,Q}$, but where now Arthur is
allowed to access the oracle $A$ in computing his strategy, then there
is an $A$-computable way to solve to the promise problem $(P\times Q,
\, Q\times P)$.

But the hypothesis of (\textsection) is equivalent to saying that
Arthur+Nimue have a winning strategy in the game reducing $\Om$ to
$[A]_{\mathrm{T0}} \vee \mathbf{c}_{P,Q}$ (because an
$A$-computable strategy for Arthur is the same thing as a strategy for
Arthur that can make the move of querying the $A$-oracle), and
the conclusion of (\textsection) is equivalent to saying that Arthur
has a winning strategy in the game reducing the decoding oracle
associated to $(P\times Q, \, Q\times P)$ to $A$.  So we have proved
the relativized version as claimed, and this concludes our entire
proof.
\end{proof}

\thingy\label{answer-to-riddle} This already provides an answer to the
second riddle asked in \ref{intro-riddles} in the introduction.
Indeed, let $K := \{e \mid \varphi_e(0){\downarrow}\}$ be the halting set
and $\bar K := \mathbb{N} \setminus K$ its complement; and let $P :=
(K\times K) \,\cup\, (\bar K \times \bar K)$ be the set of pairs of Turing
machines of which an even number halt and $Q := (K\times \bar K) \,\cup\,
(\bar K \times K)$ be the set of pairs of Turing machines of which
exactly one halts.  The riddle asks whether $\mathbf{c}_{P,Q} = \Om$.
Now if such were the case, then by \cref{main-strategy-result} we would
have $\mathbf{d}_{P\times Q, \, Q\times P} = \mathbf{0}$ (in other
words, given four Turing machines of which an odd number halt, we can
compute the parity of how many halt among the first two); but this
allows us to (computably) decide the halting of an arbitrary $e$ by decoding
$((\mathtt{loop},e), (\mathtt{halt},e))$ where $\mathtt{loop}$ is a
fixed Turing machine that never halts and $\mathtt{halt}$ is one that
halts immediately: indeed, if $e\in K$ then this is in $Q\times P$
whereas if $e\in \bar K$ this is in $P\times Q$.  But this is
impossible: so in fact $\mathbf{c}_{P,Q} < \Om$, solving the riddle.

\section{Co-Turing degrees and their order}\label{section-co-turing-degrees}

In the previous section, we have defined “decoding” and “coding”
oracles for a promise problem $(P,Q)$.  The relation between the two,
however, is not as straightforward as one might hope (we will remark
in \ref{remark-that-c-p-q-is-not-a-function-of-d-p-q} below that the
degree $\mathbf{c}_{P,Q}$ is not a function of the degree
$\mathbf{d}_{P,Q}$).  We now turn to a particular case of the
$\mathbf{d}_{P,Q}$ and $\mathbf{c}_{P,Q}$, which we will denote by
$\tilde{\mathbf{d}}_A$ and $\tilde{\mathbf{c}}_A$ respectively,
defined by a set $A \subseteq \mathbb{N}$, which will behave better in
this respect: the degree $\tilde{\mathbf{d}}_A$ will turn out to
be simply the Turing degree of $A$, while $\tilde{\mathbf{c}}_A$ will
turn out to depend only on the Turing degree $[A]_\TZero$ of $A$ and this will
provide our sought-after “co-Turing” degrees.

\begin{dfn}\label{definition-codegrees}
For $A \subseteq \mathbb{N}$ arbitrary and $i\in\{0,1\}$, we define
$H_i^A := \{e \mid \varphi_e^A(0){\downarrow} = i\}$ to be the set of
oracle programs $e$ which, when provided $A$ (that is, $1_A$) as
oracle, terminate and return the value $i$.  (Note that $H_0^A, H_1^A$
are inhabited and disjoint.)

We also define $\tilde{\mathbf{d}}_A := \mathbf{d}_{H_0^A,\, H_1^A}$
and $\tilde{\mathbf{c}}_A := \mathbf{c}_{H_0^A,\, H_1^A}$.  We will
call $\tilde{\mathbf{c}}_A$ the \textbf{co-Turing degree} associated
to $A$ (we will see in \ref{reverse-embedding-theorem} that it depends
only on the Turing degree of $A$).
\end{dfn}

\thingy Since it is the central contruction of this paper, we spell
out the following explicit description: the $\tilde{\mathbf{c}}_A$ oracle is
one which can answer an arbitrary “yes”/“no” question, but, instead of
providing a direct answer $0$ or $1$, it (adversarially) encodes its
answer in the form of a program $e$ which, when run with $A$ as
(ordinary Turing) oracle, terminates and computes the desired
answer $\varphi_e^A(0)$.

\bigskip

The deeper significance of the $H_i^A$ construction will be
investigated around \cref{construction-j-assembly-for-turing-degree}
below, but the immediate point is that is it turns a Turing-reduction on the
oracle $A$ to a many-to-one reduction on the $H_i^A$.  Indeed, the following
fact is standard (cf. \cite[theorem 3.4.3(v)]{Soare}) and can also be
seen as a special case
of \cref{statement-oracle-topology-correspondence} below, but let us
spell it out:

\begin{prop}\label{turing-reduction-implies-many-to-one-reduction-on-halting-sets}
For $A,B\subseteq\mathbb{N}$, if $B$ is Turing-reducible to $A$, then
$(H_0^B,H_1^B) \mathrel{\preceq_{\mathrm{m}}} (H_0^A,H_1^A)$.
\end{prop}
\begin{proof}
Suppose $1_B = \varphi_h^A$ witnesses that $B$ is Turing-reducible
to $A$.  Given $e\in H_0^B \cup H_1^B$, we computably convert it to a
program $e'$ by replacing every call to the oracle (supposed to yield
the characteristic function of $B$) by a call to $h$: then
$\varphi_e^B = \varphi_{e'}^A$ as partial functions, in particular
$e'\in H_i^A$ if $e\in H_i^B$, and $e\mapsto e'$, which is computable
(even primitive recursive) witnesses the claimed many-to-one reduction.
\end{proof}

\begin{prop}\label{turing-reduction-implies-codegree-reduction}
If $A,B\subseteq\mathbb{N}$ are such that $B$ is Turing-reducible
to $A$, then $\tilde{\mathbf{c}}_A \leq \tilde{\mathbf{c}}_B$ (as
T3 degrees).
\end{prop}
\begin{proof}
This follows trivially from
\cref{turing-reduction-implies-many-to-one-reduction-on-halting-sets} and \cref{many-to-one-reduction-implies-coding-oracles-reduction}.
\end{proof}

Of course, under the same assumption, we also have
$\tilde{\mathbf{d}}_B \leq \tilde{\mathbf{d}}_A$ by
\cref{turing-reduction-implies-many-to-one-reduction-on-halting-sets} and \cref{many-to-one-reduction-implies-decoding-oracles-reduction},
but in fact we can say more:

\begin{prop}\label{decoding-relative-halting-is-turing-degree}
For $A \subseteq \mathbb{N}$ arbitrary, the (a priori T1) degree
$\tilde{\mathbf{d}}_A$ is the Turing degree $[A]_{\mathrm{T0}}$
of $A$.
\end{prop}
\begin{proof}
To see that $\tilde{\mathbf{d}}_A \leq [A]_{\mathrm{T0}}$, we note
that given $e \in H_0^A \cup H_1^A$, with the help of $A$ (that is,
$1_A$) as oracle, we can simply compute $\varphi_e^A$ and return its
value, which is exactly that of the required decoding oracle.

Conversely, to see that $[A]_{\mathrm{T0}} \leq \tilde{\mathbf{d}}_A$,
given $k\in \N$ we can compute whether $k\in A$ by first computing (by
means of the relative s-m-n theorem) the program $e$ that queries
the oracle on the value $k$ and then returns its answer, and then
asking the decoding oracle whether this $e$ is in $H_0^A$ or $H_1^A$.
\end{proof}

Recall that in the previous section we have seen the importance of
considering $\mathbf{d}_{P\times Q, \, Q\times P}$ (and not just
$\mathbf{d}_{P,Q}$) to understand $\mathbf{c}_{P,Q}$, and recall that
we have seen that $\mathbf{d}_{P\times Q, \, Q\times P} \leq
\mathbf{d}_{P,Q}$ holds in general
(\cref{checker-decoding-is-easier-than-plain-decoding}) but that this
inequality can be strict
(\cref{easy-example-decoding-checker-inequality-can-be-strict}
or \ref{example-dedekind-cuts}).  In the present situation, however,
this is an equality (it is also the Turing degree of $A$):

\begin{prop}\label{checker-decoding-relative-halting-is-plain-decoding}
For $A \subseteq \mathbb{N}$ arbitrary, the degree $\mathbf{d}_{(H_0^A
  \times H_1^A), \; (H_1^A \times H_0^A)}$ equals $\mathbf{d}_{H_0^A,
  \, H_1^A} =: \tilde{\mathbf{d}}_A$.
\end{prop}
\begin{proof}
The inequality $\mathbf{d}_{(H_0^A \times H_1^A), \; (H_1^A \times
  H_0^A)} \leq \mathbf{d}_{H_0^A, \, H_1^A}$ has already been seen in
full generality
in \ref{checker-decoding-is-easier-than-plain-decoding}.  So we need
to see the converse.

But given $e \in H_i^A$ we can computably produce $\mathtt{flip}\circ
e \in H_{1-i}^A$ by simply composing $e$ with the function that
exchanges $0$ and $1$.  To decide by means of $\mathbf{d}_{(H_0^A
  \times H_1^A), \; (H_1^A \times H_0^A)}$ whether $e \in H_0^A \cup
H_1^A$ belongs to $H_0^A$ or $H_1^A$, we ask whether $(e, \,
\mathtt{flip}\circ e)$ belongs in $H_0^A \times H_1^A$ or $H_1^A
\times H_0^A$.  This proves $\mathbf{d}_{H_0^A, \, H_1^A} \leq
\mathbf{d}_{(H_0^A \times H_1^A), \; (H_1^A \times H_0^A)}$ (formally,
by applying
\cref{many-to-one-reduction-implies-decoding-oracles-reduction} to $e
\mapsto (e, \, \mathtt{flip}\circ e)$).
\end{proof}

This allows us to obtain the converse
of \cref{turing-reduction-implies-codegree-reduction} from our main
technical result:

\begin{prop}\label{codegree-reduction-implies-turing-reduction}
If $A,B\subseteq\mathbb{N}$ are such that $\tilde{\mathbf{c}}_A \leq
\tilde{\mathbf{c}}_B$ (as T3 degrees), then $B$ is Turing-reducible
to $A$.
\end{prop}
\begin{proof}
By \cref{coding-and-decoding-give-omniscience}, we have
$\tilde{\mathbf{d}}_A \vee \tilde{\mathbf{c}}_A = \Om$, so our
assumption yields $\tilde{\mathbf{d}}_A \vee \tilde{\mathbf{c}}_B = \Om$.
Now by \cref{decoding-relative-halting-is-turing-degree},
$\tilde{\mathbf{d}}_A$ is the Turing degree $\mathbf{a}$ of $A$: thus,
$\mathbf{a} \vee \tilde{\mathbf{c}}_B = \Om$.
By \cref{main-strategy-result}, this implies $\mathbf{d}_{(H_0^B \times
  H_1^B), \; (H_1^B \times H_0^B)} \leq \mathbf{a}$.  By
\cref{checker-decoding-relative-halting-is-plain-decoding} and
\cref{decoding-relative-halting-is-turing-degree} (again),
$\mathbf{d}_{(H_0^B \times H_1^B), \; (H_1^B \times H_0^B)}$ is the
Turing degree $\mathbf{b}$ of $B$.  So we have proved $\mathbf{b} \leq
\mathbf{a}$, as claimed.
\end{proof}

Putting all this together, we have shown:

\begin{thm}\label{reverse-embedding-theorem}
For $A,B\subseteq\mathbb{N}$, we have $\tilde{\mathbf{c}}_A \leq
\tilde{\mathbf{c}}_B$ iff $[B]_{\mathrm{T0}}
\leq[A]_{\mathrm{T0}}$ (equivalently: $\tilde{\mathbf{d}}_B \leq
\tilde{\mathbf{d}}_A$).

Thus, the mapping $\coT\colon [A]_{\mathrm{T0}} \mapsto
\tilde{\mathbf{c}}_A$ is a well-defined injection from
$\mathscr{D}_{\mathrm{T0}}^\op$ (where
$(\emdash)^\op$ denotes the opposite order) to
$\mathscr{D}_{\mathrm{T3}}$, which preserves and reflects the order.
\end{thm}
\begin{proof}
The first claim is the conjunction of propositions
\ref{turing-reduction-implies-codegree-reduction} and \ref{codegree-reduction-implies-turing-reduction};
the parenthetical remark
is \cref{decoding-relative-halting-is-turing-degree}.  The second
claim follows immediately from the first and from definitions.
\end{proof}

\begin{rmk}\label{remark-that-c-p-q-is-not-a-function-of-d-p-q}
Returning to our example started in \ref{example-dedekind-cuts}, if
$x \in \mathbb{R}$ is uncomputable and $P$ and $Q$ denote its lower $\{r \in
\mathbb{Q} \mid r<x\}$ and upper $\{r \in \mathbb{Q} \mid r>x\}$ Dedekind
cut respectively, then $\mathbf{d}_{P,Q} = [x]_{\mathrm{T0}}$ is the
Turing degree of $x$, which is also (after identifying $x$ with a
subset of $\mathbb{N}$, say the set of locations of $1$ in its binary
expansion) equal to $\tilde{\mathbf{d}}_x := \mathbf{d}_{H_0^x,\,
  H_1^x}$ by \cref{decoding-relative-halting-is-turing-degree}; but
$\mathbf{c}_{P,Q} = \Om$ follows from
\cref{checker-coding-is-same-as-plain-coding}(ii)
and the fact pointed out
in \cref{example-dedekind-cuts} that $\mathbf{d}_{P\times Q, \,
  Q\times P} = \mathbf{0}$, whereas
$\tilde{\mathbf{c}}_x := \mathbf{c}_{H_0^x,\, H_1^x} < \Om$
by \cref{reverse-embedding-theorem}.  This example shows that the
degree $\mathbf{c}_{P,Q}$ is \emph{not} a function of the degree
$\mathbf{d}_{P,Q}$ in general, even though the degree $\tilde{\mathbf{c}}_A$
\emph{is} a function of $\tilde{\mathbf{d}}_A$.
\end{rmk}

\begin{rmk}\label{remark-extension-to-t1-degrees}\textit{(Extension to T1 degrees.)}
If $f\colon\mathbb{N}\dasharrow\mathbb{N}$ is a partial function, we
can define $H_0^f,H_1^f$ in the obvious way (using, of course, the
convention that a program that calls the oracle at an undefined value
fails to terminate); then the analogue
of \cref{turing-reduction-implies-many-to-one-reduction-on-halting-sets}
still holds: namely, $[g]_{\mathrm{T1}} \leq [f]_{\mathrm{T1}}$
implies $(H_0^g,H_1^g) \mathrel{\preceq_{\mathrm{m}}} (H_0^f,H_1^f)$
(the only change to the proof is that $\varphi_{e'}^f$ might now
extend $\varphi_e^g$ but this doesn't affect anything), so the
analogue of \cref{turing-reduction-implies-codegree-reduction} also
holds, and by defining $\tilde{\mathbf{c}}_f := \mathbf{c}_{H_0^f,\,
  H_1^f}$ analogously to \cref{definition-codegrees} we get an
order-reversing map $\coT_1 \colon [f]_{\mathrm{T1}} \mapsto
\tilde{\mathbf{c}}_f$ that extends $\coT$
to $\mathscr{D}_{\mathrm{T1}}$.

We can also define $\tilde{\mathbf{d}}_f := \mathbf{d}_{H_0^f,\,
  H_1^f}$ in this context and we still have the analogues of
\cref{decoding-relative-halting-is-turing-degree} and \cref{checker-decoding-relative-halting-is-plain-decoding}
with the same proofs (so $\tilde{\mathbf{d}}_f$ is just the T1
degree $[f]_\TOne$ of $f$).  We do not know, however, how to obtain the
analogue of \cref{codegree-reduction-implies-turing-reduction} in this
context because it relies on the relative version of \cref{main-strategy-result}, whose proof
depends on $\mathbf{a}$ being a Turing degree.  Thus we only get the
weaker conclusion that $\coT(\mathbf{a}) \leq \coT_1(\mathbf{g})$
implies $\mathbf{g}\leq\mathbf{a}$ when $\mathbf{a}$ is a T0 degree
while $\mathbf{g}$ can be a T1 degree.  (Still, this lets us conclude
that $\coT_1(\mathbf{g}) = \Om$ implies $\mathbf{g} = \mathbf{0}$.)
See also \cref{question-codegrees-for-t1-degrees} below.
\end{rmk}

\begin{rmk}\label{remark-extension-to-t2-degrees}\textit{(Tentative extension to T2 degrees.)}
Continuing the line of thoughts of the previous remark, one might even
attempt to do something with T2 degrees as follows: if
$f\colon\mathbb{N}\dasharrow\mathcal{P}(\mathbb{N})$ is a T2 oracle,
we can define $H_0^f,H_1^f$ as $J_f(\{0\}),J_f(\{1\})$ respectively
(where $J_f$ is defined
in \ref{definition-oracle-topology-correspondence}), in other words,
the set of Arthur strategies in querying the $f$ oracle which,
\emph{whatever} the answers made by Merlin, terminate and return the
value $0$, resp. $1$.  Again the analogue
of \cref{turing-reduction-implies-many-to-one-reduction-on-halting-sets}
still holds: $[g]_{\mathrm{T2}} \leq [f]_{\mathrm{T2}}$ implies
$(H_0^g,H_1^g) \mathrel{\preceq_{\mathrm{m}}} (H_0^f,H_1^f)$ so the
analogue of \cref{turing-reduction-implies-codegree-reduction} also
holds, and by defining $\tilde{\mathbf{c}}_f := \mathbf{c}_{H_0^f,\,
  H_1^f}$ we get an order-reversing map $\coT_2 \colon
[f]_{\mathrm{T2}} \mapsto \tilde{\mathbf{c}}_f$ that extends $\coT$
(and $\coT_1$) to $\mathscr{D}_{\mathrm{T2}}$.  But now if we define
$\tilde{\mathbf{d}}_f := \mathbf{d}_{H_0^f,\, H_1^f}$ then we only get
$\tilde{\mathbf{d}}_f \leq [f]_{\mathrm{T2}}$, and this inequality can
be strict since $\tilde{\mathbf{d}}_f$ is a T1 degree but there are
T2 degrees $f$ which are not T1, in other words the analogue
of \cref{decoding-relative-halting-is-turing-degree} fails, making the
situation even less felicitous.
\end{rmk}

We will point out in \cref{example-codegree-for-error-one-third} below
that if we are even bolder and extend $\coT$ in the “obvious” way to a
map $\coT_3$ on the T3 degrees, then we can have $\coT_3(\mathbf{g}) =
\Om$ despite $\mathbf{g} > \mathbf{0}$.

\section{Comparing co-Turing degrees with Turing degrees}\label{section-comparison-of-codegree-with-degree}

In the previous section, we have defined for each
Turing degree $\mathbf{a} = [A]_{\mathrm{T0}}$ a “co-Turing”
degree $\coT(\mathbf{a}) = \tilde{\mathbf{c}}_A$ (in the T3 degrees),
and we have seen that $\coT(\mathbf{a}) \leq \coT(\mathbf{b})$ occurs
iff $\mathbf{a} \geq \mathbf{b}$.  So we understand how the co-Turing
degrees are ordered.  In this section, we would now like to compare a
co-Turing degree $\coT(\mathbf{a})$ with a Turing degree $\mathbf{b}$
(again, in the T3 degrees).

We know \textit{a priori} from \cite[prop. 3]{Phoa} that if
$\coT(\mathbf{a}) \geq \mathbf{b}$ for \emph{all} Turing
degrees $\mathbf{b}$ then in fact $\coT(\mathbf{a}) = \Om$ (so that
$\mathbf{a} = \mathbf{0}$ by \cref{reverse-embedding-theorem}); we will
proceed to show in this section that if $\coT(\mathbf{a}) \geq \mathbf{b}$ for
\emph{some} Turing degree $\mathbf{b} > \mathbf{0}$ then in fact
$\mathbf{a} = \mathbf{0}$.  (To put it differently, not only will
$\coT(\mathbf{a})$ for $\mathbf{a} > 0$ not give Arthur omniscience,
it won't even let him compute \emph{any} deterministic function
$\N\to\N$ that wasn't already computable.)

Our strategy for this is to consider the promise problem $(P,Q) :=
(H_0^A \times H_0^B, \, H_1^A \times H_1^B)$, and to compare its
coding and decoding degrees with those of $(H_0^A, \, H_1^A)$ (namely
$\tilde{\mathbf{c}}_A$ and $\tilde{\mathbf{d}}_A$) and of $(H_0^B, \,
H_1^B)$.

\begin{prop}\label{decoding-is-meet-for-two-sets}
For $A,B\subseteq\mathbb{N}$, if $P := H_0^A \times H_0^B$ and $Q :=
H_1^A \times H_1^B$, then $\mathbf{d}_{P,Q}$ equals the binary meet
(inf) $[A]_{\mathrm{T0}} \mathbin{\wedge_{\mathrm{T1}}}
[B]_{\mathrm{T0}}$, taken inside the T1 degrees
$\mathscr{D}_{\mathrm{T1}}$, of the Turing degrees of $A$ and $B$.
\end{prop}

(One may wish to compare with the constructions of the meet given in
\cref{meet-of-t1-degrees} and ¶\ref{algebraic-operations-on-l-t-topologies}.)

\begin{proof}
The inequality $\mathbf{d}_{P,Q} \leq [A]_{\mathrm{T0}}
\mathbin{\wedge_{\mathrm{T1}}} [B]_{\mathrm{T0}}$ follows from
$\mathbf{d}_{P,Q} \leq \mathbf{d}_{H_0^A,\,H_1^A}
=[A]_{\mathrm{T0}}$ and $\mathbf{d}_{P,Q} \leq
\mathbf{d}_{H_0^B,\,H_1^B} =[B]_{\mathrm{T0}}$
(the inequalities follow from projection
and \cref{many-to-one-reduction-implies-decoding-oracles-reduction},
and the equalities are given
in \cref{decoding-relative-halting-is-turing-degree}).

The inequality $\mathbf{d}_{P,Q} \geq [A]_{\mathrm{T0}}
\mathbin{\wedge_{\mathrm{T1}}} [B]_{\mathrm{T0}}$ follows from the
explicit description of $\mathbin{\wedge_{\mathrm{T1}}}$ in the
T1 degrees in \ref{meet-of-t1-degrees}: given $(p,q,m,n)$ such that
$\varphi_p^A(m){\downarrow} = \varphi_q^B(n)$, we loop through
all $N=0,1,2,\ldots$ and, for each one, computably produce an oracle
program $a_N$ which runs $p$ on $m$ and checks whether the result
is $N$, returning $1$ if true and $0$ otherwise, and $b_N$ analogously
for $q$ on $N$, then we use $\mathbf{d}_{P,Q}$ to test if $(a_N,b_N)$
is in $P = H_0^A \times H_0^B$ or in $Q = H_1^A \times H_1^B$ (note
that it must be in one of them because $\varphi_p^A(m){\downarrow} =
\varphi_q^B(n)$), and we continue the loop until it is found to be
in $Q$ at which point we return the corresponding $N$.
\end{proof}

Recall that we have seen that $\mathbf{d}_{P\times Q, \, Q\times P}
\leq \mathbf{d}_{P,Q}$ holds in general
(\cref{checker-decoding-is-easier-than-plain-decoding}) but that this
inequality can be strict
(\cref{easy-example-decoding-checker-inequality-can-be-strict}
or \ref{example-dedekind-cuts}).  In the specific case we are
interested in this section, it is, in fact,
just like
in \cref{checker-decoding-relative-halting-is-plain-decoding}, an
equality:

\begin{prop}\label{checker-decoding-for-two-sets}
For $A,B\subseteq\mathbb{N}$, if $P := H_0^A \times H_0^B$ and $Q :=
H_1^A \times H_1^B$, then the degree $\mathbf{d}_{P \times Q, \; Q
  \times P}$ equals $\mathbf{d}_{P,Q}$.
\end{prop}
\begin{proof}
The proof is essentially the same as
for \cref{checker-decoding-relative-halting-is-plain-decoding}:

The inequality $\mathbf{d}_{P \times Q, \; Q \times P} \leq
\mathbf{d}_{P,Q}$ has already been seen in full generality
in \cref{checker-decoding-is-easier-than-plain-decoding}.

For the converse, recall that given $e \in H_i^C$ (no matter what $C
\subseteq \mathbb{N}$ is) we can computably produce
$\mathtt{flip}\circ e \in H_{1-i}^C$ by simply composing $e$ with the
function that exchanges $0$ and $1$.  To decide by means of
$\mathbf{d}_{P \times Q, \; Q \times P}$ whether $(a,b) \in P \cup
Q$ belongs to $P$ or $Q$, we ask whether $((a,b), \,
(\mathtt{flip}\circ a, \, \mathtt{flip}\circ b))$ belongs in $P\times
Q$ or $Q\times P$.
\end{proof}

\begin{lem}\label{codegree-reduction-for-two-sets}
For $A,B\subseteq\mathbb{N}$, if $P := H_0^A \times H_0^B$ and $Q :=
H_1^A \times H_1^B$, then $\tilde{\mathbf{c}}_A \leq \mathbf{c}_{P,Q}$
\end{lem}
\begin{proof}
By \cref{many-to-one-reduction-implies-coding-oracles-reduction}, it is
enough to prove $(P,Q) \mathrel{\preceq_{\mathrm{m}}}
(H_0^A,\,H_1^A)$.  But the first projection witnesses this.
\end{proof}

\begin{thm}\label{comparison-of-codegree-with-degree}
\textbf{(i)} For $A,B\subseteq\mathbb{N}$, if $\tilde{\mathbf{c}}_A
\geq [B]_{\mathrm{T0}}$ then in fact one of $A,B$ is computable.

In other words: for all Turing degrees $\mathbf{a},\mathbf{b}$, if
$\coT(\mathbf{a}) \geq \mathbf{b}$ then $\mathbf{a}=\mathbf{0}$ or
$\mathbf{b}=\mathbf{0}$ (the converse being obvious).

\textbf{(ii)} For $A,B\subseteq\mathbb{N}$, it is never true that
$\tilde{\mathbf{c}}_A \leq [B]_{\mathrm{T0}}$.  In other words:
for all Turing degrees $\mathbf{a},\mathbf{b}$, we have
$\coT(\mathbf{a}) \not\leq \mathbf{b}$.
\end{thm}
\begin{proof}
We start with (i), which is the harder claim (and which depends on our
previous results).  Let $P := H_0^A \times H_0^B$ and $Q := H_1^A
\times H_1^B$.  If $\tilde{\mathbf{c}}_A \geq [B]_{\mathrm{T0}}$
then $\mathbf{c}_{P,Q} \geq [B]_{\mathrm{T0}}$
by \cref{codegree-reduction-for-two-sets}; now $[B]_{\mathrm{T0}}
\geq \mathbf{d}_{P,Q}$ by \cref{decoding-is-meet-for-two-sets}, so
$\mathbf{c}_{P,Q} \geq \mathbf{d}_{P,Q}$.
By \cref{coding-and-decoding-give-omniscience}, this implies
$\mathbf{c}_{P,Q} = \Om$.  By \cref{main-strategy-result}, this implies
$\mathbf{d}_{P\times Q, \, Q\times P} = \mathbf{0}$.
By \cref{checker-decoding-for-two-sets} this is just $\mathbf{d}_{P,Q}
= \mathbf{0}$.  By \cref{decoding-is-meet-for-two-sets} again, this
means $[A]_{\mathrm{T0}} \mathbin{\wedge_{\mathrm{T1}}}
[B]_{\mathrm{T0}} = \mathbf{0}$.
By \cref{computability-of-t1-meet}, this means that $A$ or $B$ is
computable, as claimed.

Let us now prove (ii), which is easier.  If $\tilde{\mathbf{c}}_A \leq
[B]_{\mathrm{T0}}$, then
by propositions \ref{coding-and-decoding-give-omniscience}
and \ref{decoding-relative-halting-is-turing-degree} we have
$[A]_{\mathrm{T0}} \vee [B]_{\mathrm{T0}} = \Om$, which
by \cref{join-of-degrees} means $[A \oplus B]_{\mathrm{T0}} = \Om$.
This is a contradiction as $\Om$ is not a Turing degree (we remarked
in ¶\ref{reduction-from-om} that $\Om$ is not T2).
\end{proof}

In light of the above, one might wonder if, perhaps, $\coT(\mathbf{a})
\mathbin{\wedge_{\mathrm{T3}}} \mathbf{b} = \mathbf{0}$ always holds
(since we have just seen that no Turing degree $>\mathbf{0}$ nor
co-Turing degree can be $\leq \coT(\mathbf{a})
\mathbin{\wedge_{\mathrm{T3}}} \mathbf{b}$).  In fact, this is 
essentially \emph{never} the case:

\begin{prop}\label{meet-of-degree-and-codegree}
If $\coT(\mathbf{a}) \mathbin{\wedge_{\mathrm{T3}}} \mathbf{b} =
\mathbf{0}$ then $\mathbf{b} = \mathbf{0}$ (in fact, $\mathbf{b}$ need
not even be assumed to be a Turing degree here: a T3 degree will do).
\end{prop}
\begin{proof}
This is just a special case of either one of
\cref{computability-of-t3-meet}
or \cref{computability-of-meet-of-basic-and-t3-degrees} (observing that
$\coT(\mathbf{a}) > \mathbf{0}$ always holds).
\end{proof}

\begin{rmk}
We have already mentioned in ¶\ref{digression-heyting-operation} that
the T3 degrees form not just a distributive lattice but in fact a
Heyting algebra.  Now there is one obvious order-reversing map which
comes to mind in a Heyting algebra, namely the negation map
$\mathbf{d} \mapsto {\neg_\TThree}\mathbf{d} := (\mathbf{d}
\mathbin{\Rightarrow_\TThree} \mathbf{0})$.  So the reader is
perhaps left to wonder whether our $\mathbf{d} \mapsto \coT(\mathbf{d})$ map
might be given, at least on the T0 degrees, by the negation map in the
Heyting algebra of T3 degrees (perhaps capped to $\Om$).

In fact, this is not the case: indeed, the meet $\mathbf{d}
\meetDThree \neg_\TThree \mathbf{d}$ of any $\mathbf{d}$ and its
negation is $\mathbf{0}$ by the very definition of negation; but we
have just seen in \ref{meet-of-degree-and-codegree} that $\mathbf{d}
\meetDThree \coT(\mathbf{d}) = \mathbf{0}$ is \emph{never} the case
except when already $\mathbf{d} = \mathbf{0}$.  In fact, the negation
operation is not interesting on the T0 degrees: it follows from our
remarks at the end of ¶\ref{digression-heyting-operation} that we have
$\neg_\TThree \mathbf{d} = \mathbf{0}$ for any $\mathbf{d} \in
\mathscr{D}_\TZero$.
\end{rmk}

\section{Topos-theoretic reformulation of some constructions}\label{section-topos-theory}

{\advance\leftskip by0.5\hsize\textit{„Die Mathematiker sind eine Art
    Franzosen: redet man zu ihnen, so übersetzen sie es in ihre
    Sprache, und dann ist es alsobald ganz etwas
    Anders.“}\penalty-500\quad — J. W. von Goethe\footnote{Quotation
  taken from \cite[§1279]{Goethe}. — English translation:
  “Mathematicians are like French people: when you talk to them they
  translate it into their own language, and then it soon turns into
  something completely different.”}\par}

\smallskip

We hasten to reassure the reader who does know anything about toposes
nor wish to that the present section is entirely optional.
Nevertheless, for those who enjoy abstract nonsense, we hope that the
following reformulation of some of our constructions will help
motivate them and shine some light on their significance.

We start with an informal explanation, then a precise definition, of
what Lawvere–Tierney topologies on the effective topos are, followed
by a summary statement of the main result by \cite{KiharaLT} and
\cite{KiharaOracles} regarding their correspondence with what we have
called the Arthur–Nimue–Merlin degrees (and how the T1 and T2 degrees
can be identified within this correspondence).  This subpart does not
assume prior familiarity with topos theory or realizability, and is
merely intended to recapitulate how the connection arises,
allowing us to rephrase \cref{reverse-embedding-theorem} as “we can embed
the Turing degrees backwards in the Lawvere–Tierney topologies on the
effective topos”.

However, from ¶\ref{specific-notations-for-topos-theoretic-section}
onwards, we do assume that the
reader is familiar with the theory of toposes in general, including
their “internal” logic, as well as the categorical point of view on
realizability — i.e., the effective topos — in particular.

Standard references for the former are: \cite{JohnstoneTopos}
(esp. chap. 3 and §5.4), \cite{LambekScott} (esp. §II.1–II.10),
\cite{MacLaneMoerdijk} (esp. chap. V \& VI), \cite{McLarty}
(esp. chap. 14 \& 21) and \cite{JohnstoneElephant} (esp. around
§§ A4.3–A4.4, D1 and D4).  For the latter: \cite{vanOostenBook} (but
see also \cite[chap. 24]{McLarty}).

\medskip

\thingy\label{informal-explanation-lawvere-tierney-topologies} Motivating the particular generalization of Turing oracles
described in \cref{t0-to-t3-oracles} to \cref{t0-to-t3-degrees} is its
connection with \textbf{Lawvere–Tierney topologies}.  Let us say a
word about this.

Lawvere–Tierney topologies go by various names (and can be seen in
various ways) according to the context being discussed: they are also
known as “\textbf{local operators}”.  In the context of elementary
toposes, they correspond to the subtoposes, i.e., embeddings of toposes, and
generalize the notion of Grothendieck topology for presheaf
toposes.  In the context of point-free topology (viꝫ. the study of locales), they
generally go by the name of “\textbf{nuclei}” on a frame/locale and
correspond to sublocales.  From a logician's perspective they can be
seen as a particular kind of modal operator on truth values
(such as “being locally true”).
Programmers or type theorists might prefer to think of them as monads
(the presentation used below should make it clear why this is so).

(For more on these various perspectives, we refer to
\cite[§3.1]{JohnstoneTopos} or \cite[chap. V, §1]{MacLaneMoerdijk} or
\cite[§21.1]{McLarty} for the definition of Lawvere–Tierney topologies
in the context of elementary toposes in general; and to
\cite[chap. II, §2.2]{JohnstoneStone} or \cite[chap. III, ¶5.3]{PicadoPultr} for
nuclei in the context of point-free topology; see \cite[§A4.4–§A4.5
  and §C1.1]{JohnstoneElephant} for a link between the two; see
\cite[§2.5.4]{vanOostenBook} for a presentation of Lawvere–Tierney
topologies, there known as local operators, in the context of
realizability; finally, one might refer to \cite{WilsonThesis}
concerning the order-theoretic properties of nuclei.)

\thingy Even without explaining what the “effective topos” is or what a
“Lawvere–Tierney topology” is in general, we can define a
“Lawvere–Tierney topology on the effective topos” as a map $J \colon
\mathcal{P}(\mathbb{N}) \to \mathcal{P}(\mathbb{N})$ such that
\begin{itemize}
\item $J$ is “computably monotone”: there is\footnote{The names
$\texttt{fmap},\texttt{return},\texttt{join}$ we have given to the
three parts above are taken from the conventions on monads of the
Haskell programming language, but in more mathematical terminology the
latter two would be called the unit $\eta$ and multiplication $\mu$ of
the monad whereas the former expresses functoriality.} a program
  $\texttt{fmap}$ (independent of $U,V \subseteq \mathbb{N}$) which,
  when given a program taking $U$ to $V$, returns a program taking
  $J(U)$ to $J(V)$;
\item $J$ is “computably inflationary”: there is a program
  $\texttt{return}$ (independent of $U \subseteq \mathbb{N}$) taking
  $U$ to $J(U)$;
\item $J$ is “computably idempotent”: there is a program
  $\texttt{join}$ (independent of $U \subseteq \mathbb{N}$) taking
  $J(J(U))$ to $J(U)$.
\end{itemize}
(Here we say that a program “takes $U$ to $V$” or that “when given an
element of $U$ returns an element of $V$” when it is the code of a
partial computable function which is defined at least on $U$ and such
that the image of elements of $U$ by the function fall in $V$.
See \ref{definition-lawvere-tierney-topologies} for a precise
reformulation of this definition.)

We also define a preorder $J \preceq K$ on Lawvere–Tierney topologies
as above when there is a program (independent of $U \subseteq
\mathbb{N}$) taking $J(U)$ to $K(U)$.

The following intuition might be helpful to keep in mind: if we think
of $U \subseteq \N$ as a problem to be solved (namely, the problem of
computably producing some value in the solution set $U$), then $J(U)$
is to be understood as a set of ways to solve the problem with the
help of some generalized computation device (i.e., generalized
oracle).  The three properties demanded of $J$ reflect the general
ideas that: first, if we know how to solve $V$ from a solution of $U$,
then we also know how to do this under help, second, that we can
always choose to do without help and solve $U$ directly, and third,
that our help is available without limit.

The way these notions arise naturally in the context of computability
is as follows: if $g\colon\mathbb{N}\to\mathbb{N}$ is an arbitrary integer
function considered as an ordinary Turing oracle, the map $J_g$ taking
$U\subseteq\mathbb{N}$ to the set $J_g(U) := \{e\in\mathbb{N} \mid
\varphi_e^g(0){\downarrow} \in U\}$ of oracle programs $e$ which, when
provided $g$ as oracle, terminate and return a value in $U$, satisfies
them, i.e., defines a Lawvere–Tierney topology.  Furthermore, $J_f
\preceq J_g$ occurs exactly when $f$ is Turing-reducible to $g$.  The
natural question is whether we can generalize ordinary Turing oracles
and Turing reducibility to make sense of all Lawvere–Tierney
topologies: the T3 degrees do just that.

Accordingly, the central result of \cite{KiharaLT} is that the
Arthur–Nimue–Merlin (“T3”) game degrees exactly describe the
Lawvere–Tierney topology on the effective topos (up to the equivalence
relation “$J \preceq K$ and $K \preceq J$”), along with their order
relation.  In fact, the $J_g$ associated to a given oracle $g$ is not
difficult to describe and is a straightforward generalization of the
$J_f$ described just above: if $U\subseteq\mathbb{N}$, then $J_g(U)$
is the set of (computable!)  Arthur strategies which, using the oracle
$g$ under consideration, and with the help of Nimue, successfully
produce an element of $U$.  A precise definition and theorem statement
are provided in \ref{statement-oracle-topology-correspondence} below.
Furthermore, the T2 and T1 degrees respectively appear precisely as
those defined by a $J$ which preserves (arbitrary) intersections,
resp. both intersections and unions, in $\mathcal{P}(\mathbb{N})$.

\medskip

Let us now make these informal notions precise:

\begin{dfn}\label{definition-lawvere-tierney-topologies} We use the notation
$U\Rrightarrow V$ to denote the set $\{e\in\mathbb{N} \mid \forall m\in
  U.\,(\varphi_e(m){\downarrow}\in V)\}$ of programs sending elements
  of $U$ to elements of $V$.

A \textbf{Lawvere–Tierney topology on the effective topos} (or just
“Lawvere–Tierney topology” when the context is clear) is a map $J \colon
\mathcal{P}(\mathbb{N}) \to \mathcal{P}(\mathbb{N})$ such that the
three sets
\begin{itemize}
\item $\displaystyle\bigcap_{U,V\subseteq\mathbb{N}}\big((U\Rrightarrow V) \Rrightarrow (J(U)\Rrightarrow J(V))\big)$
\item $\displaystyle\bigcap_{U\subseteq\mathbb{N}} \big(U \Rrightarrow J(U)\big)$
\item $\displaystyle\bigcap_{U\subseteq\mathbb{N}} \big(J(J(U))
  \Rrightarrow J(U)\big)$
\end{itemize}
are inhabited.  Furthermore, we quotient these by the equivalence
relation $\equiv$ described in the following paragraph.

For $J,K \colon \mathcal{P}(\mathbb{N}) \to \mathcal{P}(\mathbb{N})$
such maps, we write $J\preceq K$ when
$\displaystyle\bigcap_{U\subseteq\mathbb{N}} \big(J(U) \Rrightarrow
K(U)\big)$ is inhabited (this is a preorder relation); and we write
$J\equiv K$ when $J\preceq K$ and $K\preceq J$ both hold\footnote{We
emphasize that this means that there are computable maps (independent
of $U$) taking each $J(U)$ to the corresponding $K(U)$ and vice versa,
\emph{not} necessarily bijections.}.  We also write $[J]\leq [K]$ for
$J\preceq K$ where $[J]$ refers to the class of $J$ under $\equiv$.
\end{dfn}
\begin{proof}[References]
\cite[definition 1.1]{LeeVanOosten},
\cite[definition 1.2]{vanOostenPittsOperator}
\cite[definition 3.2]{KiharaLT} or
\cite[definition 3.3]{KiharaOracles}.
\end{proof}

Strictly speaking, the Lawvere–Tierney topology is the
class $[J]$ under $\equiv$ of maps $\mathcal{P}(\mathbb{N}) \to
\mathcal{P}(\mathbb{N})$ as described, but by abuse of language we
often refer to $J$ itself as a Lawvere–Tierney topology.

As we pointed out in ¶\ref{informal-explanation-lawvere-tierney-topologies}, for any $g \colon
\mathbb{N}\to\mathbb{N}$ (or indeed $g \colon \mathbb{N} \dasharrow
\mathbb{N}$ a T1 oracle), defining $J_g(U) := \{e\in\mathbb{N} \mid
\varphi_e^g(0){\downarrow} \in U\}$ gives a Lawvere–Tierney $J_g
\colon \mathcal{P}(\mathbb{N}) \to \mathcal{P}(\mathbb{N})$ (this is
easy to check).  Crucially, the T2 and T3 oracles also give
Lawvere–Tierney topologies, and indeed, the T3 oracles give all of
them, as we now explain.

\begin{dfn}\label{definition-oracle-topology-correspondence}
For every T3 oracle $g\colon\mathbb{N} \to
\mathcal{P}(\mathcal{P}(\mathbb{N}))$ we associate a map $J_g \colon
\mathcal{P}(\mathbb{N}) \to \mathcal{P}(\mathbb{N})$ as follows:
$J_g(U)$ is the set of Arthur strategies for producing (with Nimue's
help) an element in $U$; more formally: $J_g(U)$ is the set of Arthur
strategies which are part of an Arthur+Nimue winning strategy for the
reduction game of $\{U\}$ to $g$, where by $\{U\}$ we mean the basic
T3 oracle with this value (so that the initial configuration consists
of the secret part $U$, and Arthur's goal is to declare an element of $U$).

Conversely, to every Lawvere–Tierney topology
$J \colon \mathcal{P}(\mathbb{N}) \to \mathcal{P}(\mathbb{N})$ we
associate a T3 oracle $J^\natural \colon\mathbb{N} \to
\mathcal{P}(\mathcal{P}(\mathbb{N}))$ by: $J^\natural(m) =
\{U\in\mathcal{P}(\mathbb{N}) \mid m\in J(U)\}$.
\end{dfn}
\begin{proof}[References]
For the oracle-to-toplogy direction: \cite[observation 3.3]{KiharaLT},
where it is denoted $g^{\Game\rightarrow}$, or \cite[§5.3 before
  prop. 5.26]{KiharaOracles} where a closely related and
equivalent construction is denoted $j_U$; this is also related to
\cite[prop. 2.3]{LeeVanOosten} (but for basic oracles only).

For the topology-to-oracle direction:
\cite[theorem 2.4]{LeeVanOosten}, where it is denoted $f \mapsto
\theta$, or \cite[theorem 3.1]{KiharaLT}, where it is denoted
$\beta^{\leftarrow}$, or \cite[§5.3 before
  prop. 5.26]{KiharaOracles} where it is denoted $U_j$.
\end{proof}

\begin{thm}\label{statement-oracle-topology-correspondence}
The maps $g \mapsto J_g$ and $J \mapsto J^\natural$ defined
in \ref{definition-oracle-topology-correspondence} define
order-preserving (and hence, reflecting) bijections between the set
$\mathscr{D}_{\mathrm{T3}}$ of T3 degrees and the set of
Lawvere–Tierney topologies on the effective topos, each with their
natural order as previously defined.

Furthermore, under this correspondence, the T2 degrees
(resp. T1 degrees) $\leq\Om$ correspond precisely to those
Lawvere–Tierney topologies defined by a $J$ which preserves arbitrary
intersections (resp. both arbitrary intersections and arbitrary
unions)\footnote{Fine print: we allow the empty union (with
value $\varnothing$) in “arbitrary unions”, thus imposing
$J(\varnothing) = \varnothing$ (which corresponds to constraining the
degree to be $\leq\Om$).  Intersections, on the other hand, are
presumed to take place in some unspecified subset of $\mathbb{N}$ so
that the empty intersection does not give a constraint
on $J(\mathbb{N})$.} in $\mathcal{P}(\mathbb{N})$.
\end{thm}
\begin{proof}[References]
The statement of the first paragraph is
\cite[corollary 3.5]{KiharaLT}.  The statement of the second paragraph
is \cite[prop. 3.7 and theorem 3.10]{KiharaOracles}, but to
spare the reader the inconvenience of having to convert between many
different notations we also give a direct proof below.
\end{proof}
\begin{proof}[Proof of the statement of the second paragraph.]
If $g$ is a T2 oracle, meaning that Nimue does not have any choice in
what to play, then $J_g(U)$ is the set of Arthur strategies which, no
matter what Merlin plays, will cause Arthur to declare a value in $U$:
with this description, it is obvious that $J_g$ preserves arbitrary
intersections.  If $g$ is a T1 oracle, meaning that Merlin does not
have any choice in what to play, then $J_g(U)$ is the set
$\{e\in\mathbb{N} \mid \varphi_e^g(0){\downarrow} \in U\}$ of Arthur
strategies which produce a value in $U$, so it is just the union of
the $J_g(u)$ for $u\in U$, and thus $J_g$ preserves arbitrary unions.

Conversely, if $J$ preserves arbitrary intersections (and in
particular, preserves inclusions), then so long as $m\in
J(\mathbb{N})$ there is a smallest set $U$ such that $m\in J(U)$,
namely $\bigcap\{U \mid m\in J(U)\} = \bigcap J^\natural(m)$, let us
denote this by $J^\flat(m)$ (essentially, Nimue's obvious best move):
this defines a T2 oracle $J^\flat$, and then $J^\natural(m)$ is just
the set of all $U$ such that $U\supseteq J^\flat(m)$, so it is clear
by inspecting the reduction games that $[J^\natural]_{\mathrm{T3}} =
[J^\flat]_{\mathrm{T2}}$.  If $J$ further preserves arbitrary unions,
then $J^\flat(m)$ is the $u$ for which $m \in J(\{u\})$, if it exists,
so $J^\flat$ in fact comes from a T1 oracle.
\end{proof}

\begin{exm}
As we mentioned earlier, if $g\colon\N\dasharrow\N$ is a T1 oracle,
its degree corresponds to $J_g(U) = \{e\in\mathbb{N} \mid
\varphi_e^g(0){\downarrow} \in U\}$.  In particular, the computable
degree $\mathbf{0}$ corresponds to $J_{\mathbf{0}}(U) =
\{e\in\mathbb{N} \mid \varphi_e(0){\downarrow} \in U\}$: but this is
equivalent to the much simpler Lawvere–Tierney topology $\id \colon U
\mapsto U$ (indeed, $\bigcap_{U\subseteq\mathbb{N}}
\big(J_{\mathbf{0}}(U) \Rrightarrow U\big)$ is inhabited by a code for
a universal program $e \mapsto \varphi_e(0)$, and conversely
$\bigcap_{U\subseteq\mathbb{N}} \big(U \Rrightarrow
J_{\mathbf{0}}(U)\big)$ is inhabited by a code for a
constant-producing program: this shows $J_{\mathbf{0}} \equiv \id$).

The omniscient degree $\Om$ (here seen as that of $\{\{u\} \mid u \in
\N\}$) has the defining property that there exist a common $e$ such
that $e \in J_\Om(U)$ for every inhabited $U$, but still
$J_\Om(\varnothing) = \varnothing$.  Indeed, $e$ can be the code of an
Arthur strategy which queries the oracle and declares the value he
receives from it: a corresponding Nimue strategy consists of choosing
an element $u$ of $U$ and playing $\{u\}$, forcing Merlin to show $u$
to Arthur.  More simply, this is equivalent to the Lawvere–Tierney
topology that takes every inhabited $U$ to $\N$ and $\varnothing$ to
$\varnothing$.

Finally, the $\top$ degree has the defining property that even
$J_\top(\varnothing)$ is inhabited.
\end{exm}

\thingy\label{algebraic-operations-on-l-t-topologies}
We mention at this point that it is also possible to describe
the algebraic operations (meets and joins, for which
cf. \ref{meet-of-t2-degrees} and \ref{join-of-degrees}, as well as the
Heyting operation briefly mentioned
in \ref{digression-heyting-operation}) at the level of the
Lawvere–Tierney topologies.  Specifically, here are formulas:
\begin{itemize}
\item The \emph{meet} of $J_1$ and $J_2$ is represented by $U \mapsto
  J_1(U) \sqcap J_2(U)$.
\item The \emph{join} of $J_1$ and $J_2$ is represented by $U \mapsto
  \bigcap_{V\subseteq\mathbb{N}} ((((J_1(V) \sqcup J_2(V))
  \Rrightarrow V) \sqcap(U\Rrightarrow V)) \penalty0 \Rrightarrow V)$.
\item The \emph{Heyting operation} of $J$ and $K$ is represented by $U
  \mapsto \bigcap_{V\subseteq\mathbb{N}} (J(V) \Rrightarrow K(U\sqcup
  V))$.
\end{itemize}
Here we have used the following notations for the corresponding
operations on $\P(\N)$:
\begin{itemize}
\item $U\sqcap V := \{\langle u,v\rangle \mid u\in U,\;v\in V\}$ is
  the Gödel coding of the Cartesian product\footnote{Elsewhere we have
  simply written this $U\times V$, but here, for consistency's sake we
  prefer to emphasize its connection to the meet operation.}.
\item $U \sqcup V := \{\langle 0, u \rangle \mid u \in U\} \cup
  \{\langle 1, v \rangle \mid v\in V\}$ is the labeled disjoint union.
\item $U\Rrightarrow V := \{e\in\mathbb{N} \mid \forall m\in
  U.\,(\varphi_e(m){\downarrow}\in V)\}$ as previously.
\end{itemize}
While we have expressed these three formulas for the effective topos,
they are just special cases of general formulas for these
operations for Lawvere–Tierney topologies on an arbitrary topos (or,
indeed, nuclei on a frame): the formula for the meet of topologies is
immediate by the fact that Lawvere–Tierney topologies are ordered
pointwise, the formula for the join will follow from
\cref{smallest-l-t-topology-above-monotone-map} below, and for the
formula for the Heyting operation we refer to \cite[chap. II, §2.5,
  second line of p. 52]{JohnstoneStone} or \cite[chap. 6,
  theorem 23.4(2)]{WilsonThesis} or \cite{MO444977}.

\thingy\label{explicit-formulas-for-coding-decoding-topologies}
As we will see below,
it is even possible to compute explicitly the Lawvere–Tierney
$J_{\mathbf{d}_{P,Q}}$ and $J_{\mathbf{c}_{P,Q}}$ associated to the
degrees $\mathbf{d}_{P,Q}$ and $\mathbf{c}_{P,Q}$ respectively (in the
sense of \ref{definition-oracle-topology-correspondence}): they are
given (up to equivalence) by:
\[
J_{\mathbf{d}_{P,Q}}(U) = \bigcap_{V\subseteq\mathbb{N}}
(((((P\sqcap(\{0\}\Rrightarrow V))\cup(Q\sqcap(\{1\}\Rrightarrow V)))\Rrightarrow V)
\sqcap(U\Rrightarrow V)) \Rrightarrow V)
\]
\[
J_{\mathbf{c}_{P,Q}}(U) = \bigcap_{V\subseteq\mathbb{N}}
(((((P\Rrightarrow V)\cup(Q\Rrightarrow V))\Rrightarrow V)
\sqcap(U\Rrightarrow V)) \Rrightarrow V)
\]
(Note that, here, $\cup$ is not a typo: it is the set-theoretical
union.)  These formulas will be
explained in ¶\ref{explanation-explicit-formulas-for-coding-decoding-topologies} but do
not seem particularly useful for the study of $\mathbf{d}_{P,Q}$ and
$\mathbf{c}_{P,Q}$, so we just mention them in passing.

\bigskip

\textit{(From this point on, we assume that assume familiarity with
  the theory of toposes and the category-theoretic point of view on
  realizability.)}

\thingy\label{specific-notations-for-topos-theoretic-section}
\textbf{Additional notations for this section.} If $S$ is a
set, we denote by $\nabla S$ the corresponding “constant” object in
the effective topos (\cite[§2.4]{vanOostenBook}), which, in practice,
is represented by the assembly $(S, \; (s \mapsto \mathbb{N}))$ (in
other words, the underlying set is $S$ and every natural number
realizes every element of $S$).

We denote by $\Omega$ the subobject classifier of a topos, and
$\mathbf{N}$ its natural numbers object.  We denote by $\llbracket
\phi \rrbracket$ the truth value of a proposition $\phi$ if there is a
risk of ambiguity, but we will often use this without the brackets
(e.g., in \cref{smallest-monotonic-map-making-subobject-dense} we write
$f(x\in A)$ for $f(\llbracket x\in A\rrbracket)$).

\bigskip

We first recall a few facts which are part of the topos-theoretic (or
frame-theoretic) folklore:

\begin{lem}\label{smallest-monotonic-map-making-subobject-dense}
If $A\hookrightarrow X$ is a subobject of an object in a topos, then
the smallest monotonic $f\colon\Omega\to\Omega$ such that $\forall
x:X.\; f(x\in A)$ exists and is given explicitly by:
\[
u \; \mapsto \; \exists x:X.\; ((x\in A) \Rightarrow u))
\]
\end{lem}
\begin{proof}
We need to show (working internally) that $f$ satisfies $\forall
x:X.\; f(x\in A)$ iff we have $(\exists x:X.\; ((x\in A) \Rightarrow
u)) \; \leq \; f(u)$ for every $u:\Omega$.

For the “if” part, assume $(\exists x:X.\; ((x\in A) \Rightarrow u))
\, \Rightarrow \, f(u)$, and let $x \in X$.  We trivially have $((x\in
A) \Rightarrow (x\in A))$, so $f(x\in A)$, as claimed.

For the “only if” part, assume $\forall x:X.\; f(x\in A)$, and assume
$(x\in A) \Rightarrow u$ for some $x$: then $f(x\in A) \Rightarrow
f(u)$ because $f$ is monotonic, so we have $f(u)$, as claimed.
\end{proof}

\begin{prop}\label{smallest-l-t-topology-above-monotone-map}
If $f\colon\Omega\to\Omega$ is order-preserving where $\Omega$ is the
set of truth values in a topos, then the smallest Lawvere–Tierney
topology $j\colon\Omega\to\Omega$ that satisfies $f\leq j$ exists and
is given explicitly by:
\[
u \; \mapsto \; \forall v:\Omega.\;(((f(v) \Rightarrow v)\land (u\Rightarrow v))\Rightarrow v)
\]
\end{prop}
\begin{proof}[References]
See \cite[prop. 1.2]{LeeVanOosten} and the second line of the proof
there.  In the interest of the history of this folkore result,
additional references for it and closely related statements and/or the
following corollary include: \cite[theorem 3.57]{JohnstoneTopos} (the
core idea, which is attributed there to André \textsc{Joyal}),
\cite[prop. 16.3]{HylandEff} (stated for the effective topos,
but this is not relevant), \cite[prop. 1]{JohnstoneNote} (under the
assumption that $f$ is a “prenucleus”) and \cite[corollary
  A4.5.13(i)]{JohnstoneElephant} (without the explicit formula).  We
might also mention \cite[definition 4.2(a)]{CaramelloIL} even though
it is couched in the form of satisfying a logical formula rather than
making a subobject dense.  See
also \cite{MO484068} for some more context and yet another proof.
\end{proof}

\begin{cor}\label{smallest-l-t-topology-making-subobject-dense}
If $A\hookrightarrow X$ is a subobject of an object in a topos, then
the smallest Lawvere–Tierney topology $j$ that makes $A$ $j$-dense
in $X$ (i.e., satisfies $\forall x:X.\; j(x\in A)$) exists and is
given explicitly by:
\[
j_{A\hookrightarrow X}\colon
u \; \mapsto \; \forall v:\Omega.\;((((\exists x:X.\; ((x\in A) \Rightarrow v))\Rightarrow v)\land (u\Rightarrow v))\Rightarrow v)
\]
\end{cor}
\begin{proof}
The corollary follows trivially from
\cref{smallest-monotonic-map-making-subobject-dense} and
\cref{smallest-l-t-topology-above-monotone-map}.
\end{proof}

\begin{rmk}\label{remark-general-topos-reducibility}
For what it's worth, when $A \hookrightarrow X$ and $B \hookrightarrow
Y$ are two subobjects in a topos, we can use the above to write down a
logical formula expressing the fact that $j_{A\hookrightarrow X} \leq
j_{B\hookrightarrow Y}$, namely:
\[
\begin{aligned}
\forall v:\Omega.\;(&((\exists y:Y.\; ((y\in B) \Rightarrow v))\Rightarrow v)\\
\Rightarrow\;&((\exists x:X.\; ((x\in A) \Rightarrow v))\Rightarrow v))
\end{aligned}
\]
(It is a straightforward exercise in higher-order propositional
logic to check that this formula is indeed equivalent to $\forall
u:\Omega.\; (j_{A\hookrightarrow X}(u) \Rightarrow j_{B\hookrightarrow
  Y}(u))$ where $j_{A\hookrightarrow X}$ and $j_{B\hookrightarrow Y}$
are given in \ref{smallest-l-t-topology-making-subobject-dense}.)

While this formula does not seem particularly engaging nor even really
useful, it is interesting to know that it exists, in that (a) it gives
us a notion of Turing-like reducibility for any kind of subobjects in
any topos (maybe best applied to subobjects of the form $A \cup
(X\setminus A)$ to better imitate the usual Turing reducibility), and
(b) specifically for the effective topos, combined with
propositions \ref{t3-oracle-to-internal-topology}
and \ref{other-oracles-to-internal-topology}, it lets us, in
principle, express the notions of reducibility of our T0–T3 oracles
purely in terms of realizability and without any reference to games or
oracles (cf. \cite{CSTSE55511} for an explicit unwinding of ordinary
Turing reducibility from this formula).
\end{rmk}

\thingy\label{preparation-for-t3-oracle-to-internal-topology}
Let us now also explain or recall the internal perspective
on how a T3 oracle $g \colon
\mathbb{N} \to \mathcal{P}(\mathcal{P}(\mathbb{N}))$ gives rise
to a Lawvere–Tierney topology on the effective topos.

First, we have a standard map $\varpi\colon
\nabla\mathcal{P}(\mathbb{N}) \to \Omega$ (represented simply by the
identity map on $\mathcal{P}(\mathbb{N})$).  We note that it can be
defined internally by identifying $\nabla\mathcal{P}(\mathbb{N})$ with
the object of $(\neg\neg)$-closed subobjects of $\mathbf{N}$ (meaning
the internal hom object $(\nabla 2)^{\mathbf{N}}$, where $\nabla 2$ is
itself seen as the object $\{p\in\Omega \mid \neg\neg p = p\}$ of
$(\neg\neg)$-stable truth values; see \cite[p. 129]{vanOostenBook}):
then $\varpi(G)$ is the truth value $\llbracket \exists
v:\mathbf{N}.(v\in G) \rrbracket$ of “$G$ is inhabited” (see
\textit{op. cit.}, proof of prop. 3.1.11, which states that
$\varpi$ is surjective).

Next, $\nabla\mathcal{P}(\mathcal{P}(\mathbb{N}))$ can similarly be
defined as the object of $(\neg\neg)$-closed subobjects
of $\nabla\mathcal{P}(\mathbb{N})$.  A map of sets $g \colon
\mathbb{N} \to \mathcal{P}(\mathcal{P}(\mathbb{N}))$ defines a
morphism $\mathbf{N} \to \nabla\mathcal{P}(\mathcal{P}(\mathbb{N}))$
in the effective topos, which we somewhat abusively still denote
by $g$.  We are now concerned with two subobjects of $\mathbf{N}
\times \nabla\mathcal{P}(\mathcal{P}(\mathbb{N}))$ associated with the
situation.  The first is the subobject of pairs $\{(n,G) \mid G \in
g(n)\}$ (which we can think of as a move $n$ of Arthur and one $G$ of
Nimue for the oracle $g$): note that it has the set of the same
description as underlying set, and the only realizer of $(n,G)$ is
simply the integer $n$ itself.  The second is the subobject $\{(n,G) \mid
G \in g(n) \,\land\, \varpi(G)\}$ of the former consisting of pairs
such that, additionally, $G$ is inhabited: note that now realizers of
$(n,G)$ are given by $\{n\} \times G$ (we can think of them as a move
with a response by Merlin).

\begin{prop}\label{t3-oracle-to-internal-topology}
If $g \colon \mathbb{N} \to \mathcal{P}(\mathcal{P}(\mathbb{N}))$ is a
T3 oracle, then the Lawvere–Tierney topology\footnote{We use a
lowercase $j$ to refer to the morphism in $\mathbf{Eff}$ to contrast
with the uppercase $J$ referring to a map $\mathcal{P}(\mathbb{N}) \to
\mathcal{P}(\mathbb{N})$ that represents it.} on the effective topos
$j_g\colon\Omega\to\Omega$ associated to $g$ (in the sense
of \cref{definition-oracle-topology-correspondence}) is the one given
by \cref{smallest-l-t-topology-making-subobject-dense} for the
subobject
\[
\{(n,G) \mid G \in g(n) \,\land\, \varpi(G)\}
\hookrightarrow
\{(n,G) \mid G \in g(n)\}
\]
described above.
\end{prop}
\begin{proof}[Explanation]
This is essentially a restatement
of \cite[prop. 2.3]{LeeVanOosten}, but since the latter is left
to the reader (and since we need to be sure that the construction does
agree with the one of \cite{KiharaLT} on which our theory is based),
we give more explanations.  The map $f$ associated to the
inclusion $\{(n,G) \mid G \in g(n) \,\land\, \varpi(G)\} \hookrightarrow
\{(n,G) \mid G \in g(n)\}$ under consideration
by means of \cref{smallest-monotonic-map-making-subobject-dense}
is $u \; \mapsto \;
\exists (n,G). \; (G \in g(n) \,\land\, (\varpi(G) \Rightarrow u))$.
Now a realizer for $\exists (n,G). \; (G \in g(n) \,\land\, (\varpi(G)
\Rightarrow u))$, if we denote $U \subseteq \mathbb{N}$ for the set
of realizers of the truth value $u:\Omega$,
is the code $\langle n,e\rangle$ for a pair where, for
some $G \in g(n)$ we have $e \in (G \Rrightarrow U)$.  Up to trivial
notational changes, this is what is denoted $g^\rightarrow(U)$ in
\cite[beginning of the proof of theorem 3.1]{KiharaLT}.  Now
\textit{op. cit.}, theorem 3.4 ensures that the $j$ defined from $f$
by \ref{smallest-l-t-topology-above-monotone-map} is the Lawvere–Tierney
topology associated with $g$ (denoted there as $g^{\Game\rightarrow}$).
\end{proof}

\thingy\label{other-oracles-to-internal-topology} When $g\colon
\mathbb{N} \dasharrow \mathcal{P}(\mathbb{N})$ is a T2 oracle, the
above description of $j_g$ simplifies as follows.  First, let us
consider $g$ as a total function $g\colon D_g \to
\mathcal{P}(\mathbb{N})$ (with $D_g := \{n\in\mathbb{N} :
g(n){\downarrow}\}$).  So $D_g$ defines a $(\neg\neg)$-closed
subobject of $\mathbf{N}$ in $\mathbf{Eff}$ (very accurately, it is
$\nabla D_g \cap \mathbf{N}$ where the intersection is taken as
subobjects of $\nabla\mathbb{N}$), which, slightly abusively, we still
denote by $D_g$.  Then the set map $g\colon D_g \to
\mathcal{P}(\mathbb{N})$ gives rise to a morphism $D_g \to \nabla
\mathcal{P}(\mathbb{N})$ in $\mathbf{Eff}$, and the above description
translates to $j_g$ being the topology associated to the subobject
\[
\{n \mid n\in D_g \,\land\, \varpi(g(n))\} \hookrightarrow D_g
\]

When $g\colon \mathbb{N} \dasharrow \mathbb{N}$ is a T1 oracle, which again
we had better see as a total function $g\colon D_g \to \mathbb{N}$ and
which gives rise to a morphism from a $(\neg\neg)$-closed subset $D_g$
of $\mathbf{N}$ to $\nabla\mathbb{N}$ in $\mathbf{Eff}$, the topology
$j_g$ is now that associated to the subobject
\[
\{n \mid n\in D_g \,\land\, g(n) \in \mathbf{N}\} \hookrightarrow D_g
\]
(keeping in mind that $\mathbf{N}$ is a subobject\footnote{We can
see $\nabla\mathbb{N}$ as the object of $(\neg\neg)$-closed,
subterminal and $(\neg\neg)$-inhabited subobjects of $\mathbf{N}$, in
which case $\mathbf{N}$ consists of those that are inhabited (see
\cref{construction-j-sheafification} for a justification), i.e., are
mapped to true by $\varpi$.} of $\nabla\mathbb{N}$).  This can be a
bit confusing as we have now two “domains of definition” of $g$: the
“effective” one on the left being the one which makes $g$ into a
morphism to $\mathbf{N}$ and the “classical” one on the right being
the one which makes $g$ into a morphism to $\nabla\mathbb{N}$; and the
topology $j_g$ is precisely the one which makes the effective domain
dense in the classical domain.

When $g\colon \mathbb{N} \to \mathbb{N}$ is a T0 oracle, which gives
rise to a morphism $\mathbf{N} \to \nabla\mathbb{N}$ in
$\mathbf{Eff}$, the topology $j_g$ is that associated to the subobject
\[
\{n \mid g(n) \in \mathbf{N}\} \hookrightarrow \mathbf{N}
\]
where again the subobject on the left might be seen as the “effective”
domain of definition of $g$, i.e., the domain of $g$ as a partial
morphism $\mathbf{N} \dasharrow \mathbf{N}$.  If the set-theoretic
$g$ is the characteristic function of a subset $A \subseteq
\mathbb{N}$, which we identify with a $(\neg\neg)$-closed subobject of
$\mathbf{N}$, then our inclusion of subobjects above is $\{(n\in A)
\lor\neg (n\in A)\} \hookrightarrow \mathbf{N}$.  Thus we have made
the connection with \cite[theorem 17.2]{HylandEff}.

Finally, when $\mathscr{G} \in \mathcal{P}(\mathcal{P}(\mathbb{N}))$
is a basic T3 oracle, the topology $j_\mathscr{G}$ is that associated
to the subobject
\[
\{G \in \nabla \mathscr{G} \mid \varpi(G)\}
\hookrightarrow
\nabla \mathscr{G}
\]

\bigbreak

Having discussed Lawvere–Tierney topologies in general, let us now
explain how the $\mathbf{d}_{P,Q}$ and $\mathbf{c}_{P,Q}$ degrees of
\cref{section-decoding-and-coding-degrees} (or their associated
topologies) fit into the picture.

\begin{dfn}\label{definition-x-p-q-object}
If $P,Q\subseteq\mathbb{N}$ are inhabited and disjoint, we denote by
$X_{P,Q}$ the object of the effective topos $\mathbf{Eff}$ given by
the assembly $(\{0,1\}, \; (0 \mapsto P, \; 1 \mapsto Q))$.  (In other
words, the set of realizers of $0$ is $P$ and the set of realizers
of $1$ is $Q$.)
\end{dfn}

Note as a special case that $X_{\{0\},\{1\}} = 2$ is the two-element
set in $\mathbf{Eff}$ (the coproduct of two copies of the terminal
object).

We also denote by $\nabla 2$, as is standard, the “constant” object
$\nabla\{0,1\}$ given by the assembly $(\{0,1\}, \; (0 \mapsto
\mathbb{N}, \; 1 \mapsto \mathbb{N}))$.

\begin{prop}\label{toposic-construction-of-coding-decoding-degrees}
If $P,Q\subseteq\mathbb{N}$ are inhabited and disjoint.  Then the
identity on the underlying set $\{0,1\}$ defines monomorphisms of
assemblies
\[
2 \hookrightarrow X_{P,Q} \hookrightarrow \nabla 2
\]
and the topologies $j_{2 \hookrightarrow X_{P,Q}}$ and $j_{X_{P,Q}
  \hookrightarrow \nabla 2}$ defined
by \cref{smallest-l-t-topology-making-subobject-dense} are the ones
associated respectively to $\mathbf{d}_{P,Q}$ and $\mathbf{c}_{P,Q}$
under the
correspondence from \cref{statement-oracle-topology-correspondence}.
\end{prop}
\begin{proof}
We note that all objects under discussion are assemblies.  The arrow
$2 \to X_{P,Q}$ comes from the identity on sets, which is
tracked by a program taking $0$ to some fixed element of $P$ and $1$
to some fixed element of $Q$.  The arrow $X_{P,Q} \to
\nabla 2$ also comes from the identity on sets, which is trivially
tracked.  Both are monomorphisms because the identity is injective.

To make the following description less confusing, we need a notation
to distinguish a subset $E \subseteq \mathbb{N}$ and the
$(\neg\neg)$-closed subobject of $\mathbf{N}$ it defines in the
effective topos (precisely: $\nabla E \cap \mathbf{N}$ where the
intersection is taken as subobjects of $\nabla\mathbb{N}$).  So let us
denote the latter by $\class(E)$.

Now concerning $\mathbf{d}_{P,Q}$: the description given
in ¶\ref{other-oracles-to-internal-topology} for the T1 level ensures
that it is associated with the subobject $(\class(P) \,\cup\, \class(Q))
\hookrightarrow \class(P \,\cup\, Q)$.  So, to be clear, $\class(P \,\cup\,
Q)$ is the assembly with $P \,\cup\, Q$ as underlying set, where each
element of $P \,\cup\, Q$ is realized by itself, whereas $\class(P) \,\cup\,
\class(Q)$ is that with the same underlying set, where each element
$p$ of $P$ is realized by $\langle 0,p\rangle$ and each element $q$ of
$Q$ by $\langle 1,q\rangle$.  Now we have a pullback diagram
\[
\begin{array}{ccc}
(\class(P) \,\cup\, \class(Q)) & \hookrightarrow & \class(P \,\cup\, Q) \\
\downarrow & \square & \downarrow \\
2 & \hookrightarrow & X_{P,Q} \\
\end{array}
\]
where the vertical arrow on the right comes at the level of sets from
taking elements of $P$ to $0$ and elements of $Q$ to $1$ and tracked
by the identity program, and it is surjective (meaning, epimorphic)
and indeed even has a section.  So to say that $\forall x \in
\class(P \,\cup\, Q). \; j(x \in (\class(P) \,\cup\, \class(Q)))$ is the same
as to say $\forall x \in X_{P,Q}. \; j(x \in 2)$.

Concerning $\mathbf{c}_{P,Q}$: the description given
in ¶\ref{other-oracles-to-internal-topology} for basic T3 oracles
ensures that it associated with the subobject $\{G \in \nabla \{P,Q\}
\mid \varpi(G)\} \hookrightarrow \nabla\{P,Q\}$, and the left hand side
is just a fancy name for $X_{P,Q}$ while the right hand side is one
for $\nabla 2$.
\end{proof}

\thingy\label{explanation-explicit-formulas-for-coding-decoding-topologies}
This explains the formulas stated
in ¶\ref{explicit-formulas-for-coding-decoding-topologies}: they are
obtained by applying
\cref{smallest-l-t-topology-making-subobject-dense} to $2
\hookrightarrow X_{P,Q}$ and $X_{P,Q} \hookrightarrow \nabla 2$
respectively.  Indeed, (if we denote by $V$ the set of realizers
of $v:\Omega$) the set of realizers of $\exists x:X_{P,Q}.\; ((x\in 2)
\Rightarrow v)$, or equivalently, of $\exists x:\nabla 2.\; ((x\in
X_{P,Q}) \land ((x\in 2) \Rightarrow v))$ is
$((P\sqcap(\{0\}\Rrightarrow V))\cup(Q\sqcap(\{1\}\Rrightarrow
V)))\Rrightarrow V$, and that the set of realizers of $\exists
x:\nabla 2.\; ((x\in X_{P,Q}) \Rightarrow v)$ is $((P\Rrightarrow
V)\cup(Q\Rrightarrow V))\Rrightarrow V$.

\begin{rmk}
To say that $X_{P_2,Q_2} \subseteq X_{P_1,Q_1}$ as subobjects
of $\nabla 2$ means precisely (keeping in mind that a morphism
$X_{P_2,Q_2} \to X_{P_1,Q_1}$ compatible with the map to $\nabla 2$
needs to be the identity on the underlying set $\{0,1\}$) that there exists
$f\colon\mathbb{N}\dasharrow\mathbb{N}$ partial computable tracking
the identity on $\{0,1\}$, in other words, defined at least on
$P_2\cup Q_2$, and such that $f(P_2) \subseteq P_1$ and $f(Q_2)
\subseteq Q_1$: this is precisely the definition we have given
in \ref{definition-many-to-one-reduction} for many-to-one reduction
$(P_2,Q_2) \mathrel{\preceq_{\mathrm{m}}} (P_1,Q_1)$ of promise
problems.

So the $X_{P,Q}$ can be seen as the incarnation in the effective topos
of the promise problems $(P,Q)$ up to many-to-one reduction, the
easiest one being $2$ (the decidable promise problem) and the
“ultimate” one being $\nabla 2$ (not itself a promise problem, but
see \cref{degenerate-cases-of-decoding-and-coding-degrees}).  The
decoding degree $\mathbf{d}_{P,Q}$ corresponds to the smallest
topology making $2 \hookrightarrow X_{P,Q}$ dense, and the coding
degree $\mathbf{c}_{P,Q}$ corresponds to the smallest topology
making $X_{P,Q} \hookrightarrow \nabla 2$ dense.  Now propositions
\ref{many-to-one-reduction-implies-decoding-oracles-reduction}
and \ref{many-to-one-reduction-implies-coding-oracles-reduction}
become obvious.

(One might, of course, be interested in the “relative-coding-decoding”
degree defined, when $(P_1,Q_1)$ and $(P_2,Q_2)$ are two promise
problems such that $(P_2,Q_2) \mathrel{\preceq_{\mathrm{m}}}
(P_1,Q_1)$, as corresponding to the smallest topology $j_{X_{P_2,Q_2}
  \hookrightarrow X_{P_1,Q_1}}$ making $X_{P_2,Q_2} \hookrightarrow
X_{P_1,Q_1}$ dense.  An adaptation of the proof of
\cref{toposic-construction-of-coding-decoding-degrees} shows that it
corresponds to the oracle which takes an instance of the harder
problem $(P_1,Q_1)$ and returns the answer in the form of an instance
of the easier problem $(P_2,Q_2)$.  But we do not mention it any
further since we do not have much to say about it.)
\end{rmk}

\bigbreak

Having discussed how the $\mathbf{d}_{P,Q}$ and $\mathbf{c}_{P,Q}$
are associated with the $X_{P,Q}$ object,
we now turn to the $\tilde{\mathbf{d}}_A$ and $\tilde{\mathbf{c}}_A$
degrees (or their associated topologies).  To better make sense of the
appearance of $H_0^A,H_1^A$ in \cref{section-co-turing-degrees},
and, specifically, to make the appearance of the object $X_{P,Q}$ in this case
(that is, $X_{H_0^A, H_1^A}$) less ad hoc, as we will do in
\cref{construction-j-assembly-for-turing-degree}, we now
turn to the notion of \textbf{sheafification} for a Lawvere–Tierney topology,
whose construction we first recall.

(Let us also mention that a more general definition of sheafification,
and reducibilities related to our
\cref{remark-general-topos-reducibility}, but applicable in the wider
context of type theory, has been given and investigated
in \cite{AhmanBauer}.)

\begin{prop}\label{construction-j-sheafification}
If $j\colon\Omega\to\Omega$ is a Lawvere–Tierney topology on a topos,
and $X$ is an object of said topos which is $j$-separated (meaning:
$\forall x,y:X.(j(x=y) \Rightarrow (x=y))$), then the sheafification
$\sh_j(X)$ of $X$ for $j$ is the object of subobjects $S$ of $X$ that
satisfy the following three properties:
\begin{enumerate}
\item $S$ is $j$-closed: $\forall x\in X.\, (j(x\in S) \Rightarrow
  (x\in S))$,
\item $S$ is subterminal: $\forall x,y\colon X.\, (x,y\in S \Rightarrow
  x=y)$,
\item $S$ is $j$-inhabited: $j(\exists x\colon X.\, (x\in S))$.
\end{enumerate}
with the unit morphism $X \to \sh_j(X)$ taking $x$ to the subobject
$\{x\}$ (and it is injective).
\end{prop}
\begin{proof}[Explanation]
This is essentially a restatement of the construction of the
associated sheaf object in \cite[chap. V, §3]{MacLaneMoerdijk}.  In
summary: first observe that $\Omega_j := \{u\in\Omega \mid j(u)=u\}$ is a
$j$-sheaf, so this is also the case of the object $(\Omega_j)^X$ of
$j$-closed subobjects of $X$, then observe that the next two
conditions are themselves $j$-closed in $(\Omega_j)^X$ in the sense
that if $S$ is $j$-closed then $j(\phi)$ implies $\phi$ for either one
of the two conditions $\phi$.  But a $j$-closed subobject of a sheaf
is a sheaf, so we have defined a sheaf $X_j$.  The third condition
guarantees that $X \hookrightarrow X_j$ is $j$-dense.  So now if $Y$ is any sheaf,
the morphism $X\to Y$ factors uniquely through $X_j$ by the very
definition of a sheaf.  So $X_j$ is indeed the $j$-sheafification
of $X$.
\end{proof}

(If $X$ were not assumed $j$-separated, we would first need to
quotient by the equivalence relation $j(x=y)$.  We can, however, do
both in one go by replacing the second condition by “$S$ is
$j$-subterminal”, meaning $\forall x,y\colon X.\, (x,y\in S
\Rightarrow j(x=y))$; but then $X \to \sh_j(X)$ is given by $x \mapsto
\{y\in X \mid j(x=y)\}$, and it is, of course, no longer
injective in general.)

\thingy\label{reminder-on-j-false}
In what follows, the reader is advised to keep in mind that,
when $j\colon\Omega\to\Omega$ is a Lawvere–Tierney topology on a
topos, $j \leq (\neg\neg)$ is equivalent to $j(\bot) = \bot$.
(\textit{Proof:} if $j \leq (\neg\neg)$ then in particular $j(\bot)
\leq \neg\neg\bot = \bot$; but conversely, assuming $j(\bot) = \bot$,
if $j(u)$ holds, i.e., $j(u) = \top$, then $\neg(u=\bot)$, which is
exactly to say $\neg\neg u$, and this shows $j(u) \leq \neg\neg u$.)

Furthermore, it is worth keeping in mind that this is also equivalent
to $2$ being $j$-separated (because $j(0=1) \Rightarrow (0=1)$ boils
down to $j(\bot) = \bot$).

\medskip

The following proposition is certainly well known (e.g., it is
implicit in \cite[§2]{vanOostenPittsOperator}), but we could not find
the explicit statement in the literature:

\begin{prop}\label{construction-j-assembly}
If $J \colon \mathcal{P}(\mathbb{N}) \to \mathcal{P}(\mathbb{N})$ is a
Lawvere–Tierney topology on the effective topos, and $J(\varnothing) =
\varnothing$ (which amounts to $J \leq (\neg\neg)$), and if $(X,E)$ is
an assembly, then $(X,E)$ is $J$-separated, and its $J$-sheafification
is $(X,\; J\circ E)$ (meaning the set of realizers of $x$ in the
$J$-sheafification is $J$ of the set of realizers of $x$ in the
original assembly).
\end{prop}
\begin{proof}
That $(X,E)$ is $J$-separated means that given realizers of $x$ and
$y$ and a $J$-realizer of $x=y$ (meaning an element of $J(E(x)\cap
E(y))$) we can (uniformly in $x,y$) produce a realizer of $x=y$: the
given realizer of $x$ will do.

We define a morphism $(X,\; J\circ E) \to \sh_J (X,E)$ using the
construction of $\sh_J (X,E)$ given
in \cref{construction-j-sheafification} as follows:
since $\sh_J (X,E)$ is a subobject of the object $\Omega^{(X,E)}$ of
subobjects of $(X,E)$, we will define a morphism $(X,\; J\circ E) \to
\Omega^{(X,E)}$ by adjunction from a morphism $(X,\; J\circ E) \times
(X, E) \to \Omega$, which takes $(x,y) \in X \times X$ to
$\varnothing$ if $x\neq y$, and to $J(E(x))$ if $x=y$.  It is then a
straightforward affair in unrolling the definitions to check that it
injects $(X,\; J\circ E)$ in $\Omega^{(X,E)}$ and that its image is
precisely the $\sh_J (X,E)$ we have constructed
in \ref{construction-j-sheafification}.
\end{proof}

\thingy For example, when $J$ is the $(\neg\neg)$ topology given by
$\varnothing \mapsto \varnothing$ but $U \mapsto \mathbb{N}$ when $U
\subseteq \mathbb{N}$ is inhabited, the above construction (provided
we remove the useless elements with $E(x) = \varnothing$ in $X$) gives
$\sh_{\neg\neg} (X,E) = \nabla X$ as sheafification.

\begin{rmk}\label{remark-on-hom-nat-to-nat}
One consequence of \cref{construction-j-assembly} which may help
clarify matters (and which does not seem to appear explicitly in the
literature) is this: if $g$ is a T3 oracle and $J_g$ the topology it
defines (cf. \cref{definition-oracle-topology-correspondence}), then
in the topos of $J_g$-sheaves (that is, the subtopos of the effective
topos defined by $J_g$), the morphisms $\mathbf{N}_g \to
\mathbf{N}_g$, where $\mathbf{N}_g$ is its natural numbers object, are
precisely given by maps of sets $f \colon \N \to \N$ such that $f
\redTThree g$.  In fact, even more is true: the object
$\mathbf{N}_g^{\mathbf{N}_g}$ in the topos of $J_g$-sheaves is given
by the assembly whose underlying set is the set of such maps, with the
realizers of $e$ being the winning Arthur strategies (i.e., the Arthur
part of a winning Arthur+Nimue strategy) for the reduction.
\end{rmk}

\begin{proof}[Justification of remark]
Indeed, to see this, let us denote by $\mathbf{Eff}_g$ the topos of
$J_g$-sheaves, and $i^* \dashv i_* \colon \mathbf{Eff}_g \to
\mathbf{Eff}$ the embedding; and let us first recall that, by
\cite[proof of lemma A4.1.14]{JohnstoneElephant}, the natural numbers
object $\mathbf{N}_g$ in $\mathbf{Eff}_g$ is the $J_g$-sheafification
$i^* \mathbf{N}$ of the natural numbers object $\mathbf{N}$ in the
original topos $\mathbf{Eff}$.  So by \cref{construction-j-assembly}
this $\mathbf{N}_g$ is given by the assembly $(\mathbb{N}, \; k
\mapsto J_g(\{k\}))$.  Now morphisms $\mathbf{N}_g \to \mathbf{N}_g$
in $\mathbf{Eff}_g$ are the same as in $\mathbf{Eff}$ (because $i_*$
is fully faithful), which are also the same as morphisms $\mathbf{N}
\to \mathbf{N}_g$ in $\mathbf{Eff}$ (because $i^* \dashv i_*$ gives
$\Hom_{\mathbf{Eff}_g}(i^*\mathbf{N}, \mathbf{N}_g) =
\Hom_{\mathbf{Eff}}(\mathbf{N}, i_* \mathbf{N}_g)$, that is
$\Hom_{\mathbf{Eff}_g}(\mathbf{N}_g, \mathbf{N}_g) =
\Hom_{\mathbf{Eff}}(\mathbf{N}, \mathbf{N}_g)$).  Now morphisms in
$\mathbf{Eff}$ from the assembly $(\mathbb{N}, \; m \mapsto \{m\})$
(which represents $\mathbf{N}$) to $(\mathbb{N}, \; k \mapsto
J_g(\{k\}))$ are functions $f\colon \N\to\N$ which are tracked by a
computable function which takes $m \in \mathbb{N}$ to a an Arthur
strategy for declaring $k := f(m)$, and that is just the same as a
winning Arthur strategy for the reduction game of $f$ to $g$.  And the
same reasoning works for the internal Hom-sets, providing the
description of ${\mathbf{N}_g}^{\mathbf{N}_g}$ as an assembly.
\end{proof}

\begin{cor}\label{construction-j-assembly-for-turing-degree}
If $A \subseteq \mathbb{N}$ then the sheafification of $2$ for the
Lawvere–Tierney topology $J_A \colon U \mapsto \{e\in\mathbb{N} \mid
\varphi_e^A(0){\downarrow} \in U\}$ associated with the Turing
oracle $A$ is the object $X_{H_0^A, H_1^A}$ defined
in \ref{definition-x-p-q-object}, where $H_0^A, H_1^A$ are
$J_A(\{0\}), J_A(\{1\})$ as in \ref{definition-codegrees}.
\end{cor}
\begin{proof}
This is just a special case of \cref{construction-j-assembly}.
\end{proof}

\begin{cor}\label{codegree-construction-from-sheafification}
If $A \subseteq \mathbb{N}$, then the degrees $\tilde{\mathbf{d}}_A$
and $\tilde{\mathbf{c}}_A$ can be defined as follows: start with the
Lawvere–Tierney topology $J_A$ associated with the Turing oracle $A$,
consider the inclusions
\[
2 \hookrightarrow \sh_{J_A}(2) \hookrightarrow \nabla 2
\]
where $\sh_{J_A}(2)$ is the sheafification of $2$ for this topology
(and $\nabla 2$ for the $(\neg\neg)$ topology).  Then
$\tilde{\mathbf{d}}_A$ (which turns out to be
just $[A]_{\mathrm{T0}}$) and $\tilde{\mathbf{c}}_A$ are the
degrees associated to the topologies defined
by \cref{smallest-l-t-topology-making-subobject-dense} from the
subobjects $2 \hookrightarrow \sh_{J_A}(2)$ and $\sh_{J_A}(2)
\hookrightarrow \nabla 2$ respectively.
\end{cor}
\begin{proof}
This follows immediately from
\cref{toposic-construction-of-coding-decoding-degrees}
and \cref{construction-j-assembly-for-turing-degree}.
\end{proof}

\thingy What makes the above corollary interesting is that this
describes our “co-Turing degree” construction $\coT\colon
[A]_{\mathrm{T0}} \mapsto \tilde{\mathbf{c}}_A$ as an instance of
a construction that makes sense for any Lawvere–Tierney topology $j
\leq (\neg\neg)$ in any topos: construct the subobject $\sh_j(2)
\hookrightarrow \sh_{\neg\neg}(2)$ and then consider the smallest
topology $k$ (per \cref{smallest-l-t-topology-making-subobject-dense})
that makes this $k$-dense.  But let us first explain why $\sh_j(2)
\hookrightarrow \sh_{\neg\neg}(2)$ is injective in general and how it
can be seen fairly explicitly.

\begin{prop}\label{description-j-sheafification-of-2}
In any topos, we can identify the sheafification $\sh_{\neg\neg}(2)$
of $2$ for the $(\neg\neg)$ topology with the object $\{u \in \Omega \mid
\neg\neg u \Rightarrow u\}$ of $(\neg\neg)$-stable truth values.

Furthermore, once this identification is performed, if $j \leq
(\neg\neg)$ is a Lawvere–Tierney topology on a topos, then the
morphism $\sh_j(2) \to \sh_{\neg\neg}(2)$ (coming from the fact that
since $j \leq (\neg\neg)$, the $(\neg\neg)$-sheaf $\sh_{\neg\neg}(2)$
is \textit{a fortiori} a $j$-sheaf) is the injection of the subobject of
the object $\{u \in \sh_{\neg\neg}(2) \mid j(u\lor\neg u)\}$
consisting of those $(\neg\neg)$-stable truth values $u$ such that
$j(u\lor\neg u)$ holds.
\end{prop}
\begin{proof}
By \cref{construction-j-sheafification}, and using the fact that $j
\leq (\neg\neg)$ implies that $2$ is $j$-separated
(cf. ¶\ref{reminder-on-j-false}), if we identify a subobject
$S\hookrightarrow 2$ with the pair $(p,q)$ of truth values $p :=
\llbracket 0\in S\rrbracket$ and $q := \llbracket 1\in S\rrbracket $
respectively, we can see $\sh_j(2)$ as the object of
pairs $(p,q)\in\Omega^2$ such that (1) $j(p)=p$ and $j(q)=q$,
(2) $\neg(p\land q)$, and (3) $j(p\lor q)$.  Working internally for
such a pair $(p,q)$: if $q$ holds then $\neg p$ holds by (2); but
conversely, if $\neg p$ holds then so does $j(\neg p)$ and by (3) we
have $j(q)$, which by (1) gives us $q$; so we have shown the
equivalence of $q$ with $\neg p$, that is, $q = (\neg p)$.  And of
course $p = (\neg q)$ as everything is symmetrical.  So we can just
keep one of the truth values, rename it to $u$, which is required to
satisfy $\neg\neg u = u$.  The conditions (1) $j(u) = u$ and $j(\neg
u) = \neg u$ are then automatic since $u$ and $\neg u$ are
$(\neg\neg)$-stable so they are $j$-stable since $j \leq (\neg\neg)$;
the condition (2) $\neg(u\land\neg u)$ is also automatic.  So the only
surviving condition is (3) $j(u\lor\neg u)$, and we have performed the
required identification.  (And the injection $\sh_j(2) \to
\sh_{\neg\neg}(2)$ we got is the natural one coming from $j \leq
(\neg\neg)$ because we started from the description of them as objects
of subterminal objects of $2$.)
\end{proof}

To put it more succinctly, $\sh_j(2) = \{u \in \sh_{\neg\neg}(2) \mid
j(u\lor\neg u)\}$ is the $j$-closure of $2 = \{u\in\sh_{\neg\neg}(2) \mid
u\lor\neg u\}$ in $\sh_{\neg\neg}(2) = \{u \in \Omega \mid \neg\neg u
\Rightarrow u\}$ (this actually constitutes a proof of the proposition,
if we grant that last identification).

\begin{dfn}\label{codegree-construction-for-topologies}
If $j \leq (\neg\neg)$ is a Lawvere–Tierney topology on a topos, we
define another topology $\coT(j)$ as the smallest one (as
per \cref{smallest-l-t-topology-making-subobject-dense}) which makes
the inclusion $\sh_j(2) \hookrightarrow \sh_{\neg\neg}(2)$ dense,
where $\sh_j(2)$ is the $j$-sheafification of $2$ that we have just
described.
\end{dfn}

More generally, if $j \leq \ell \leq (\neg\neg)$, we can define $\coT(j\leq \ell)$
as the smallest topology making the inclusion $\sh_j(2)
\hookrightarrow \sh_\ell(2)$ dense.  (We have will little intelligent to
say about this operation, however.)

\thingy Thus, \cref{codegree-construction-from-sheafification}
expresses the fact that this $j \mapsto \coT(j)$ construction for Lawvere–Tierney
topologies on the effective topos is compatible, under the
correspondence \ref{definition-oracle-topology-correspondence}, with
the $\coT$ construction from \cref{reverse-embedding-theorem} (and even
$\coT_1$ and $\coT_2$ from remarks \ref{remark-extension-to-t1-degrees}
and \ref{remark-extension-to-t2-degrees}).

The \emph{easy} implication $(\mathbf{b} \leq \mathbf{a}) \;
\Rightarrow \; (\coT(\mathbf{a}) \leq \coT(\mathbf{b}))$ (seen
in \cref{turing-reduction-implies-codegree-reduction}) is an instance
of the completely general phenomenon that $k \leq j$ implies $\coT(j)
\leq \coT(k)$ with the $\coT$ defined
in \ref{codegree-construction-for-topologies}: indeed, $k \leq j
\leq (\neg\neg)$ gives a decomposition $\sh_k(2) \hookrightarrow
\sh_j(2) \hookrightarrow \sh_{\neg\neg}(2)$, and clearly the smallest
topology making $\sh_k(2) \hookrightarrow \sh_{\neg\neg}(2)$ dense
must also make $\sh_j(2) \hookrightarrow \sh_{\neg\neg}(2)$ dense.
The \emph{hard} implication $(\coT(\mathbf{a}) \leq \coT(\mathbf{b}))
\; \Rightarrow \; (\mathbf{b} \leq \mathbf{a})$ (seen
in \cref{codegree-reduction-implies-turing-reduction}, and whose proof
traces back to \cref{main-strategy-result}) \emph{cannot} be part of a
general phenomenon because of the following counterexamples:

\begin{exm}\label{example-wlem-sheaf-topos}
Let $X$ be a Hausdorff, extremally disconnected but non discrete
topological space, e.g., $\beta\mathbb{N}$ (the Stone–Čech
compactification of a countably infinite discrete space),
and $\Sh(X)$ the topos of sheaves of sets on $X$.  Then the
weak law of excluded middle WLEM ($\forall u.(\neg u\lor\neg\neg
u)$) holds in $\Sh(X)$, because “extremally disconnected” means the
closure $\operatorname{cl}(U)$ of any open set $U$ is open, hence
clopen, so it is $\neg\neg U =
\operatorname{int}(\operatorname{cl}(U))$ and it is also the
set-theoretic complement of $\neg U = \operatorname{int}(X \setminus
U)$.  On the other hand, the law of excluded middle LEM
($\forall u.(u\lor\neg u)$) does not hold
in $\Sh(X)$, because the complement $X\setminus\{x\}$ of a
non-isolated point $x$ is a dense open set that is not $X$.

Now by the description in \cref{description-j-sheafification-of-2}
applied to $j$ being the identity, WLEM is exactly equivalent to the
assertion that $2 \hookrightarrow \sh_{\neg\neg}(2)$ be an
isomorphism.  So in $\Sh(X)$, the two topologies $\id$ and
$(\neg\neg)$ have the same value under $\coT$ (namely $\id$), despite
being distinct.

The same example also shows that
the smallest topology “$\coT(\id \leq j)$” making $2 \hookrightarrow
\sh_j(2)$ dense need not always be $j$ itself, contra what
\cref{decoding-relative-halting-is-turing-degree} tells us about the
specific case of Turing degrees.  (We have already encountered this
phenomenon in the final sentence of \cref{remark-extension-to-t2-degrees}.)

We note that the specific case of the smallest topology “$\coT(\id
\leq (\neg\neg))$” making $2 \hookrightarrow \sh_{\neg\neg}(2)$ dense
(which, as we have just pointed out, is equivalent to WLEM, or
equivalently, the De Morgan laws), has been studied under the name
“De Morgan topology” in \cite{CaramelloDM}.
\end{exm}

\begin{exm}\label{example-codegree-for-error-one-third}
Let $\mathtt{Error}_{1/m}$ denote the basic T3 oracle given by the set
$\{\{1,\ldots,m\}\setminus\{i\} \mid 1\leq i\leq m\}$ (in
$\mathcal{P}(\mathcal{P}(\mathbb{N}))$); for $m=2$ this is our $\Om$,
but let us now consider $m=3$.  In other words, this is an oracle
which, out of $3$ choices, rules out one “bad” choice.  (See
\cite[§4.1]{KiharaLT} for more on this degree and its properties.)  It
is not difficult to see that if $f\colon\N\dasharrow\N$ is a T1 oracle
and $f \redTThree \mathtt{Error}_{1/3}$ then in fact $f$ is
computable.  (Essentially, a reduction strategy of $f$ to
$\mathtt{Error}_{1/3}$ defines a $3$-branching computable tree in
which, at each branching node, $2$ of the $3$ branches lead to a
correct computation, which is always the same as $f$ is assumed
deterministic, and this makes it possible to compute $f$ by majority.)

But this means that if we define $J := J_{\mathtt{Error}_{1/3}}$ to be
the Lawvere–Tierney topology associated to the degree
$\mathtt{Error}_{1/3}$, and if we let $P := J(\{0\})$ and $Q :=
J(\{1\})$, then the function taking the value $0$ on $P$ and $1$
on $Q$ (and undefined elsewhere) is, in fact, computable, that is,
$\mathbf{d}_{P,Q} = \mathbf{0}$, and consequently $\mathbf{c}_{P,Q} =
\Om$ (by \cref{coding-and-decoding-give-omniscience}).

In the language of this section, this means that the sheafification of
$2$ for the topology associated to the $\mathtt{Error}_{1/3}$ degree
is $2$ itself (it is already a sheaf), so, as in the previous example,
the $\coT$ construction gives the same result on this topology as on
the bottom ($\id$) topology.

In the language of remarks \ref{remark-extension-to-t1-degrees}
and \ref{remark-extension-to-t2-degrees}, this example means that if
we extend $\coT$ to the T3 degrees as per our general construction,
i.e., as the map $\coT_3 \colon \mathscr{D}_\TThree \to
\mathscr{D}_\TThree$, given explicitly by $\mathbf{c}_{J(\{0\}),
  J(\{1\})}$, then we have $\coT_3(\mathtt{Error}_{1/3}) = \Om$ even
though $\mathtt{Error}_{1/3} > \mathbf{0}$, and we have lost
injectivity.
\end{exm}

\thingy\label{explicit-formula-for-cotopology} Applying the formula
of \cref{smallest-l-t-topology-making-subobject-dense} to the
description of \cref{description-j-sheafification-of-2}, we find that
for $j \leq (\neg\neg)$, our $\coT(j)$ is given by the explicit
formula
\[
u \; \mapsto \; \forall v:\Omega.\;((((\exists w:\Omega.\; ((\neg\neg w \Rightarrow w) \land (j(w\lor\neg w) \Rightarrow v)))\Rightarrow v)\land (u\Rightarrow v))\Rightarrow v)
\]
While this formula does not seem particularly inspiring, it is
interesting to note that it makes sense in considerable generality
(for a nucleus on a frame), so it would be worth investigating what
this general construction means.

\section{Questions and further discussion}\label{section-questions}

\begin{ques}
Can we define $\coT(\mathbf{a})$ “intrinsically” from the T0 degree
$\mathbf{a}$?  Is it the smallest T3 degree $\mathbf{c}$ such that
$\mathbf{a} \vee \mathbf{c} = \Om$?  (When does the latter exist?)
\end{ques}

While ¶\ref{explicit-formula-for-cotopology} offers an answer of sorts
to the first part of this question, we do not deem it to be
satisfactory in that it is not purely in reference to the order
structure of the degrees themselves.  At any rate, it seems that
$\coT$ has a kind of “dual” flavor compared to the Heyting operation
or negation (for which cf. ¶\ref{digression-heyting-operation}):
whereas the negation would be the largest $\mathbf{d}$ such that
$\mathbf{a} \wedge \mathbf{d} = \mathbf{0}$ (and this exists but is
just $\mathbf{0}$ for $\mathbf{a}$ a Turing degree as we explained at
the end of ¶\ref{digression-heyting-operation}), here we are more
inclined to search along the lines of the smallest $\mathbf{c}$ such
that $\mathbf{a} \vee \mathbf{c} = \Om$ (something which generally
does not exist).  Still, we suspect that the answer to the second
question is negative; however, this does not preclude a different
answer along similar lines.

\begin{ques}\label{question-codegrees-for-t1-degrees}
Does the construction in \cref{remark-extension-to-t1-degrees} of a map
$\coT_1 \colon \mathscr{D}_{\mathrm{T1}}^\op \to
\mathscr{D}_{\mathrm{T3}}$ satisfy the analogue
of \cref{reverse-embedding-theorem}?  In other words, is it injective
and order-reflecting?  What about $\coT_2$ defined
in \cref{remark-extension-to-t2-degrees}?
\end{ques}

(We have pointed out in \cref{example-codegree-for-error-one-third}
that the “obvious” extension $\coT_3$ is \emph{not} injective.)

\begin{ques}\label{question-parallelizable-t1-degrees}
Is there a reasonable notion of “parallelizable” T1 degree which would
make the proof of \cref{main-strategy-result} work, giving the
following conclusion: if $\mathbf{a}$ is parallelizable and
$\mathbf{a} \vee \mathbf{c}_{P,Q} = \Om$ then $\mathbf{d}_{P\times Q,
  \, Q\times P} \leq \mathbf{a}$?
\end{ques}

This is also related to the remark made in \ref{remark-on-t1-meet},
which strongly suggests, for example, that the T1 meet of two
incomparable Turing degrees should \emph{not} be parallelizable.  The
following questions are also possibly related.

\begin{ques}\label{question-computability-of-t1-meet}
Does \cref{computability-of-t1-meet} generalize to two T1 degrees?  In
other words, does $\mathbf{f} \mathrel{\wedge_{\mathrm{T1}}}
\mathbf{g} = \mathbf{0}$ imply $\mathbf{f} = \mathbf{0}$ or
$\mathbf{g} = \mathbf{0}$ when $\mathbf{f},\mathbf{g}$ are two T1
degrees?
\end{ques}

\begin{ques}\label{question-t2-versus-t1-meet}
When do we have $\mathbf{f} \mathrel{\wedge_{\mathrm{T1}}} \mathbf{g}
= \mathbf{f} \mathrel{\wedge_{\mathrm{T2}}} \mathbf{g}$ for T1
degrees?
\end{ques}

(We have seen in \cref{remark-on-t2-versus-t1-meet} that this is not
always the case.)

\begin{ques}
What is the order on the $\coT(\mathbf{a})
\mathbin{\wedge_{\mathrm{T3}}} \mathbf{b}$, where
$\mathbf{a},\mathbf{b}$ range over the Turing degrees?  What about
that on the $\coT(\mathbf{a}) \vee \mathbf{b}$?
\end{ques}

(The above question is probably not too hard.)

\begin{ques}
Is $\mathbf{c}_{P,Q}$ a function of $\mathbf{d}_{P\times Q, \, Q\times
  P}$?
\end{ques}

(We have seen in \cref{remark-that-c-p-q-is-not-a-function-of-d-p-q}
that $\mathbf{c}_{P,Q}$ is not a function of $\mathbf{d}_{P,Q}$.)

\begin{ques}
Does $\mathbf{c}_{P,Q} \geq \mathbf{b}$, for $\mathbf{b}$ a Turing
degree, imply that either $\mathbf{c}_{P,Q} = \Om$ or $\mathbf{b} =
\mathbf{0}$?
\end{ques}

(We have seen in \cref{comparison-of-codegree-with-degree} that this is
true for the $\tilde{\mathbf{c}}_A$.)

\bigskip

We have seen in \cref{comparison-of-codegree-with-degree} how to
compare the co-Turing degrees with the Turing degrees: the following
questions ask whether we can compare the co-Turing degrees with other
important T3 degrees.

\begin{ques}\label{comparison-of-codegree-with-error-degrees}
Does $\coT(\mathbf{a}) \geq \mathtt{Error}_{1/m}$ imply $\mathbf{a} =
\mathbf{0}$?  (We suspect that the answer is positive.)
\end{ques}

Here, $\mathtt{Error}_{1/m}$ denotes the basic T3 degree defined by
the set $\{\{1,\ldots,m\}\setminus\{i\} \mid 1\leq i\leq m\}$ (in
$\mathcal{P}(\mathcal{P}(\mathbb{N}))$); for $m=2$ this is our $\Om$.
In other words, the reduction game here is one where Merlin secretly
chooses an element $1\leq i\leq m$ and Arthur's goal is to
\emph{avoid} that $i$ by getting help from Nimue, whose “no”/“yes”
moves are encrypted by Merlin to elements of $H_0^A$ or $H_1^A$.

(We also do not know whether $\coT(\mathbf{a}) \geq \mathtt{LLPO}$,
where $\mathtt{LLPO}$ was defined in ¶\ref{digression-separating-degree}, and, by
\cite[prop. 4.2]{KiharaLT}, satisfies $\mathtt{LLPO} \leq
\mathtt{Error}_{1/3}$.)

Note that we do have $\coT(\mathbf{a}) \geq
\mathtt{Error}_{1/\mathbb{N}}$ for all $\mathbf{a} \in
\mathscr{D}_{\mathrm{T0}}$, where $\mathtt{Error}_{1/\mathbb{N}}$
denotes the basic T3 degree defined by the set
$\{\mathbb{N}\setminus\{i\} \mid i\in\mathbb{N}\}$; this is a general
fact (\cite[prop. 5.5(ii)]{LeeVanOosten}), but we can see it
very explicitly: Merlin secretly chooses an element $i\in\mathbb{N}$,
then Nimue simply plays $0$ or $1$ according as $i$ is \emph{not} in
$H_0^A$ or $H_1^A$, forcing Merlin to give Arthur an element $\neq i$
which Arthur can then declare.

In the other direction: we have $\coT(\mathbf{a}) \not\leq
\mathtt{Error}_{1/m}$ for any $m\geq 3$
by \cite[prop. 4.17]{LeeVanOosten} (applied to $n=2$).

\begin{prop}\label{comparison-of-codegree-with-pitts-operator}
For all Turing degrees $\mathbf{a} > \mathbf{0}$, the (T3) degrees
$\coT(\mathbf{a})$ and $\mathtt{Cofinite}$ are incomparable.  Here,
$\mathtt{Cofinite}$ denotes the basic T3 degree (first considered
in \cite[example 5.8]{PittsThesis}) defined by the set of cofinite
subsets of $\mathbb{N}$.
\end{prop}

(For context, we recall that
by \cite[theorem 2.2]{vanOostenPittsOperator}, for Turing degrees
$\mathbf{a}$, we have $\mathbf{a} \leq \mathtt{Cofinite}$ iff
$\mathbf{a}$ is hyperarithmetic.)

\begin{proof}[Proof of \cref{comparison-of-codegree-with-pitts-operator}]
By the result we have just recalled (or by a straightforward
verification), we have $\mathtt{Cofinite} \geq \mathbf{0}'$, where
$\mathbf{0}'$ denotes the degree of the halting set, whereas
$\coT(\mathbf{a}) \not\geq \mathbf{0}'$ by
our \cref{comparison-of-codegree-with-degree}(i).  This implies
$\coT(\mathbf{a}) \not\geq \mathtt{Cofinite}$

In the other direction, we have $\coT(\mathbf{a}) \not\leq
\mathtt{Cofinite}$ by the same
result \cite[prop. 4.17]{LeeVanOosten} (applied to $n=2$) that
we already mentioned under
question \ref{comparison-of-codegree-with-error-degrees}.
\end{proof}

\begin{ques}
To what extent can we prove \cref{reverse-embedding-theorem}
“internally” or “synthetically”?
\end{ques}

(We have seen in examples \ref{example-wlem-sheaf-topos} and
\ref{example-codegree-for-error-one-third} — and immediately preceding
remarks — that it cannot follow from completely formal reasons.
However, it might still be obtainable from arguments valid in certain
realizability toposes.)

\begin{ques}
Does the subtopos of the effective topos $\mathbf{Eff}$ defined by
(the Lawvere–Tierney topology corresponding to) the co-Turing degree
$\coT(\mathbf{0}')$ of the halting set degree $\mathbf{0}'$ have
interesting internal properties?
\end{ques}

We single out this example because it seems to be the most interesting
co-Turing degree.  Morphisms $\mathbf{N} \to \mathbf{N}$ in this topos
are given by computable maps $\N \to \N$ (because of
\cref{comparison-of-codegree-with-degree}(i) taken together
with \cref{remark-on-hom-nat-to-nat}), but it seems that we
are just “one Turing jump away” from obtaining all maps $\N \to \N$,
which is intriguing.

\end{document}